\documentclass[12pt]{article}
\usepackage{amssymb}
\usepackage{amsthm}
\usepackage{graphicx}
\usepackage{amsmath}
\usepackage{cite}
\usepackage{tikz}
\usepackage[margin=1.21in]{geometry}
\usetikzlibrary{arrows}

\newtheorem{theorem}{Theorem}[section]
\newtheorem{conjecture}[theorem]{Conjecture}
\newtheorem{proposition}[theorem]{Proposition}
\newtheorem{corollary}[theorem]{Corollary}
\newtheorem{lemma}[theorem]{Lemma}
\newtheorem{claim}[theorem]{Claim}

\theoremstyle{definition}
\newtheorem{definition}[theorem]{Definition}
\newtheorem{remark}[theorem]{Remark}
\newtheorem{example}[theorem]{Example}

\newcommand{\C}{\mathbb{C}}
\newcommand{\Z}{\mathbb{Z}}

\newcommand{\PP}{\mathbb{P}}
\newcommand{\R}{\mathbb{R}}

\def\CP{\mathbb{C}{\rm P}}
\def\tr{{\rm tr}\,}
\def\endproof{{$\Box$}}

\include{epsf}

\begin{document}

\def\CP{\mathbb{C}{\rm P}}
\def\tr{{\rm tr}\,}
\def\endproof{{$\Box$}}

\title{$S^{1}$-invariant symplectic hypersurfaces in dimension $6$ and the Fano condition}
\date{\today}
\author{Nicholas Lindsay\footnote{Supported by the Engineering and Physical Sciences Research Council [EP/L015234/1].  
The EPSRC Centre for Doctoral Training in Geometry and Number Theory (The London School of Geometry and Number Theory), University College London.  PhD student at King's College London.}  and Dmitri Panov \footnote{Supported by a Royal Society University Research Fellowship.}}

\maketitle

\begin{abstract}

We prove that any symplectic Fano $6$-manifold $M$ with a Hamiltonian $S^1$-action is simply connected and satisfies $c_1 c_2(M)=24$. This is done by showing that the fixed submanifold $M_{\min}\subseteq M$ on which the Hamiltonian attains its minimum is diffeomorphic to either a del Pezzo surface, a $2$-sphere or a point. In the case when $\dim(M_{\min})=4$, we use the fact that symplectic Fano $4$-manifolds are symplectomorphic to del Pezzo surfaces. The case when $\dim(M_{\min})=2$ involves a study of $6$-dimensional Hamiltonian $S^1$-manifolds with $M_{\min}$ diffeomorphic to a surface of positive genus. By exploiting an analogy with the algebro-geometric situation we construct in each such $6$-manifold an $S^1$-invariant symplectic hypersurface ${\cal F}(M)$ playing the role of a smooth fibre of a hypothetical Mori fibration over $M_{\min}$. This relies upon applying Seiberg-Witten theory to the resolution of symplectic $4$-orbifolds occurring as the reduced spaces of $M$.  
 
\end{abstract}

\tableofcontents

\section{Introduction}  

In this article, we study compact symplectic manifolds $(M, \omega)$ for which $c_1(M) = [\omega]$ in $H^2(M, \R)$. These are called \emph{symplectic Fano manifolds} by analogy with complex projective Fano manifolds. An important question is to determine to what extent symplectic Fano manifolds differ from complex projective Fanos.

In dimension $4$, a theorem of Ohta and Ono \cite[Theorem 1.3]{OhtaOno} building on the works of Gromov \cite{Gr}, McDuff \cite{Mc1} and Taubes \cite{Tau}  states that every $4$-dimensional symplectic Fano manifold is diffeomorphic to a del Pezzo surface. It is known moreover that symplectic Fano $4$-manifolds admit a compatible complex projective structure (see \cite{S} and the references contained within). The corresponding question in dimension $6$, $8$ and $10$ is open. As was noted in \cite{FP1}, starting from dimension $12$, symplectic twistor spaces of hyperbolic manifolds studied by Reznikov in \cite{Res}  give rise to infinitely many symplectic Fano manifolds which are not K\"ahler.
Whilst the full $6$-dimensional problem seems intractable at the moment, there is a weaker version of this question stated in the form of a conjecture in \cite{FP2}.

\begin{conjecture}\label{S1_Fano_question}
Let $(M,\omega)$ be a $6$-dimensional symplectic Fano manifold with a Hamiltonian $S^1$-action. Then $M$ is diffeomorphic to a complex projective Fano $3$-fold.
\end{conjecture}
Our first result in this paper is a step towards the solution of this conjecture.

\begin{theorem}\label{maintheorem} Let $(M,\omega)$ be a symplectic Fano $6$-manifold with a Hamiltonian $S^1$-action. Then $M_{\min}$ is diffeomorphic to a del Pezzo surface, a 2-sphere or a point. In particular, $M$ is simply connected.
\end{theorem}

This theorem has the following corollary, which we prove in Section \ref{hirsection} using the localisation formula for the Hirzebruch $\chi_{y}$-genus.

\begin{corollary}\label{chern}
Any symplectic Fano $6$-manifold $M$ with a Hamiltonian $S^1$-action satisfies $c_{1}c_{2}(M) = 24$.
\end{corollary}

Apart from Theorem \ref{maintheorem}, there are several results in support of Conjecture \ref{S1_Fano_question}. First is by McDuff \cite{Mc2} and Tolman \cite{To}, stating that there exist exactly four symplectic Fano $6$-manifolds with a Hamiltonian $S^1$-action with $b_2=1$, and all of them are complex projective. Next, in \cite{Cho} a classification of symplectic Fano $6$-manifolds with a semi-free Hamiltonian $S^1$-action is undertaken. It is proven there that for each such manifold there is a complex projective one with the same $S^1$-fixed data set. In a different case, when the action has isolated fixed points and the weights of the action at each isolated point are coprime, it is proven in \cite{GHS} that $b_2$ is bounded, namely $b_2\le 7$.
Finally, Conjecture \ref{S1_Fano_question} is confirmed for the case of symplectic Fanos that are fat $S^2$-bundles over $4$-manifolds (for a precise formulation see \cite[Theorem 1.3]{FP2}). 

{\bf Complex projective Fanos.} Recall that complex projective Fano manifolds are, indeed, simply connected. There exist two paths to prove this statement. One is by Mori via characteristic $p$, proving that Fano manifolds are rationally connected; the other is via existence of metrics of positive Ricci curvature by Yau. Neither of these methods have direct analogues in symplectic geometry. 

In the case of complex dimension $3$, Fano manifolds were classified by Iskoviskikh, Mori and Mukai and there are exactly $105$ families. Note, however, that the classification of Fano $3$-folds with $\mathbb C^*$-actions and Picard rank greater than $1$ was not yet carried out. The case of Picard rank $1$ was settled in \cite{KPS, Pro}.

\subsection{Symplectic fibre and the proof of Theorem \ref{maintheorem}}

In this section we explain the key ideas of the article, state further results and give some ideas of proofs. All these results are needed for Theorem \ref{maintheorem}, and so we structure the discussion in terms of its proof. 

In order to prove that $M$ is simply connected, we study the fixed point set of the $S^{1}$-action, which we denote by $M^{S^1}$. We denote by $M_{\min}$ and $M_{\max}$ the connected components of $M^{S^1}$ on which the Hamiltonian of the $S^1$-action attains its minimum and maximum respectively; we call these submanifolds of $M$ {\it extremal submanifolds}. The extremal submanifolds are connected symplectic submanifolds, furthermore it is known by \cite{Li} that the fundamental group of $M$ is isomorphic to that of $M_{\min}$ and $M_{\max}$. Hence, in the case when $M_{\min}$ is a point, $M$ is simply connected.

In the case when $M_{\min}$ has dimension $4$ the proof of Theorem \ref{maintheorem} is split into two subcases according to dimension of $M_{\max}$.  The subcase when $\dim(M_{\max})\ge 2$ is treated in Section \ref{4dimcasection}. We prove there that the $S^1$-action on $M$ is semi-free and show that $M_{\min}$ can be deformed symplectically to a symplectic Fano. It follows that $M_{\min}$ is diffeomorphic to a del Pezzo surface, in particular it is simply connected. 
Let us mention here a related result of Cho \cite{Cho}. He shows that for a symplectic Fano manifold $X$ with a semi-free Hamiltonian $S^1$-action such that $0$ is a regular value of the Hamiltonian, the reduced space $X_0$ is a symplectic Fano\footnote{See Proposition \ref{relcho} for a canonical choice of a Hamiltonian on a symplectic Fano $S^1$-manifold.}. 

The subcase when $M_{\max}$ is a point is postponed to the final Section \ref{hirsection} since it involves techniques developed in Sections  \ref{sectopology}-\ref{sympspheresec}. Here, the action is not neccesarily semi-free, however in Proposition \ref{listofweights} the weights at $M_{\max}$ are classified into two cases, both of which are exhibited by toric Fano $3$-folds.


The hardest case of Theorem \ref{maintheorem} to treat is when  both submanifolds $M_{\min} $ and $M_{\max}$ are $2$-dimensional. Thus, we are led naturally to study the class of Hamiltonian $S^1$-manifolds of dimension $6$ with $M_{\min} $ and $M_{\max}$  surfaces of positive genus. This study takes up a large portion of the article, culminating in the following theorem.

\begin{theorem}\label{symfibmain} Let  $(M,\omega)$ be a symplectic  $6$-manifold with a Hamiltonian $S^1$-action such that $M_{\min} $ is a surface of genus $g>0$. Then $M$ contains a $4$-dimensional, $S^1$-invariant, symplectic submanifold ${\cal F}(M)$ transversal to $M^{S^1}$ and intersecting $M_{\min}$ in a unique point. This submanifold is unique up to $S^1$-invariant symplectomorphism.
\end{theorem}


The proof of Theorem \ref{maintheorem} relies on Theorem \ref{symfibmain} and so we first elaborate on Theorem \ref{symfibmain}. The submanifold ${\cal F}(M)\subset M$ constructed in this theorem is called the {\it symplectic fibre}. Note, that if $M$ were K\"ahler there would exist a holomorphic fibration $M\to M_{\min}$ and we could choose ${\cal F}(M)$ as a generic fibre of this fibration. For more motivation and the analogy between the K\"ahler and symplectic cases, the reader may consult the discussion in the beginning of Section \ref{sectopology}.

{\bf Existence of symplectic fibre.} Building the symplectic fibre will take us Sections \ref{sectopology}-\ref{sympfibresec}. This construction amounts to producing a smooth (in an appropriate sense) path of symplectic orbi-spheres in the reduced spaces $M_t$ for $t$ varying from $H_{\min}$ to $H_{\max}$. We deduce the existence of such a path from Theorem \ref{twoinoneOrbisphere}. A version of this result for the case of $4$-dimensional symplectic manifolds can be found in \cite[Corollary 2]{LTJ} and \cite[Proposition 3.2]{LiWu}.

\begin{theorem}\label{twoinoneOrbisphere} Let $(M^4,\omega)$ be a $4$-dimensional symplectic orbifold with cyclic stabilizers and with $\pi_1(M^4)\ne 0$. Suppose $M^4$ contains a smooth sub-orbifold sphere that is transversal to the orbifold locus of $M^4$ and satisfies the two properties: $F\cdot F=0$ and $\int_F\omega>0$. Then the following statements hold.

\begin{enumerate}
\item $M^4$ contains a symplectic sub-orbifold sphere $F'$ in the same homology class as $F$, that is transversal to the orbifold locus of $M$.

\item Any two such symplectic sub-orbifold spheres can be isotoped one to another in the class of symplectic sub-orbifolds transversal to the orbifold locus.
\end{enumerate}

\end{theorem}

The proof of Theorem \ref{twoinoneOrbisphere} is given in Section \ref{orbispheresection}, where it is split into Theorem \ref{symplorbispheres} and Theorem \ref{orbisotopy}. In order to produce symplectic orbi-spheres in $M^4$ we construct a desingularisation of its orbifold symplectic structure. This is done in three steps. First, in Theorem \ref{semilocal4orb}  we introduce a K\"ahler structure on a neighbourhood of the orbifold locus $\Sigma_{\cal O}(M^4)$ of $M^4$. Next, using Theorem \ref{smoothingdisivors}, we smooth the orbifold K\"ahler metric along all divisors in $\Sigma_{\cal O}(M^4)$ to obtain in this way an orbifold $\widetilde M^4$ with isolated quotient singularities. Finally, we take a holomorphic resolution of singularities of $\widetilde M^4$ and obtain a  smooth symplectic manifold $\overline M^4$. 

We are aware of two alternative methods to desingularise symplectic $4$-orbifolds, the first by Niederkr\"uger and Pasquotto \cite{NP1} using symplectic cutting, and the second, very recent one, by Chen \cite{Chen}. The second approach is closer in spirit to the one that we've  chosen, however our approach is designed specifically for proving results like Theorems \ref{symfibmain} and \ref{twoinoneOrbisphere}.
We believe our approach will find some further applications.

The smooth symplectic manifold $\overline M^4$ is an irrational ruled symplectic manifold due to the existence of a smooth orbi-sphere in $M^4$ and to \cite[Corollary 2]{LTJ}. Once this is established, we are able to apply results of Weiyi Zhang \cite{Zh} to $\overline M^4$, in particular Theorem \ref{zhang}. This theorem states that for any tamed $J$ on an irrational ruled surface, the surface is fibred by $J$-holomorphic spheres with a finite number of singular fibres. We choose $\overline J$ on $\overline M^4$ by extending the holomorphic structure defined on $\overline M^4$ close to the preimage of $\Sigma_{\cal O}M$. By projecting the obtained fibration back to $M^4$, we get in $M^4$ the desired orbi-spheres. This proves the first half of Theorem \ref{twoinoneOrbisphere}. The second half relies on some standard facts on deformations of $J$-holomorphic curves.

Once Theorem \ref{twoinoneOrbisphere} is proven, it is relatively straightforward to construct a ``smooth path" of symplectic orbi-spheres in reduced spaces of $M$ in order to cut the symplectic fibre out inside $M$. Indeed, for regular values of the Hamiltonian $H$ the reduced spaces are orbifolds with cyclic isotropy groups. To help orbi-spheres pass smoothly through critical levels of $H$ we introduce in Section \ref{locKahlsec} a K\"ahler metric on a neighbourhood of $M^{S^1}$. The proof of Theorem \ref{symfibmain} is given in Section \ref{sympfibresec}.

{\bf Back to the proof of Theorem \ref{maintheorem}.} To settle Theorem \ref{maintheorem} in the case when $\dim(M_{\min})=\dim(M_{\max})=2$, we will work with the class of {\it relative symplectic Fano} manifolds\footnote{Such manifolds are usually called {\it weakly monotone}, but we would like to use the alternative terminology, to draw a parallel with algebraic geometry and to exclude non-compact manifolds.}.

\begin{definition} Let $(M,\omega)$ be a compact symplectic manifold. $M$  is called  \emph{relative symplectic Fano} if for any class $A\in H_2(M)$ that can be represented by a continuous mapping from $S^2$ one has $\langle c_1(M),A \rangle = \omega(A)$. 
\end{definition}

It is not hard to see that in the case  $\dim(M_{\min})=\dim(M_{\max})=2$, Theorem \ref{maintheorem} follows from the next result.

\begin{theorem} \label{theo}
Let  $(M,\omega)$ be a relative symplectic Fano $6$-manifold with a Hamiltonian $S^1$-action such that $M_{\min} $ and $ M_{\max} $ are surfaces of genus $g>0$. Then there exists a fixed surface $\Sigma$ with $g(\Sigma) \geq g$ such that $\langle c_{1}(M),\Sigma \rangle \leq 2-2g$.
\end{theorem} 

We prove additionally in Corollary \ref{iso} that unless the non-zero weights of $S^1$-action at $M_{\max}$ and  $M_{\min}$ are $\pm 1$, fixed surfaces in $M$ have genus either $0$ or $g$. In such case, the theorem states that $M$ contains a fixed surface of genus $g$ whose normal bundle has non-positive $c_1$. 

The proof of Theorem \ref{theo} will take us Sections \ref{fixsurfsec}-\ref{sectionproftheo} and uses the existence of the symplectic fibre. However, in the case when the weights of the $S^1$-action on $M$ at all fixed surfaces are different from $\pm 1$ this theorem can be proven in 10 lines. We have the following general statement, which doesn't rely on the symplectic Fano condition.

\begin{lemma}\label{nonunitweights}Let  $(M,\omega)$ be a  symplectic  $6$-manifold with a Hamiltonian $S^1$-action such that $M_{\min} $ and $ M_{\max} $ are surfaces of genus $g>0$. Suppose that the weights of the $S^1$-action at all fixed surfaces are different from $\pm 1$. Then there exists a fixed surface $\Sigma$ with $g(\Sigma)=g$ such that the normal bundle to $\Sigma$ in $M$ has $c_1\le 0$.
\end{lemma}
\begin{proof} Note first, that each fixed surface $\Sigma$ in $M$ is contained in two isotropy submanifolds of $M$ of dimension $4$ (since the weights at $\Sigma$ are not $\pm 1$). If $g(\Sigma)>0$, then by Corollary \ref{genus} both isotropy submanifolds containing $\Sigma$ contain a different fixed surface of genus $g(\Sigma)$. It follows that one can find in $M$ a cycle $\Sigma_{1}, \Sigma_{2}...,\Sigma_{n} = \Sigma_{1}$ of fixed surfaces, such that any two consecutive $\Sigma_{i},\Sigma_{i+1}$ are contained in a 4-dimensional isotropy submanifold $N_{i}$. Let $N(\Sigma_{i})$ denote the normal bundle of $\Sigma_{i}$ in $M$. Applying Lemma \ref{fourcor} to each $N_{i}$ we deduce that $$ \sum_{i=1}^{n-1} c_{1}(N(\Sigma_{i})) \leq 0.$$ Hence there exists a fixed surface of genus $g$ in $M$ with non-positive normal bundle.
\end{proof}

{\bf Proving Theorem \ref{theo} in general.} The simple proof of Lemma \ref{nonunitweights} breaks down if some of the weights of the $S^1$-action at some fixed surface are equal to $\pm 1$. This is because some isotropy submanifolds from the cycle constructed in Lemma \ref{nonunitweights} are missing. However, we are able to resurrect them to some extent by finding constraints on $M^{S^1}$ and investigating the behaviour of gradient spheres in $M$. An important tool in this analysis is the following {\it weight sum formula}. For symplectic Fano manifolds this proposition appeared previously in \cite[Example 4.3]{CK} for semi-free $S^1$-actions and in \cite[Corollary 4.32]{Cho} for general $S^1$-actions.

\begin{proposition}\label{relcho} \textbf{The weight sum formula}. Let $(M,\omega)$ be a relative symplectic Fano manifold and suppose there is a Hamiltonian $S^{1}$-action on $M$. Then after adding a constant to the Hamiltonian $H$, for any fixed submanifold $F$, the following weight sum formula holds
\begin{equation}\label{weightcho}
H(F)=-w_{1}-\ldots-w_{n} 
\end{equation} 
where $\{w_{i}\}$ is the set of weights for the $S^{1}$-action along $F$. 
\end{proposition}

Proposition \ref{relcho} and its converse Proposition \ref{converse} are proven in Section \ref{hamnornsec}. 
From now on we assume that the Hamiltonian $H$ on $M$ is normalized according to Proposition \ref{relcho}, so that weight sum formula (\ref{weightcho}) holds. Proposition \ref{converse}  tells us then that the symplectic fibre ${\cal F}(M)$ of $M$ is a $4$-dimensional symplectic Fano. So we can apply  \cite[Theorem 5.1]{Ka}  to deduce that ${\cal F}(M)$ is symplectomorphic to a toric del Pezzo surface. In particular, since all fixed surfaces of positive genus intersect ${\cal F}(M)$, the number of such fixed surfaces in $M$ is at most $6$.

We split the proof of Theorem \ref{theo} in the following two cases.

{\bf The general case.} In the case when $H(M)\not\subseteq [-3,3]$, we are able to use the idea of Lemma \ref{nonunitweights} and find in $M$ a genus $g$ surface, whose normal bundle has $c_1\le 0$ (see Theorem \ref{aim}).   
It turns out in this case, that all fixed surfaces in $M$ with positive genus  have genus exactly $g$ and intersect ${\cal F}(M)$ in a unique point. Since  ${\cal F}(M)$ is toric, its fixed points have a natural cyclic order, and so, fixed genus $g$ surfaces in $M$ inherit such an order as well. Since some of the weights at fixed genus $g$ surfaces are $\pm 1$,  the isotropy submanifolds connecting these surfaces don't form a complete cycle, some of the submanifolds are missing. Thus, we need to find a replacement for Lemma \ref{fourcor} for missing submanifolds. This is the content of Theorem \ref{isotropy inequality} which relies on Seiberg-Witten theory.

To establish Theorem \ref{isotropy inequality} we first prove restrictions on the fixed submanifolds contained in levels close to $H_{\min}$. This allows us to flow a fixed surface $\Sigma$ with weights $\{-1,n\}$ down towards $M_{\min}$, eventually proving 	Theorem \ref{isotropy inequality} by studying the Euler numbers of the associated orbi-bundles. As a result we obtain a proof of Theorem \ref{theo} along the lines of Lemma \ref{nonunitweights}.

{\bf Small Hamiltonian case.} The case when $H(M)\subseteq [-3,3]$ is treated in Section \ref{smallhamsec}. In this case some of fixed curves of positive genus in $M$ can be of genus greater than $g$ and intersect ${\cal F}(M)$ in two points. Hence, they don't necessarily form a cycle. For this reason, instead of pushing the idea of Lemma \ref{nonunitweights} we prove Theorem \ref{theo} using localisation of $c_1^{S^1}(M)$. This is done in  Section \ref{smallhamsec}.

{\bf Outline of the paper.} We now briefly sum up the plan of the paper. Section \ref{prelimsec} contains definitions and preliminary results used throughout the article. Section \ref{4dimcasection} gives a proof of Theorem \ref{maintheorem}
in the case when $\dim(M_{\min})=4$, $\dim(M_{\max})\ge 2$.

Starting from Section \ref{sectopology} we primarily study Hamiltonian $S^1$-manifolds of dimension $6$ with $M_{\min}$ and $M_{\max}$ surfaces of positive genus $g$. Sections \ref{sectopology}-\ref{sympfibresec} are devoted to the construction of the symplectic fibre. 

In Section \ref{fixsurfsec} we prove Theorem \ref{struc}, which relates the fixed points of the symplectic fibre to the fixed surfaces of positive genus in $M$. Here, we establish the uniqueness of symplectic fibre up to  equivariant symplectomorphism.

In Section \ref{toricdelPezzo} we collect some simple facts about $S^1$-actions on toric del Pezzo surfaces. Section \ref{sectionproftheo} is devoted to the proof of Theorem \ref{theo}. 

In the final Section \ref{hirsection} we finish the proof of Theorem \ref{maintheorem}. All the cases of this theorem apart from the case when $\dim(M_{\min})=4$, $\dim(M_{\max})=0$ are proven by now, so we settle the remaining case. Additionally to this we prove Corollary \ref{chern}. 

The appendix contains the majority of results related to K\"ahler metrics, such as existence and smoothing. 

{\bf Acknowledgements.} We are very grateful to Weiyi Zhang for patiently answering our numerous questions on $J$-holomorphic curves and explaining the results of his paper \cite{Zh}. We would also like to thank Yunhyung Cho, Joel Fine, Chris Gerig, Tian-Jun Li, Heather Macbeth,  Simon Salamon, Song Sun, Valentino Tosatti and Chris Wendl for their help.

\section{Preliminaries} \label{prelimsec}

In this section we set up notation,  and state  various results that are used throughout the article. None of these results are new, with the exception of Proposition \ref{relcho} and Proposition \ref{converse}.

\subsection{Hamiltonian circle actions}\label{prelimperlim}
\textbf{Hamiltonian circle actions}. Let $(M,\omega)$ be a symplectic manifold. A Hamiltonian $S^{1}$-action on $(M,\omega)$ is an action of the group $S^{1}=U(1)$ on $M$ generated by a vector field $X$, with an associated smooth function $H: M \rightarrow \R$ such that $dH = \iota_{X} \omega$. The $S^{1}$-action is always assumed to be non-trivial and effective, unless stated otherwise.  

We often will refer to a manifold with a Hamiltonian $S^{1}$-action as a \textit{Hamiltonian $S^{1}$-manifold}. An $S^{1}$-action is called semi-free if the stabiliser of each point is either $S^{1}$ or the identity.

\textbf{Compatible almost complex structures.} Recall that an almost complex structure $J$ on $(M,\omega)$ is said to be compatible with $\omega$, if $ \omega(J \cdot, \cdot)$ defines a Riemannian metric on $M$. In this article we often assume that a compatible $S^1$-invariant $J$ is chosen on $M$ and don't always say this explicitly.

\textbf{Fixed submanifolds}. Define $M^{S^{1}} \subset M$ to be the subset of points fixed by all elements of $S^{1}$. Note that $M^{S^{1}}$ is the subset where $X=0$ and also the subset where $dH = 0$, i.e. the set of critical points for $H$. Let $\Z_n$ be the subgroup of $S^{1} = U(1)$ generated by $e^{\frac{2 \pi i}{n}}$. Define $M^{\Z_n}$ as the subset of points in $M$ that are fixed by $\Z_n$.

\begin{lemma} \cite[Lemma 5.53-5.54]{MS}\label{manysimpleproperties}
Let $(M,\omega)$ be a compact symplectic manifold with a Hamiltonian $S^{1}$-action generated by the Hamiltonian $H:M \rightarrow \R$. Then the following holds.

\begin{itemize}
\item There exists an almost complex structure $J$ compatible with $\omega$ and preserved by the $S^{1}$-action.

\item $H$ is a Morse-Bott function.

\item For any positive integer $n$,  $M^{\Z_n}$ is a union of symplectic submanifolds (of possibly different dimensions).

\item  $M^{S^{1}}$ is a union of symplectic submanifolds (of possibly different dimensions).

\item Let $H_{\min},H_{\max}$ be the minimum/maximum values for $H$. The corresponding level sets $M_{\min}$ and $M_{\max}$ are connected symplectic submanifolds in $M$. 

\end{itemize} 
\end{lemma}

\begin{definition} By a \emph{fixed surface} in $M$ we always refer to a $2$-dimensional component of $M^{S^{1}}$. \end{definition}

\textbf{Weights of the action along a fixed submanifold.} Fix an invariant and compatible almost complex structure $J$ on $M$. Let $F$ be a connected component of $M^{S^{1}}$, which is a symplectic submanifold by Lemma \ref{manysimpleproperties}. For each $p \in F$, $S^1$ acts linearly on $T_{p}M$ and we can split $T_{p}M$ into irreducible representations of the form $V_{k} = \C$ $(k \in \Z)$ where $S^{1}$ acts by $z.w = z^{k}w$. It is not hard to see that the type of splitting is independent of  $p\in F$ and we call the numbers $k$  the {\it weights} at $F$. In the case when $F$ has positive dimension some of its weights are equal to zero and we will omit them from time to time.



\textbf{The gradient flow and gradient spheres.} Fix again an invariant $J$ on $M$. Denote by $g $ the Riemannian metric $ \omega(J \cdot,\cdot)$. The  integral curves of the gradient vector field $\nabla_{g}H$ are called {\it gradient flow lines}. 

\begin{definition}
Let $O$ be a non-trivial orbit of the $S^{1}$-action, and let $\tilde{O}$ be the union of all gradient flow lines intersecting $O$. Then the closure $S$ of $\tilde{O}$ is homeomorphic to a $2$-sphere, and call $S$ a \emph{gradient sphere}. The {\it weight} of $S$ is defined to be the order of the stabilizer of $O$. Denote by $S_{\min},S_{\max} \in M$  the fixed points on $S$ where $H|_{S}$ attains its minimum and maximum respectively.
\end{definition}
A gradient sphere $S$ may not be smooth at $S_{\min}$ or $S_{\max}$. However, there always exists an $S^1$-equivariant homeomorphism $\varphi: \mathbb S^2\to S$, such that $\varphi$ is a diffeomorphism outside of the two fixed points and $\varphi^*(\omega)$ extends to a symplectic form on $\mathbb S^2$.

\textbf{Isotropy submanifolds}.
An \emph{isotropy submanifold} $N\subset M$ is a connected component of $M^{\Z_{n}}$ for some $n \geq 2$ that is not contained entirely in $M^{S^{1}}$.

\begin{itemize}
\item  We define the \emph{weight} of $N$ to be the largest $n$ such that $N \subseteq M^{\Z_{n}}$.

\item Define $N_{\min},N_{\max}$ to be the subsets where $H|_{N}$ attains its minimum and maximum respectively.

\end{itemize}
{\bf Trace of $S^1$-invariant submanifolds.} 
\begin{definition} Let $N\subset M$ be an $S^1$-invariant submanifold of $M$. For  $c \in [H_{\min}, H_{\max}]$ define the \emph{trace} $N_{c}$ to be the image of $N \cap H^{-1}(c)$ under the quotient map $Q: H^{-1}(c) \rightarrow M_{c} = H^{-1}(c)/S^{1}$.  We will say that $N_c$ is {\it traced} by $N$ in $M_c$.
\end{definition}

\textbf{Reduced spaces and the associated orbi-$S^{1}$-bundle}. Let $c$ be a regular value of $H$, the quotient space 
\begin{displaymath}
M_{c} = H^{-1}(c)/S^{1}
\end{displaymath}
inherits a natural symplectic orbifold structure from $M$ \cite[Section 5.4]{MS}. Denote the symplectic orbifold form by $\omega_{c}$. $M_{c}$ is called the \textit{reduced space} or \textit{symplectic quotient} at $c$.

Following \cite[Chapter 5, Section 2.4]{A}, let us recall  how to define the Euler class \begin{displaymath}
e(H^{-1}(c)) \in H^{2}(M_{c},\Z).
\end{displaymath}  Note first that all orbits in $H^{-1}(c)$ have finite cyclic stabilizer groups, and since $H^{-1}(c)$ is compact the order of such stabilizer groups is bounded. Choose $N>0$ divisible by all the orders of such stabilizers. Then $H^{-1}(c)/\Z_{N}$ is a principal $S^{1}$-bundle. Define \begin{displaymath}
e(H^{-1}(c)) = \frac{1}{N}e(H^{-1}(c)/\Z_{N})
\end{displaymath}
where on the right hand side $e$ denotes the usual Euler class of a principal $S^{1}$-bundle.

\textbf{The gradient map}. Let $(M,\omega,g)$ be a Hamiltonian $S^1$-manifold with a compatible $S^1$-invariant metric $g$. Let $c_1,c_2$ be two values of $H$  in $(H_{\min}, H_{\max})$.
 
\begin{definition}\label{gradmapdef} The partially defined {\it gradient map} $gr_{c_2}^{c_1}:M_{c_1}\dashrightarrow M_{c_2}$ is constructed as follows. For $x_1\in X_{c_1}$ we set $gr_{c_2}^{c_1}(x_1)=x_2$ in case when there is a gradient flow line in $M$ whose intersections with $H^{-1}(c_1)$  and $H^{-1}(c_2)$ projects to $x_1$ and $x_2$ respectively. 
\end{definition}

\begin{remark}\label{gradmapremark} Assume that $c_{1} < c_{2}$. In the case when a critical value of $H$ lies in  the interval $[c_1,c_2]$, the map $gr_{c_2}^{c_1}$ is not defined on a closed subset of $M_{c_1}$ of positive codimension. However, if no such critical values are contained in $[c_1,c_2]$, the map is a diffeomorphism of orbifolds. In the case when the only  critical value in $[c_1,c_2]$ is $c_2$, the partially defined gradient map can be extended to a continuous map   $gr^{c_{1}}_{c_{2}} : M_{c_1} \rightarrow M_{c_{2}}$.
\end{remark}

\textbf{The Duistermaat-Heckman theorem}. Suppose $(a,b)$ is an interval of regular values for $H$. Fixing $c_{0} \in (a,b)$, we consider the smooth family of symplectic forms $(gr^{c_{0}}_{c})^{*}(\omega_{c})$ on $M_{c_{0}}$ (which we refer to also as $\omega_{c}$ for brevity).

\begin{lemma} \label{DH} Duistermaat-Heckman Theorem. The cohomology class of the symplectic form $[\omega_{c}] \in H^{2}(M_{c_{0}},\R)$ varies linearly as a function of $c$ and the derivative of this path is $-e(H^{-1}(c_{0}))$. 
\end{lemma}

\begin{lemma}\label{DuisHeckSphere} Let $M$ be $4$-dimensional Hamiltonian $S^1$-manifold such that $M_{\min}$ is an isolated fixed point. Let $a,b>0$ be the weights of the action at $M_{\min}$ and let $m=H_{\min}$.

\begin{enumerate} 
\item for $\varepsilon > 0 $ sufficiently small \begin{displaymath}
\langle e(H^{-1}(m + \varepsilon)),[M_{m + \varepsilon}] \rangle = \frac{-1}{ab}.
\end{displaymath}

\item Suppose furthermore that $M$ is a compact manifold with isolated fixed points. For any $c>0$, $\omega_{c}([M_{m+c}])\le \frac{c}{a\cdot b}$. Moreover the equality holds if and only if  all values in $(m, m+c)$ are regular values for $H$. 
\end{enumerate}
\end{lemma}
\begin{proof}
For 1 see \cite[Exercise 3.4.5]{A}.

To prove 2, first note that by \cite[Theorem 5.1]{Ka} $(M,\omega)$ is toric and the Hamiltonian is obtained by composing the toric moment map $\mu: M \rightarrow \R^{2}$ with a linear map $L: \R^{2} \rightarrow \R$. By Lemma \ref{DH}, for $c$ sufficiently close to $m$, we have:$$\frac{d \omega_{c}([M_{m+c}])}{dc} = \frac{1}{ab}.$$ By the convexity of the toric moment polytope of $M$ \cite[Theorem 5.47]{MS} this derivative is decreasing in terms of $c$, and changes whenever the corresponding level contains a vertex of the moment polytope. Hence 2 follows. 
\end{proof}

\textbf{Duistermaat-Heckman function.} 

\begin{definition} Let $(M,\omega)$ be a symplectic $2n$-manifold with a Hamiltonian $S^1$-action and let $H$ be the Hamiltonian. The function 
$$DH(t)=\int_{M_t}\frac{1}{(n-1)!}\omega_t^{n-1}$$
is called the \emph{Duistermaat-Heckmann function}.
\end{definition}

The Duistermaat-Heckman function is piecewise polynomial, and it is polynomial of degree at most $n-1$ on each interval in the complement to the set of critical values of $H$. The following formula due to Guillemin, Lerman, and Sternberg \cite{GLS} describes the behaviour of $DH$ near  critical values of $H$, see as well \cite[Theorem 3.2]{Cho2}.

\begin{theorem}\cite{GLS}
Assume that $c$ is a critical value of $H$ and $F_1,\ldots, F_k$ are connected components of $M^{S^1}$ in the level set $H=c$. Then the jump of $DH(t)$ at $c$ is given by
\begin{equation}\label{DuisHecMeasure}
DH_+-DH_-=\sum_{i=1}^k\frac{{\rm vol}(F_i)}{(d_i-1)!\Pi_j \omega_j(F_i)}(t-c)^{d_i-1}+O((t-c)^{d_i}),
\end{equation}
where $d_i$ is half the  codimension of $F_i$ in $M$, and  $\omega_j(F_i)$ are the weights of the normal bundle to $F_i$.
\end{theorem}

Note that for zero-dimensional $F_i$ (i.e., isolated fixed points) we have ${\rm vol}(F_i)=1$, and moreover the second term in the right hand side of (\ref{DuisHecMeasure}) is missing, since $DH(t)$ is piecewise of degree $n-1$. As a result we have the following corollary.

\begin{corollary}\label{isoDuisCor} Suppose that all fixed points in $M$ that lie in the set $H>s$ are isolated, and denote by $p_1,\ldots, p_k$ these points. Then we have
\begin{equation}\label{isoDuis}
\int_{M_s} \omega_s^{n-1}=-\sum_{i=1}^k\frac{(s-H(p_i))^{n-1}}{\Pi_j \omega_j(p_i)}.
\end{equation} 
\end{corollary}

\textbf{$\pi_1$ of Hamiltonian $S^1$-manifolds}. Here we collect some results that are important for studying the fundamental group of Hamiltonian $S^{1}$-manifolds.

\begin{theorem} \cite{Li} \label{Li}
Let $(M,\omega)$ be a compact symplectic manifold with a Hamiltonian $S^{1}$-action and Hamiltonian $H:M \rightarrow \R$. Then the inclusion of $M_{\min}$ and $M_{\max}$ into $M$ induces an isomorphism on $\pi_{1}$. In addition, $\pi_{1}(M_{x}) \cong \pi_{1}(M)$ for all $x \in H(M)$. 
\end{theorem}

\begin{corollary}\label{genus}
Suppose that $(M,\omega)$ is a symplectic $4$-manifold with a Hamiltonian $S^{1}$-action. If $M$ contains a fixed surface of genus $g>0$, then both of the extremal fixed submanifolds of $M$ are diffeomorphic to surfaces of genus $g$.
\end{corollary}

\begin{theorem}\cite{Li2} \label{Li2}
Let $(M,\omega)$ be a compact symplectic manifold with a Hamiltonian $S^{1}$-action. Then the quotient map $Q:M \rightarrow M/S^{1}$ induces an isomorphism on fundamental groups.
\end{theorem}
\begin{proof}
Since $M$ is compact, the Hamiltonian has some critical points. Hence $M^{S^{1}}$ is non-empty, 
and $Q_{*} : \pi_{1}(M) \rightarrow \pi_{1}(M/S^{1})$ is an isomorphism by the main result of \cite{Li2}.
\end{proof}
\subsection{Localisation results}
The following lemmas follow from  applying the Atiyah-Bott-Berline-Vergne localisation formula to certain equivariant Chern classes. All of these results are known, we collect them here for reference.

\begin{lemma} \label{wsphere}
Suppose that $S^{1}$ acts on $S \cong \C\PP^1$ with the action $z[z_{0}:z_{1}] = [z^k z_{0} : z_{1}]$. Consider a rank n, $S^1$-equivariant, complex vector bundle $E \rightarrow S$. Suppose that the weights of the action (on the fibres of $E$) at $[1:0]$, $[0:1]$ are  $\{a_{1},a_{2},\ldots,a_{n}\}$, $\{b_{1},b_{2},\ldots,b_{n}\}$ respectively. Then \begin{displaymath}c_{1}(E) = \frac{-a_{1}-\ldots-a_{n}+b_{1}+\ldots+b_{n}}{k}. \end{displaymath}
\end{lemma}
\begin{proof}
This formula follows by applying the Atiyah-Bott-Berline-Vergne localisation theorem to $c_{1}^{S^{1}}(E)$.
\end{proof}
\begin{lemma} \cite[Remark 2.5]{To} \label{bigloc}
Let $S^1$ act on $6$-dimensional symplectic manifold $(M,\omega)$ by a Hamiltonian $S^{1}$-action and suppose that the fixed submanifolds have dimension at most $2$.

For a fixed point p with weights $w_{1},w_{2},w_{3}$,  define

\begin{displaymath}
\alpha(p) = \frac{w_{1} + w_{2} + w_{3}}{w_{1}w_{2}w_{3}}.
\end{displaymath}

For a fixed surface $S$ with genus g and normal bundle $L_{1} \oplus L_{2}$ let the weight of the action of $L_{i}$ be denoted $w_{i}$ and let $n_{i} = c_{1}(L_{i})$ define
 \begin{displaymath}
\beta(S) = \frac{2-2g}{w_{1} w_{2}} - \frac{n_{1}}{w_{1}^2} - \frac{n_{2}}{w_{2}^2}.
\end{displaymath}

Let $p_{1},\ldots,p_{N_{1}}$ be the isolated fixed points for the action and $S_{1}, \ldots , S_{N_{2}}$ be the surface components of $M^{S^{1}}$. Then we have the following localisation formula.

\item[] \begin{equation} \label{bigloceq}
0  = \sum_{1 \leq i \leq N_{1}} \alpha(p_{i})+\sum_{1 \leq i \leq N_{2}} \beta(S_{i}).
\end{equation}

\end{lemma}
\begin{proof}
This follows from applying Atiyah-Bott-Berline-Vergne to $c_{1}^{S^{1}}(M)$. The localisation of $c_{1}^{S^{1}}(M)$ to fixed points and surfaces is computed in  \cite[Remark 2.5]{To}.
\end{proof}

The following Lemma will be useful for studying isotropy $4$-manifolds for $S^{1}$-actions on $6$-manifolds.

\begin{lemma}\cite[Lemma 2.15]{Ka}\label{fourcor}
Suppose that $(M,\omega)$ is a symplectic $4$-manifold with a Hamiltonian $S^{1}$-action. For each fixed surface $\Sigma$, let $n(\Sigma_{i})$ denote the degree of the normal bundle of $\Sigma_{i}$, and $\mathcal{S}$ be the set of all fixed surfaces. Let the weights at each fixed point $p_{i}$ be denoted $(a_{i},b_{i})$, and let $\mathcal{P}$ denote the set of isolated fixed points. Then
\begin{displaymath}
\sum_{p_{i} \in \mathcal{P}} \frac{1}{a_{i}b_{i}} - \sum_{\Sigma_{i} \in \mathcal{S}}  n_i(\Sigma_{i}) = 0.
\end{displaymath}

In particular, if the extremal submanifolds are both surfaces and $n_{1},n_{2}$ denote the degree of the normal bundle of $M_{\min}$ and $M_{\max}$ then $n_{1} + n_{2} \leq 0$.
\end{lemma}

 
 \subsection{Restrictions on weights in terms of the range of the Hamiltonian}
 
In this section we prove some restrictions on the weights of a Hamiltonian $S^{1}$-action along fixed submanifolds in terms of the range of the Hamiltonian. 

\begin{lemma}\label{gradsphereint} Let $(M,\omega)$ be a symplectic manifold with a Hamiltonian $S^1$-action, with Hamiltonian $H$. Let $S$ be a gradient sphere in $M$. Then \begin{displaymath}
\langle [\omega], S\rangle=\frac{H(S_{\max})-H(S_{\min})}{w(S)}.
\end{displaymath}
\end{lemma}
\begin{proof} If $S$ is smooth at points $S_{\max}$ and $S_{\min}$, then the claim is standard. Suppose the sphere is not smooth and let $\varepsilon$ be a small positive number. Denote by $C_{\varepsilon}$ the sub-cylinder of $S$ consisting of points on distance at least $\varepsilon$ to $S_{\max}$ and $S_{\min}$ (for some $S^1$-invariant metric). Then we have 
$$\int_{C_{\varepsilon}}\omega=\frac{1}{w(S)}(\max H|_{C_{\varepsilon}}-\min H|_{C_{\varepsilon}}),$$
since $C_{\varepsilon}$ is a symplectic cylinder with a Hamiltonian $S^1$-action. It is now not hard to see that by taking limit $\varepsilon\to 0$ we get the desired equality.
\end{proof}

\begin{lemma} \label{missinglemma}
Let $(M,\omega)$ be a compact symplectic manifold such that $\omega$ is an integral form, and there is a Hamiltonian $S^{1}$-action on $M$. Suppose that $p \in M^{S^{1}}$ and $-w$   is a weight at $p$ with $w>0$, then $w \leq H(p) - H_{\min}$. If $w>0$ is a weight at $p$, then $w \leq H_{\max} - H(p)$.
\end{lemma}

\begin{proof}
There is a gradient sphere $S$ with weight $w$ such that $S_{\max} = p$. We apply Lemma \ref{gradsphereint} to $S$, obtaining that $ H(S_{\min}) \leq H(p) -w $. The Lemma follows since $H_{\min} \leq H(S_{\min})$. For positive weights the argument is similar.
\end{proof}
 
The following lemma is standard so we omit its proof.

\begin{lemma} \label{resweight}
Let $(M,\omega)$ be a compact symplectic manifold with a Hamiltonian $S^{1}$-action on $M$. Let $S$ be a gradient sphere in $M$ with weight $w$ such that $S_{\min} \in M_{\min}$.  \begin{enumerate}
\item If $\emph{codim}(M_{\min}) =2$, then $w = 1$.
\item If $\emph{codim}(M_{\min}) = 4$ and the weights at $M_{\min}$ are $\{1,m\}$. Then $w$ is equal to $1$ or $m$.
\end{enumerate} 
\end{lemma}

\subsection{Proof of Proposition \ref{relcho}}\label{hamnornsec}

In this section we prove the weight sum formula stated in Proposition \ref{relcho}. For symplectic Fano manifolds this formula was given by Cho and Kim \cite[Example 4.3]{CK} in the case when $S^1$-action is semi-free and in \cite{Cho} for all actions.


\begin{proof}[Proof of Proposition \ref{relcho}] By adding a constant to the Hamiltonian, we may assume that Equality  (\ref{weightcho}) holds for $M_{\min}$. We will deduce then that Equality (\ref{weightcho}) holds for all points in $M^{S^1}$. The main idea here is the following: any connected component of $M^{S^{1}}$ can be connected to $M_{\min}$ by a  chain of gradient spheres. For this reason it is enough to prove that for each gradient sphere $S\subset M$ the difference of values of $H$ at its two fixed points is equal to the difference of the sums of weights.

Let $S$ be a gradient sphere in $M$. Denote by $w_{i}$ the weights at $S_{\max}$ and by $w_{i}'$ the weights at $S_{\min}$. By finding an equivariant parametrisation $p : S^{2} \rightarrow S$ and applying Lemma \ref{wsphere} to the pull-back bundle $p^{*} TM$ on $S^2$, we have that \begin{displaymath}
\langle c_{1}(M), S \rangle = \frac{\sum_{i = 1}^{n} w_{i} - \sum_{i = 1}^{n} w'_{i}}{w(S)}.
\end{displaymath}

On the other hand by Lemma \ref{gradsphereint} we have that \begin{displaymath}
\langle [\omega], S \rangle = \frac{H(S_{\max}) - H(S_{\min})}{w(S)}.
\end{displaymath}

Since $S$ is represented by a sphere, the relative symplectic Fano condition gives that $\langle [\omega], S \rangle =  \langle c_{1}(M), S \rangle $. Hence,
\begin{displaymath}
H(S_{\max}) - H(S_{\min}) = \sum_{i = 1}^{n} w_{i} - \sum_{i = 1}^{n} w'_{i}.
\end{displaymath}

This finishes the proof of the Proposition. \end{proof}

\subsubsection{The converse}






Here we give a partial converse to Proposition \ref{relcho}.

\begin{proposition}\label{converse} Let $(M,\omega)$ be a symplectic manifold with a Hamiltonian $S^1$-action and Hamiltonian $H$. Suppose that for any fixed submanifold $F$ of index two or zero, $H(F)=-w_{1}-\ldots-w_{n} $ where $\{w_{i}\}$ is the set of weights for the action along $F$. Suppose moreover that restrictions of $\omega$ and $c_1(M)$ to $M_{\min}$ coincide in cohomology. Then $M$ is a symplectic Fano manifold. 
\end{proposition}
The proof relies on the following standard lemma which we state without a proof.
\begin{lemma}\label{sublevelH2} Let $(M,\omega)$ be a compact Hamiltonian $S^1$-manifold with a compatible $S^1$-invariant metric.  For any value $t\in [H_{\min},\,H_{\max}]$ set $M_{\le t}=H^{-1}[H_{\min},t]$. Then the second homology group $H_2(M_{\le t},\mathbb R)$ is generated by $H_2(M_{\min},\mathbb{R})$ and the classes of all gradient spheres in $M$ that are contained in $M_{\le t}$.  
\end{lemma}
\begin{proof}[Proof of Proposition \ref{converse}] To prove the proposition we will show that the cohomology classes  of $\omega$ and $c_1(M)$ have the same restriction to $M_{\le t}$ for all $t\in [H_{\min}, H_{\max}]$:
\begin{equation}\label{equalityM_t}
c_1(M)|_{M_{\le t}}=[\omega]|_{M_{\le t}}\in H^2(M_{\le t},\mathbb R).
\end{equation}
We will do so by increasing $t$ from $H_{\min}$ to $H_{\max}$. Equality (\ref{equalityM_t})  holds for $t=H_{\min}$ by our assumptions. To prove it for a larger $t$, by Lemma  \ref{sublevelH2} we need to show that for any gradient sphere $S$ entirely contained in $M_{\le t}$ we have $\langle c_{1}(M), S \rangle=\langle [\omega], S \rangle$. 
Clearly, we only need to show that (\ref{equalityM_t}) is preserved when $t$ passes a critical value of $H$.

Let $c$ be a critical value of $H$ and let $\varepsilon>0$ be such that all values in $[c-\varepsilon,c+\varepsilon]$ apart from $c$ are regular. By induction, suppose that Equality (\ref{equalityM_t})  holds on $M_{\le c-\varepsilon}$ and let us prove that it holds on $M_{\le c+\varepsilon}$. By Lemma \ref{sublevelH2} we need to show that for any gradient sphere $S$ emanating from a fixed point $q$ on the level $H=c$ we have $\langle c_{1}(M), S \rangle=\langle [\omega], S \rangle$.

Observe first that the weight sum formula (\ref{weightcho}) holds on $M_{\le c-\varepsilon}$. Indeed, Equality (\ref{equalityM_t}) holds on $M_{\le c-\varepsilon}$, and by our assumption  (\ref{weightcho}) holds at $M_{\min}$. So, to prove this claim we can repeat the argument from the proof of Proposition \ref{relcho}. 

Consider now the case when $q=S_{\max}$ has index $2$. By finding an equivariant parametrisation $p : S^{2} \rightarrow S$ and applying Lemma \ref{wsphere} to $p^{*} TM$, we have that \begin{displaymath}
\langle c_{1}(M), S \rangle = \frac{\sum_{i = 1}^{n} w_{i} - \sum_{i = 1}^{n} w'_{i}}{w(S)}.
\end{displaymath}

By Lemma \ref{gradsphereint} we have that \begin{displaymath}
\langle [\omega], S \rangle = \frac{H(S_{\max}) - H(S_{\min})}{w(S)}.
\end{displaymath}

Since the point $S_{\max}$ is of index $2$ we have $H(S_{\max})=-\sum_{i = 1}^{n} w_{i}$. Since $S_{\min}$ belongs to $M_{\le c-\varepsilon}$ we have $H(S_{\min})=-\sum_{i = 1}^{n} w'_{i}$. Hence, we have $
\langle [\omega], S \rangle =  \langle c_{1}(M), S \rangle $ as required.

Suppose now that $q$ has index greater than $2$. In this case from standard Morse theory considerations it follows that $S$ is homotopic to a two-sphere that lies entirely in $M_{\le c-\varepsilon}$, and so $\langle c_{1}(M), S \rangle=\langle [\omega], S \rangle$ since (\ref{equalityM_t}) holds on $M_{\le c-\varepsilon}$.  \end{proof}

%
%
%
%
%
%

\subsection{Symplectic spheres and irrational ruled surfaces}

Here we collect results concerning symplectic spheres and irrational ruled surfaces.

\begin{definition} Let $(M^4,\omega)$ be a symplectic manifold. A class $e\in H_2(M^4,\mathbb Z)$ is called \emph{exceptional} if it can be represented by a smooth $2$-sphere and satisfies $e \cdot e =-1$, $ K_{M^4}\cdot e =1$ and $\int_e \omega>0$. A smooth sphere representing $e$ is called an \emph{exceptional sphere}.
\end{definition}

The following theorem is contained in \cite{Tau} and \cite{LiLiu}.

\begin{theorem}\label{TaubesLiLui} Let $(M^4,\omega)$ be a symplectic manifold and $e\in H_2(M^4,\mathbb Z)$ be an exceptional class. Then $e$ can be represented as a symplectically embedded $S^2$.  Moreover, for any almost complex structure $J$ tamed by $\omega$ there is an almost complex curve in $M^4$ realising $e$. 
\end{theorem}

The following lemma gives a source of exceptional classes in an iterated blow-up of a complex surface.

\begin{lemma} \label{meromorphic}
Let $\psi: S \rightarrow \C^2$ be an iterated blow-up of $\C^2$ at $(0,0)$. Let $\{E_{j} \}$ denote the irreducible components of $\psi^{-1}(0,0)$. Then for each $E_{i}$ there is an exceptional class $\alpha = \sum n_{j} [E_{j}] \in H_{2}(S,\Z)$ with $n_{j} \geq 0$ for each $j$ and $n_{i} = 1$.
\end{lemma}  
\begin{proof}
The map $\psi$ can be decomposed as a sequence of simple blow-ups $\psi = \psi_{n} \circ \ldots \circ \psi_{1}$. After re-enumerating divisors $E_j$ if necessary we may assume that $E_1$ is contracted by $\psi_1$, $E_2$ by $\psi_2\circ\psi_1$, and so on.
Let us denote by $E_i'$ the divisor $\psi_{i-1}\circ\ldots\circ\psi_1(E_i)$; this is the exceptional divisor of the simple blow up $\psi_i$.

Let us isotope $E'_{i}$ to a smooth sphere that avoids points of indeterminacy of the map  $(\psi_{i-1}\circ\ldots\circ\psi_1)^{-1}$. It is clear then that the homology class represented by the preimage of this sphere under $(\psi_{i-1}\circ\ldots\circ\psi_1)^{-1}$ satisfies the required properties.
\end{proof}

The next result follows from  \cite{LiZh}.
\begin{theorem} \label{inttheo}
Suppose that $e$ is an exceptional class in a symplectic $4$-manifold. Let $\Sigma$ be an embedded symplectic surface with positive genus, then $e \cdot \Sigma \geq 0$. 
\end{theorem}
\begin{proof}
Let $E$ be a symplectic sphere representing $e$. By \cite[Claim 3.8]{LiZh} we can find a compatible $J$ so that both $E$ and $\Sigma$ are $J$-holomorphic. Since $E$ and $\Sigma$ are distinct, irreducible $J$-holomorphic curves they intersect non-negatively.
\end{proof}

\begin{definition} A symplectic manifold $(M^4,\omega)$ is called \emph{rational} if it is a symplectic blow-up in a finite number of points of $\mathbb CP^2$ or of a symplectic $S^2$-bundle over $S^2$.

A symplectic manifold $(M^4,\omega)$ is called an \emph{irrational ruled surface } if it is a symplectic blow-up in a finite number of points  of a symplectic $S^2$-bundle over a positive genus surface.

\end{definition}

The following result of Tian-Jun Li \cite[Corollary 2]{LTJ} gives a characterisation of rational surfaces and irrational ruled surfaces.

\begin{theorem}\label{ruledcriterion} Let $M$ be a symplectic $4$-manifold. If $M$ contains a smoothly embedded sphere $S$ with non-negative self-intersection and of infinite order in $H^2(M,\mathbb Z)$. Then $M$ is rational or ruled, and $M$ contains a symplectically embedded sphere with non-negative self-intersection. 

In particular, if $\pi_1(M)\ne 0$, $M$ is irrational ruled and contains a symplectically embedded $S^2$ with self-intersection zero. On the other hand, if $S^2>0$, $M$ is rational.
\end{theorem}

\begin{definition} The \emph{symplectic fibre} of an irrational ruled surface $M^4$ is any symplectic $S^2\subset M^4$ that has zero self-intersection. The homology class of the symplectic fibre is denoted by $\cal F$.  A \emph{ symplectic section} of $M^4$ is any symplectic surface $\Sigma\subset M^4$ such that $[\Sigma]\cdot  {\cal F}=\pm 1$. 
\end{definition}

\begin{remark}  The symplectic fibre $F$ is unique up to symplectic isotopy by \cite[Proposition 3.2]{LiWu}. It follows further from Theorem \ref{zhang} that  $[\Sigma]\cdot  {\cal F}= 1$ for any symplectic section $\Sigma$.
\end{remark}

The strongest result on irrational ruled surfaces that we need is the following theorem of Zhang  \cite{Zh}.

\begin{theorem}\label{zhang} Let $M$ be an irrational ruled surface of base genus $g$. Then for any tamed $J$ on $M$, the following holds.
\begin{enumerate}

\item There is a unique $J$-holomorphic subvariety in the symplectic fibre class $\cal F$ passing
through a given point.

\item The moduli space ${\cal M}_{\cal F}$ of subvarieties in class $\cal F$ is homeomorphic to a genus $g$ surface. The number of reducible subvarieties in class $\cal F$ is at most $b_2(M)-2$.

\item Every irreducible rational curve in $M$ is an irreducible component of a subvariety
in class $\cal F$.

\item The complement to the set of points in ${\cal M}_{\cal F}$ corresponding to reducible subvarieties has a structure of a smooth surface and the natural map $f: M\to {\cal M}_{\cal F}$ is a continuous map, smooth over this complement.

\end{enumerate}

\begin{remark} Note that all irreducible $J$-holomorphic curves in the class $\cal F$ are  smooth spheres in $M$.
\end{remark}

\end{theorem}
\begin{proof} The first three statements of this theorem is \cite[Theorem 1.2]{Zh}, the only addition is the bound on the number of reducible fibres. This bound holds, since each reducible fibre contains a sphere with negative self-intersection.

The continuity of the map $M\to {\cal M}_{\cal F}$ is \cite[Corollary 3.9]{Zh}. The smoothness of the moduli space of irreducible fibres and the smoothness of the map to this space is explained in discussions after the proof of \cite[Corollary 3.9]{Zh}. \end{proof}

We will need three corollaries of this theorem, two immediate and one which is a bit more technical.
\begin{corollary}\label{fibreintersection} Any symplectic surface $\Sigma$ of positive genus in an irrational ruled surface $M$ intersects positively the fibre class of the surface.
\end{corollary}
\begin{proof} Choose a tamed $J$ on $M$ such that $\Sigma$ is almost complex. Let $p\in \Sigma$ be a point and consider the almost complex subvariety $V$ in class ${\cal F}$ passing through $p$.  $\Sigma$ is not a connected component of $V$ by Theorem \ref{zhang} 3). Hence $V\cdot \Sigma>0$. \end{proof}

\begin{corollary}\label{sphereareabound}  Let $(M,\omega)$ be an irrational ruled surface  and let ${\cal F}$ be the fibre class. If $\Sigma\subset M$ is a symplectic sphere with $[\Sigma]\ne {\cal F}$ then $\int_{\cal F}\omega>\int_{\Sigma}\omega$. 
\end{corollary}

\begin{proof} Consider any tamed $J$ on $M$ for which $\Sigma$ is almost complex. Then by  Theorem \ref{zhang} 3) surface $\Sigma$ is an irreducible component of an almost-complex subvariety representing the class ${\cal F}$. This proves the corollary since $[\Sigma]\ne {\cal F}$. \end{proof}



\begin{corollary}\label{finitangence} Suppose we are in the setting of Theorem \ref{zhang} and let $f: M\to {\cal M}_{\cal F}$ be the natural map to the moduli space of fibres. Let $\Sigma\subset M$ be a smooth almost complex surface with $g(\Sigma)>0$. Then the induced map $f:\Sigma\to {\cal M}_{\cal F}$ is a topological ramified cover of degree $d=\Sigma\cdot {\cal F}$. Moreover, the number of irreducible fibres tangent to $\Sigma$ is at most $d(2-2g)-\chi(\Sigma)$. 
\end{corollary}

\begin{proof} Since $\Sigma$ is almost complex, it intersects all fibres (in particular reducible ones) in at most $d$ points. Hence, by Theorem \ref{zhang} 4) on the complement to the finite set $Y\subset \Sigma$ of points where $\Sigma$ intersects reducible fibres, the map $f:\Sigma\to {\cal M}_{\cal F}$ is smooth. Note as well that $f$ is orientation preserving at points where $\Sigma$ is transversal to fibres. It follows from the proof of Lemma \ref{foliationlemma} that close to points of  $(\Sigma\setminus Y)$ where $\Sigma$ is tangent to fibres the map $f$ is a ramified cover. Hence, to prove the statement it is sufficient to establish the following Lemma \ref{ramificationlemma}. \end{proof}

\begin{lemma}\label{ramificationlemma} Let $\Sigma_1$ and $\Sigma_2$ be two compact, smooth and oriented $2$-dimensional surfaces and let $\{y_1,\ldots,y_k\}=Y$ be a finite collection of points in $\Sigma_1$. Suppose that $\varphi: \Sigma_1\to \Sigma_2$ is a continuous map of degree $d>0$, whose restriction to  $\Sigma_1\setminus Y$ is a smooth ramified cover. Then the map $\varphi$ is a (topological) ramified cover and the number of branching points is at most  $d\cdot \chi(\Sigma_2)-\chi(\Sigma_1)$.
\end{lemma}
\begin{proof}
Let us show first that the number of branching points in $\Sigma_1 \setminus Y$ is finite and in fact at most $2d-\chi(\Sigma_1)$. Indeed, let $X=\{x_1,\ldots, x_{e}\}\subset \Sigma_1\setminus Y$ be a subset of the set of ramification points of $\varphi$. Perturbing slightly $\varphi$ we can assure that the sets $\varphi(X)$ and $\varphi(Y)$ are disjoint. Let us chose a disk $D\subset \Sigma_2$ that contains $\varphi(X)$ and is disjoint from $\varphi(Y)$. Then $\varphi$ induces a ramified cover from $\varphi^{-1}(D)$ to $D$ and so by the Riemann-Hurwitz formula $\chi(\varphi^{-1}(D))\le d-e$. At the same time, since $\varphi^{-1}(D)$ is a subsurface of $\Sigma_1$ with at most $d$ boundary components we have $\chi(\varphi^{-1}(D))\ge \chi(\Sigma_1)-d$. This gives the desired bound on the number of  branching points in $\Sigma\setminus Y$. Since this number is finite, it is not hard to see that $\varphi$ is a topological ramified cover. Now, the bound on the total number of ramifications follows again from the Riemann-Hurwitz formula. \end{proof}

\subsection{Orbifolds and reduced spaces of Hamiltonian $S^1$-manifolds}

Orbifolds play a crucial role in this article, we collect here some necessary definitions (omitting the most standard details) and then discuss reduced spaces. 

For a point $p$ in an orbifold $M$ the {\it isotropy group} or {\it stabilizer} of $p$ is denoted by $\Gamma_p$. The {\it orbifold locus} of an orbifold $M$ is the union of all points of $M$ with non-trivial stabilizers, it will be denoted by $\Sigma_{\cal O}(M)$. The orbifold locus $\Sigma_{\cal O}(M)$ has a natural stratification, where a $k$-dimensional stratum consists of points $p$ in $M$ such that the action of $\Gamma_p$ of $T_pM$ fixes a plane of co-dimension $k$. All the points of the zero-dimensional stratum of $\Sigma_{\cal O}(M)$ will be called {\it maximal} points.

For each orbifold $M$ one can speak about {\it sub-orbifolds}, i.e., subsets $N\subset M$ such that for each point $x\in N$ the preimage of $N$ in a local orbi-chart of $M$ is a submanifold.
For each sub-orbifold $N$ in an orbifold $M$ the {\it orbifold normal vector bundle} to $N$ in $M$ is defined\footnote{Here, slightly informally, an {\it orbifold vector bundle} of rank $k$ over an $n$-dimensional orbifold is an orbifold that admits an open cover by orbifolds $(B^n\times \mathbb R^k)/\Gamma$ where $\Gamma$ is a finite group that sends $\mathbb R^k$-fibres to $\mathbb R^k$-fibres and preserves the linear structure on them.}. 

We will be especially interested in symplectic $4$-dimensional orbifolds with cyclic isotropy groups. In this case the complement to the zero-dimensional stratum in $\Sigma_{\cal O}(M)$ is the union of $2$-dimensional surfaces. We will need the following definition later on.

\begin{definition} Let $M^4$ be a smooth $4$-dimensional orbifold and let $M^2$ be a $2$-dimensional sub-orbifold. We will say that $M^2$ is \emph{transversal to the orbifold locus} if in local orbi-charts $M^2$ is transversal to the set of points with non-trivial stabilizers and doesn't pass through  maximal points in $\Sigma_{\cal O}(M^4)$.
\end{definition}

It is well known that for a symplectic manifold $(M,\omega)$ with a Hamiltonian $S^{1}$-action the reduced spaces corresponding to regular values of the Hamiltonian inherit a natural symplectic orbifold structure \cite{We}. This statement has the following refinement when the dimension of $M$ is at most $6$.

\begin{lemma}\label{toporbifold} In the case when the dimension of $M$ is (at most) $6$ the reduced spaces have the structure of topological orbifolds for all values of $H$. 
\end{lemma}
\begin{proof} Indeed at critical values different from $H_{\min}$ and $H_{\max}$ the symplectic quotient can be locally identified with a GIT quotient of $\mathbb C^3$ by some linear $\mathbb C^*$-action, which is always a complex orbifold of dimension $2$. \end{proof}

This lemma does not hold in dimension $8$ and higher, as the example of linear Hamiltonian $S^1$-action with weights $(1,1,-1,-1)$ on $\mathbb C^4$ shows.

We now focus on the case when $\dim(M) = 6$. Note that since $M_{c}$ is a topological orbifold for all $c \in (H_{\min}, H_{\max} )$, that $H^{2}(M_{c})$ has a well-defined intersection form. 

\begin{theorem} \label{Mcsingature}
Suppose $M$ is a $6$-dimensional closed symplectic manifold with a Hamiltonian $S^{1}$-action, such that $\dim(M_{\min}) = 2$. Then for any $c \in (H_{\min},H_{\max})$ $$b^{+}(M_{c}) = 1 .$$
\end{theorem}

\begin{proof}
From the wall-crossing formulas for signature and Betti numbers given in \cite{Me}, we see that $b^{+}(M_{c})$ does not vary for different values of $c$. For $c$ close enough to $H_{\min}$, $M_{c}$ is homeomorphic to an $S^{2}$-bundle over $M_{\min}$, hence satisfies $b^{+}(M_{c}) =1$.
\end{proof}

The next lemma is standard and follows from Chevalley-Shephard-Todd theorem. 
\begin{lemma}\label{complexquotient} The underlying topological space of a $2$-dimensional complex orbifold has a  structure of a complex analytic surface with (isolated) quotient singularities.
\end{lemma} 


{\bf Pre-symplectic submanifolds and orbifolds.} The following definition and theorem explain the relation between a Hamiltonian $S^1$-manifold and symplectic sub-orbifolds of its reduced spaces.

\begin{definition}\label{presymdef}
A pre-symplectic $S^{1}$-manifold is a $2k+1$-dimensional smooth manifold with a fixed-point free $S^1$-action and a $2$-form $\omega$ that has constant rank $k$ and is preserved by the $S^1$-action. 
\end{definition}

\begin{lemma}\label{presymth}
Let $(M^{2n},\omega)$ be a symplectic manifold with a Hamiltonian $S^1$-action. Let $c$ be a regular value of $H$ and let $\pi: H^{-1}(c) \rightarrow M_{c}$ be the projection. The preimage under $\pi$ of any symplectic sub-orbifold of $M_{c}$ is a pre-symplectic submanifold of $H^{-1}(c)$. 
\end{lemma}

\begin{proof}
This follows immediately from the slice theorem \cite[Theorem 2.1.1]{A}.
\end{proof}

\subsection{K\"ahler reduction}

In this section we collect some facts about the complex structure on reduced spaces for $S^1$-actions on K\"ahler manifolds. 

\begin{theorem}\label{analyticstr} Let $(X,g)$ be a complex projective manifold with an isometric $S^1$-action with Hamiltonian $H$. Then the following statements hold

\begin{enumerate}
\item For any regular value $c$ of $H$ the reduced space $X_c=H^{-1}(c)/S^1$ is a K\"ahler orbifold with respect to the quotient metric.

\item For a general $c\in (H_{\min}, H_{\max})$  let $U_c\subset X$ be the Zarisky open subset of $X$ consisting of  points whose $\mathbb C^*$ orbit closures intersect $H^{-1}(c)$.
Then the  space $X_c$ can be equipped with a complex analytic structure by identifying it with the categorical quotient $U_c//\mathbb C^*$. 

\item In the case where $X$ has complex dimension $3$, the reduced space $X_c$ has a complex analytic structure of a complex surface with isolated quotient singularities for any $c\in (H_{\min},H_{\max})$.

\end{enumerate}
\end{theorem}

These facts are quite well known,  we will give a brief explanation. 

\begin{proof}

Statement 1) is classical. Statement 2) can be found, for example, in \cite{Hu} and \cite{Fu1}, we will recall the main idea. Note, that there is a natural continuous map $\phi_c:U_c\to X_c$ which sends each $\mathbb C^*$-orbit in $U_c$ to a point in $X_c$ corresponding to the intersection of the closure of the orbit with $H^{-1}(c)$ (this intersection is either a circle or a point from $X^{S^1}$). The map $\phi_c$ identifies certain $\mathbb C^*$-orbits in $U_c$, namely the orbits whose closure  contains the same point in $X^{S^1}\cup H^{-1}(c)$. On the topological level this is exactly the identification that one has to do to get the categorical quotient $U_c//\mathbb C^*$ from the usual quotient $U_c/\mathbb C^*$. Hence, the map  $\phi_c$ induces a homeomorphism from $U_c//\mathbb C^*$ to $X_c$.

As for statement 3), in case $c$ is a regular value of $H$ the reduced space $X_c$ has a structure of complex $2$-dimensional orbifold, and so we can use Lemma \ref{complexquotient}. In the case $c$ is a critical value,  $X_c$ is a topological orbifold  (see Lemma \ref{toporbifold}), and its is well known that such a complex analytic surface has a structure of a complex surface with isolated quotient singularities. \end{proof}

\begin{theorem}\label{bimeromorphic} Let $(X,g)$ be a complex projective manifold with an isometric $S^1$-action with Hamiltonian $H$. Let $c_1<c_2$ be two values of $H$  in $(H_{\min}, H_{\max})$  and let $gr_{c_2}^{c_1}:X_{c_1}\dashrightarrow X_{c_2}$ be the partially defined gradient map. Then the following statements hold

\begin{enumerate}

\item The map $gr_{c_2}^{c_1}$ is a bi-meromorphic with respect to the analytic  structures on $X_{c_1}$ and $X_{c_2}$ from Theorem \ref{analyticstr}. The map is an isomorphism of complex analytic spaces if the interval $[c_1,c_2]$ does not contain critical values of $H$.

\item  Suppose that $c_2$ is the only critical value of $H$ in $[c_1, c_2]$. Then the map $gr_{c_2}^{c_1}$ can be extended to a regular map from $X_{c_1}$ to $X_{c_2}$. Moreover for any complex submanifold $Y\subset M^{S^1}\cap H^{-1}(c_2)$ with $\dim_{\mathbb C} Y=\dim X_{\mathbb C}-2$ the map $gr_{c_2}^{c_1}$ is invertible on the pre-image of a neighbourhood of $Y$ .

\end{enumerate}

\end{theorem}

\begin{proof} 1) Consider the Zariski open subset $U_{c_1c_2}$ of $X$, consisting of all $\mathbb C^*$-orbits that intersect both level sets $H=c_1$ and $H=c_2$. Then $U_{c_1c_2}/\mathbb C^*$ is embedded as an open subset in both $U_{c_1}//\mathbb C^*$ and $U_{c_2}//\mathbb C^*$ and consequently, by Theorem \ref{analyticstr}, in $X_{c_1}$ and $X_{c_2}$. This gives us an identification of two open subsets of $X_{c_1}$ and $X_{c_2}$. Since the geodesics in $X$ are $\mathbb R^*$-orbits, the partially defined map $gr_{c_2}^{c_1}: X_{c_1}\dashrightarrow X_{c_2}$ gives us the same identification.

2)  According to Remark \ref{gradmapremark} the bi-meromorphic map $gr_{c_2}^{c_1}: X_{c_1}\dashrightarrow X_{c_2}$ can be extended to a continuous map, i.e.,  this map extends to a regular map. Since $X_{c_2}$ is smooth along $Y$ and the extended map is one-to one close to $Y$, its inverse exists in a neighbourhood of $Y$. \end{proof}

\begin{lemma} \label{Kahlerblowup}
Let $(M,\omega)$ be a compact symplectic manifold with a Hamiltonian $S^1$-action. Suppose that $N$ is a fixed submanifold of complex codimension $2$, such that there is an invariant, compatible K\"ahler metric on a neighbourhood $U$ of $N$, that restricts to a K\"ahler form on $N$. Consider the equivariant blow up \begin{displaymath}
\pi : Bl_{N}(M) \rightarrow M
\end{displaymath}
constructed in Lemma \ref{smoothblowup}, with invariant symplectic form $\tilde{\omega} = \pi^{*}(\omega) + \omega_{E}$ (which is K\"ahler on $\pi^{-1}(U)$). Let $S(\omega_{E}) \subseteq Bl_{N}(M)$ be the support of $\omega_{E}$. Then the following statements hold.
\begin{enumerate} 

\item The $S^{1}$-action on $Bl_{N}(M)$ is Hamiltonian. The Hamiltonian $\tilde{H} : Bl_{N}(M) \rightarrow \R$ may be chosen so that $ \tilde{H} = \pi^*H$ on $Bl_{N}(M)\setminus S(\omega_{E})$.

\item Let $c = H(N)$. Then the $S^{1}$-action on $\tilde{H}^{-1}(c) \cap \pi^{-1}(U)$ is fixed point free. Set $M_{c}' = \tilde{H}^{-1}(c)/S^{1} $. Then $\pi$ induces a map $\pi_c : M'_{c} \rightarrow M_{c}$, which is a biholomorphism over $U_{c}$. 
\end{enumerate}

\end{lemma}

\begin{proof}
1. Let $\tilde{X}$ be the vector field generating the  $S^{1}$-action on $Bl_{N}(M)$. Note that $\iota_{\tilde{X}} \tilde{\omega}$ is exact, indeed any $1$-cycle $\alpha$ in $Bl_{N}(M)$  may be homotoped to $Bl_{N}(M)\setminus S(\omega_{E})$, on which $\pi$ is an equivariant symplectomorphism. Therefore $$ \int_{\alpha} \iota_{\tilde{X}} \tilde{\omega} = 0$$
and we conclude that $[\iota_{\tilde{X}} \tilde{\omega}] = 0$ in $H^{1}(Bl_{N}(M),\R)$.
Let $\tilde{H}$ be  a function on $Bl_{N}(M)$ such that $d \tilde{H} = \iota_{\tilde{X}} \tilde{\omega}$. Since $\pi$ is an equivariant symplectomorphism on $Bl_{N}(M)\setminus S(\omega_{E})$, $d \tilde{H} = d(\pi^{*} H)$ there. Hence, $(-\tilde{H} +  \pi^{*}H)|_{Bl_{N}(M) \setminus S(\omega_{E})}$ is constant, and we may remove this discrepancy by adding a constant to $\tilde{H}$.

2. Let $E$ be the exceptional divisor of $\pi$ and $E_{\min}$,$E_{\max}$ the submanifolds where $\tilde{H}$ acquires its $\min/\max$. Let us first show 
$$\tilde{H}(E_{\min}) < c < \tilde{H}(E_{\max}),$$ 
which will imply that there are no fixed points in $\tilde{H}^{-1}(c) \cap \pi^{-1}(U)$. Indeed, before we blow up $N$, for each point $p$ of $N$ there are two gradient spheres $S_1$ and $S_2$ in $N$ containing $p$. The preimage of $S_1\cup S_2$ in the blow up is a union of three spheres two of which $S_1'$ and $S_2'$ project bijectively to $S_1$ and $S_2$. The area of $S_i'$ is less than the area of $S_i$, so applying Lemma \ref{gradsphereint} to these spheres  together with  claim 1. we obtain the inequality.


On $Bl_{N}(M) \setminus S(\omega_{E})$, $\tilde{H} = \pi^*H $ so here $\pi$ maps $\tilde{H}^{-1}(c)$ to $H^{-1}(c)$, inducing a map on reduced spaces which we denote $\pi_c$.  Let us show next that $\pi_c$ can be extended to all of $M_{c}'$, restricting to a biholomorphism over $U_{c}$.

Let $F: U \rightarrow U_{c}$ be the  map that associates to each $p\in U$ the trace in $U_c$ of the gradient sphere containing $p$. Note that $F$ is holomorphic by construction. 


Over $U$, $\pi$ is holomorphic and $S^{1}$-equivariant, so $\pi$ maps gradient spheres in $Bl_{N}(M)$ to gradient spheres or fixed points in $M$. Hence, where it is defined, the composition $F \circ \pi$ descends to a map on $M'_{c}$ which we set to be $\pi_{c}$. Since $F\circ \pi$ is holomorphic, the induced map on the reduced space is also holomorphic. The fact that $\pi_c$ is a biholomorphism over $U_{c}$ follows since it is a holomorphic bijection. 
\end{proof}

\subsection{Equivariant Darboux-Weinstein theorem}

\begin{definition} A symplectic $G$-orbifold is a symplectic orbifold with a Hamiltonian action of a compact group $G$. We say that a symplectomorphism $\varphi: X\to Y$ of two symplectic $G$-orbifolds is a $G$-symplectomorphism if it commutes with the $G$-actions.
\end{definition}

\begin{theorem}[Darboux-Weinstein]
\label{relative_local_neighbourhood}
Let $(M, \omega)$ and $(M', \omega')$ be two symplectic $G$-orbifolds and let $N\subset M$ and $N' \subset M'$ be symplectic sub-orbifolds with normal orbifold bundles $E$ and $E'$ respectively. 
\begin{enumerate}
\item
Let $D \colon E \to E'$ be a $G$-isomorphism of symplectic orbifold bundles covering a $G$-symplectomorphism $\phi_0 \colon N \to N'$. Then there exist neighbourhoods $W$ and $W'$ of $N$ and $N'$ and a $G$-symplectomorphism $\phi \colon W\to W'$ extending $\phi_0$ such that $D\phi = D$ on $E$.
\item
Suppose in addition that there are points $p_1, \ldots , p_m \in N^G$ and locally defined $G$-symplectomorphisms $\psi_1, \ldots , \psi_m$ with $\psi_i$ defined near $p_i$, taking values in $M'$ and with $\psi_i|_N = \phi_0$. Suppose, moreover, that on the normal bundle of $N$, we have $D = D\psi_j$. Then we can arrange that for each $i$, there is a neighbourhood of $p_i$ (possibly smaller than the domain of $\psi_i$) on which $\phi = \psi_i$.

\end{enumerate}
\end{theorem}
\begin{proof} The proof (and the statement) of the standard equivariant Darboux-Weinstein theorem for symplectic manifolds can be found in \cite[Theorem 22.1]{GS1}. The extension of this proof to orbifolds and to  the relative statement 2) is straightforward. \end{proof}

The following lemma is quite standard so we omit a proof. 

\begin{lemma}\label{sympextension} Let $M$ be a symplectic manifold, $x\in M$ be a point, and $U(x)$ a neighbourhood of $x$. Then for any symplectic map $\varphi: U(x)\to M$ with $\varphi(x)=x$ there exists a smaller neighbourhood $U'(x)\subset U(x)$ and a symplectic automorphism $\varphi'$ of $M$ that satisfies the following:

1) $\varphi'=\varphi$ on $U'(x)$. 2) The restriction of $\varphi'$ to $M\setminus U(x)$ is the identity map.
  
\end{lemma}

\begin{remark}\label{orbisympextension}
Note that  Lemma \ref{sympextension} has a version for orbifolds in which one should assume that the linear automorphism induced on the tangent orbi-space to $x$ by $\varphi$ is isotopic to the identical map as a $G_x$-map, where $G_x$ is the stabilizer of $x$. 
\end{remark}

\subsection{Hamiltonian $S^{1}$-actions on del Pezzo surfaces}\label{toricdelPezzo}

Recall that a symplectic $4$-manifold has a Hamiltonian $S^{1}$-action with isolated fixed points if and only if it is symplectomorphic to a toric surface \cite[Theorem 5.1]{Ka}. 

In this subsection, we prove some basic facts about actions on the del Pezzo surfaces. 
 There are $5$ toric del Pezzo's: $\C \PP^{2}$ blown up in up to $3$ non-aligned points and $\C\PP^{1} \times \C\PP^{1}$. 
 
One of the useful results proven here is that if an $S^{1}$-action on a del Pezzo contains a fixed point with weights equal to any of $\{1,1\}$, $\{-1,-1\}$ or $\{1,-1\}$, then the Hamiltonian is bounded between $-3$ and $3$.

Let $(X,\omega)$ be a toric symplectic $4$-manifold with moment map $\mu : X \rightarrow \R^{2}$. Suppose we have a Hamiltonian $S^{1}$-action on $X$ generated by $H:X \rightarrow \R$. Then $H $ is the composition of $\mu$ with a linear projection to $\R$ \cite[Theorem 5.1]{Ka}. The choice of such a projection amounts to a choice of linear embedding $S^{1} \hookrightarrow \mathbb{T}^{2}$.

\begin{definition} A \textit{boundary divisor} in $X$ is the pre-image $\mu^{-1}(E)$ where $E$ is an edge of the moment polygon.
\end{definition}

Note that boundary divisors are gradient spheres. Recall that a fixed point $p$ is called \textit{extremal} if $H$ attains its minimum or maximum at $p$. 

\begin{lemma}\label{dum} Let $(X,\omega,\mu)$ be a toric symplectic $4$-manifold with a Hamiltonian $S^{1}$-action generated by Hamiltonian $H = L \circ \mu$, where $\mu$ is the toric moment map and $L:\R^{2} \rightarrow \R$ is a linear projection. Suppose further that the $S^{1}$-action has isolated fixed points. 
\begin{enumerate}

\item A gradient sphere in $X$ containing a non-extremal fixed point is a boundary divisor.

\item Isotropy spheres in $X$ are boundary divisors.
\end{enumerate}
\end{lemma}
\begin{proof}
1. There are precisely two gradient spheres containing a non-extremal fixed point $p$ (by considering a local model for the action around $p$). These are the two boundary divisors containing $p$. 

2. Let $p$ be an extremal fixed point. Then since the $S^{1}$-action is effective, any gradient sphere containing $p$ with weight greater than $1$ is a boundary divisor. Hence 2 follows from 1.
\end{proof}

Next, we state a Proposition due to Karshon which will be useful for us.

\begin{proposition}
 \cite[Proposition 5.2]{To} \label{equallemma} Suppose that $(N,\omega)$ is symplectic $4$-manifold with a Hamiltonian $S^{1}$-action with isolated fixed points. Let $p_{\min},p_{\max} \in N$ be the fixed points where the Hamiltonian attains its minimum and maximum respectively. 

Then the multiplicity of the weight 1 at $p_{\min}$ is equal to the number of fixed points of index $2$ with a weight equal to $-1$.  Similarly, the multiplicity of the weight -1 at $p_{\max}$ is equal to the number of fixed points of index $2$ with a weight equal to $1$. 
\end{proposition}

\begin{corollary} \label{neededcor}
Suppose that $(N,\omega)$ is symplectic $4$-manifold with a Hamiltonian $S^{1}$-action with isolated fixed points and Hamiltonian $H : N \rightarrow \mathbb{R}$. Suppose that there exists fixed points $p_{1},p_{2} \in N$, both with weights $\{-1,n\}$ $(n>1)$ and such that $H(p_{1})=H(p_{2})$. Then the only fixed point contained in $H^{-1}([H_{\min},H(P_{1}))) $ is $p_{\min}$.
\end{corollary}
\begin{proof}
By Proposition \ref{equallemma} the weights at $p_{\min}$ are $\{1,1\}$. Consider the two boundary divisors $S_{1}$ and $S_{2}$ containing $p_{\min}$. Then for each $i$, one of the weights at $(S_{i})_{\max}$ is $-1$. By Proposition \ref{equallemma} these fixed points must be $\{p_{1},p_{2}\}$, the result follows.
\end{proof}

\begin{lemma}\label{fourbound}
Consider a Hamiltonian $S^1$-action with isolated fixed points on a toric del Pezzo surface $(X,\omega)$. Then any gradient sphere in $X$ with weight $1$ contains an extremal fixed point.
\end{lemma}

\begin{proof} Let $S$ be a gradient sphere with weight $1$. Suppose that neither of $S_{\min},S_{\max}$ is an extremal fixed point. Then the weights at $S_{\max},S_{\min}$ are $\{n,-1\},\{1,-n'\}$ respectively where $n,n'>0$. Hence, by the weight sum formula (\ref{weightcho}) we have that $H(S_{\min}) \geq  0 \geq H(S_{\max})$ which is a contradiction.
\end{proof}

\begin{remark} Note that amongst the boundary divisors there are at most $4$ containing an extremal fixed point of $X$. Hence, by Lemma \ref{fourbound} there are at most $4$ boundary divisors that can have weight $1$.
\end{remark}

\begin{lemma}\label{4small}

Let $(X,\omega)$ be a toric del Pezzo surface and $S \subseteq X$ be a boundary divisor. Then \begin{displaymath}
\int_{S} \omega \leq 3.
\end{displaymath}
\end{lemma}
\begin{proof}Note that $S$ is the strict transform of a boundary divisor in a minimal model for $X$ or the exceptional curve of a blow-up. Exceptional curves have area $1$ in the anti-canonical polarization. The minimal model for $X$ is either $\C\PP^{2}$ or $\C\PP^1 \times \C\PP^1$, and boundary divisors in these surfaces have area $3$ and $2$ respectively.
\end{proof}

\begin{corollary}\label{us}
Let $(X,\omega)$ be a toric del Pezzo surface with a Hamiltonian $S^{1}$-action with isolated fixed points. Let $p_{\min}$ be the fixed point where $H$ attains its minimum. Suppose that $p$ is a fixed point with weights $\{-1,n\}$ where $n  \geq 2$. Then
\begin{enumerate}
\item $H(p) -  H(p_{\min}) \leq 3$.
 
\item The weights at $p_{\min}$ are $\{1,m\}$ where $m \geq (H(p) -  H(p_{\min}) )$.

\item There is no fixed point $p' \in X$ such that $H(p') \in (H(p_{\min}),H(p))$.
\end{enumerate}
\end{corollary}
\begin{proof}
1. Let $S$ be the boundary divisor with $S_{\max} = p$. By Lemma \ref{fourbound} we have that $S_{\min} = p_{\min}$. Note also that $S$ has weight $1$. Hence, by Lemma \ref{gradsphereint} combined with Lemma \ref{4small} we have that $H(S_{\max}) - H(S_{\min}) = \int_{S} \omega \leq 3$.

2. Note that the weights at $p_{\min}$ are $\{1,m\}$. By the weight sum formula (\ref{weightcho}) 
$$H(p) = -n+1,\;\;H(p_{\min}) = - 1- m.$$ 
Therefore $H(p) - H(p_{\min}) = m+(-n+2)\le m$, since $n\ge 2$.

3. The weights at $p_{\min}$ are $\{1,m\}$ for some $m \geq 1$. Consider the two boundary divisors with minimum equal to $p_{\min}$, say $S_{1}$ with weight $1$ and $S_{2}$ with weight $m$.  Since $(S_{1})_{\max} = p$ we have $H((S_{1})_{\max}) = H(p)$. Since $S_{2}$ has weight $m$, $H((S_{2})_{\max}) \geq H(p_{\min}) +m $ by Lemma \ref{gradsphereint}. Finally, $H(p_{\min})+m \geq H(p)$ by 2. The lemma follows. 
\end{proof}


\begin{lemma} \label{calc} Let $(X,\omega)$ be a toric del Pezzo surface with a Hamiltonian $S^1$-action. Suppose $X$ has a fixed point with weights equal to any of $\{ 1,  1\}$ , $\{-1,-1\}$ or $\{1,-1\}$. Then we have

1) The weights of $S^1$-action on all boundary divisors are $0$, $1$ or $2$.

2) $H(X)\subseteq [-3,3]$.
\end{lemma}

\begin{proof} 1) Note that the $S^1$-action on $X$ extends to a holomorphic $\mathbb C^*$-action. We can blow up $X$ in $4-b_2(X)$  points fixed by $\mathbb C^*$, to get the toric del Pezzo surface $X'$ isomorphic to  $\mathbb CP^2$ blown up in $3$ non-aligned points. The $\mathbb C^*$-action on $X$ lifts to a $\mathbb C^*$-action on $X'$.

The surface $X'$ has $6$ boundary divisors and the weights on its opposite  divisors coincide. Since $X$ has a fixed point with weights $\{\pm 1, \pm 1 \}$, it follows that the weights of the $\mathbb C^*$-action on $4$ boundary divisors of $X'$ are equal to $1$. It is not hard to check that there are exactly two such $\mathbb C^*$-actions on $X'$ up to a conjugation by an automorphism of $X'$. One of the actions has weight $0$ on two remaining boundary divisors and the other has weight $2$ on two remaining divisors. Note finally, that the weights of boundary divisors of $X$ form a subset of the weights of boundary divisors of $X'$.

2) Since by 1) the weights of the $S^1$-action at $X_{\min}$ and $X_{\max}$ are coprime and are equal to $0$, $\pm 1$ or $\pm 2$, the statement follows from the weight sum formula \ref{relcho} (\ref{weightcho}).
 \end{proof}

\section{Proof of Theorem \ref{maintheorem} in the case $\dim(M_{\min}) = 4$ and $\dim(M_{\max}) \geq 2$} \label{4dimcasection}

Throughout this section we assume that $M$ is a $6$-dimensional symplectic Fano manifold with a Hamiltonian $S^{1}$-action. Our aim is to prove Theorem \ref{maintheorem} in the case when $\dim(M_{\min}) = 4$ and $\dim(M_{\max}) \geq 2$.  Namely, we prove in this case that  $M_{\min} $ is diffeomorphic to a del Pezzo surface, this is done in Theorem \ref{fourtheo}.

To prove this result we study the reduced spaces $(M_{x},\omega_{x})$, and show that as $x$ approaches $0$ from below, the cohomology class of $\omega_{x}$ approaches the first Chern class of $M_{x}$, in an appropriate sense. In the course of the proof we show that the $S^1$-action on $M$ should be semi-free. 

Let us mention related results of Cho \cite{Cho} and Futaki \cite{Fu}. Cho shows that for a symplectic Fano manifold $X$ with a semi-free Hamiltonian $S^1$-action such that $0$ is a regular value of the Hamiltonian, the reduced space $X_0$ is a symplectic Fano. Futaki has proven the analogous result in the K\"ahler category. 


\subsection{Maximal downward chains}
In this subsection we introduce and study maximal downward chains. Recall that $w(S)$ denotes the weight of a gradient sphere $S$ and $S_{\min},S_{\max}$ are the fixed points in $S$ where the Hamiltonian attains its $\min/\max$.
\begin{definition} \label{maximal} Let $X$ be a Hamiltonian $S^1$-manifold.
A \emph{maximal downward chain} in $X$ consists of a sequence of fixed points $p_1,\ldots, p_k  \in X$ ($k \geq 2$) along with a sequence of gradient spheres $S_1,\ldots, S_{k-1}$, such that the following conditions hold. 
\begin{enumerate}
\item $(S_{i})_{\max}= p_i$ and $(S_{i})_{\min} = p_{i+1}$ for each $i$. 
\item  $w(S_i)>1$ for each $i$.
\item  The weights at $p_k$ are all greater than or equal to $-1$.

\end{enumerate}
\end{definition}

\begin{lemma} \label{chain} Let $X$ be a Hamiltonian $S^1$-manifold and
let $F$ be a fixed submanifold in $X$ with a weight $w$ such that $|w|>1$. Then there exists a maximal downward chain $p_{1},\dots,p_{k}$ such that $p_{i} \in F$ for some $i$ and $|w| = w(S_{j})$ for some $j$.
\end{lemma}
\begin{proof}
Note first that if $S$ is a gradient sphere with $w(S)>1$, then  $S_{\max},S_{\min}$ may be extended to form a maximal downward chain. Indeed, set $p_{1} = S_{\max}$ and $p_{2}=S_{\min}$. Inductively if $p_{k}$ has a weight $v<-1$, then we can find a gradient sphere $S_{k}$ of weight $v$ with $(S_{k})_{\max} = p_{k}$, and set $(S_{k})_{\min} = p_{k+1}$ continuing the sequence. We may continue until all the weights at $p_{k}$ are at least $-1$, i.e. we have a maximal downward chain.

If $F$ has a weight $w$, then any point of $F$  is contained in a gradient sphere of weight $|w|$. So, if $|w|>1$, then by the above there is a maximal downward chain such that $p_{i} \in F$ for $i=1$ or $i=2$.
\end{proof}

\begin{lemma}\label{chainres} Let $M$ be a symplectic Fano $6$-manifold with a Hamiltonian $S^1$-action. Suppose that $\dim(M_{\min}) = 4$, then the following holds. \begin{enumerate}
\item Any maximal downward chain $p_{1},\dots,p_{k} $ in $M$  satisfies $k =2$ and $w(S_{1}) = 2$. Furthermore, $p_{1}$ and $p_{2}$ are isolated fixed points, the weights at $p_{2}$ are $\{-1,-1,2\}$ and $H(p_{1}) \geq 2$. 

\item Suppose $M$ contains a fixed submanifold $F$ with a weight $w$ such that $|w|>1$. Then $F$ is an isolated fixed point. Moreover, $F$ coincides either with $p_{1}$ or with $p_{2}$ for some maximal downward chain as in 1. 

\item The weights at each fixed submanifold in $M$ have modulus at most 2.
\end{enumerate}
\end{lemma}
\begin{proof}
1. Firstly, note that since the action is effective the weights at $M_{\min}$ are $\{0,0,1\}$. Hence,  $H(M_{\min}) = -1$ by the weight sum formula \ref{relcho}(\ref{weightcho}).

 Let $p_{1}, \dots, p_{k}$ be a maximal downward chain in $M$. First, we will prove that the weights at $p_{k}$ are $\{-1,-1,2\}$. Let $w_1\le w_2\le w_3$ be the weights at $p_k$.   By the weight sum formula \ref{relcho}(\ref{weightcho}) we have that
$$-w_1-w_2-w_3=H(p_k).$$ 

By condition 2) of Definition \ref{maximal}, $p_{k}$ has a weight with modulus at least $2$, hence $p_{k} \notin M_{\min}$. Therefore, $p_{k}$ must have some negative weight, and so $w_{1}<0$. On the other hand, condition 3) of Definition \ref{maximal} tells us that $w_1\ge -1$, so $w_{1}=-1$.

Since $p_{k} \notin M_{\min}$, $H(p_{k}) \geq  H(M_{\min}) + 1 = 0$. By the above expression for $H(p_{k})$, $$-1 + w_{2} + w_{3} \leq 0.$$ Condition 2) of Definition \ref{maximal} implies $w_3\ge 2$. Together with the above inequality, this implies that $w_{2} \leq -1$, but since also $w_{2} \geq w_{1} = -1$ we deduce that $w_1=w_2=-1$, $w_3=2$ and $H(p_{k}) = 0$. Let us now show that $k=2$ and $H(p_1)\ge 2$.

Since $w_3=2$, one of the weights at $p_{k-1}$ is $-2$. Let $w_1'\le w_2'$ be the two other weights. From \cite[Lemma 2.6]{To} it follows that $w_1'$ and $w_2'$ are odd integers, so that $p_{k-1}$ is an isolated fixed point with weights $\{-2,w_{1}',w_{2}'\}$.

Now we assume for a contradiction that $k\geq 3$. Applying condition 2) of Definition \ref{maximal} to $S_{k-2}$ we deduce that $w_2'\ge 3$. Applying the weight sum formula \ref{relcho}(\ref{weightcho}) to $p_{k-1}$ we see that 
$$-w_1'=H(p_{k-1})-2+w_2'\ge H(p_{k-1})+1.$$
Consider finally the gradient sphere $S$ with weight $|w_{1}'|$ and  $S_{\max} = p_{k-1}$. From the above inequality, along with Lemma \ref{gradsphereint}, we must have $S_{\min} \in M_{\min}$, which contradicts Lemma \ref{resweight}. 

Hence, $k=2$ and $p_{1} = p_{k-1}$ is an isolated fixed point. By applying Lemma \ref{gradsphereint} to $S_{1}$ we see that $H(p_{1}) \geq H(p_{2}) + 2 = 2$. 

2. Let $F$ be a fixed submanifold with a weight $w$ such that $|w|>1$. By Lemma \ref{chain} we can form a maximal downward chain with $p_{i} \in F$ for some $i$, which must be of the required form by 1. In particular, $F$ is an isolated fixed point.

3. Suppose for a contradiction that there was a fixed submanifold $F$ with a weight $w$, such that $|w|>2$, then by Lemma \ref{chain} we could form a maximal downward chain such that $w(S_{j})>2$ for some $j$, contradicting 1. \end{proof}

\begin{corollary} \label{newref}
Let $M$ be a symplectic Fano $6$-manifold with a Hamiltonian $S^1$-action. Suppose that $\dim(M_{\min}) =4$ and $\dim(M_{\max}) \geq  2$. Then the $S^{1}$-action on $M$ is semi-free and any component of $M^{S^{1}}$ in the level set $H^{-1}(0)$ is a surface with weights $\{-1,1,0\}$.
\end{corollary}
\begin{proof}
By Lemma \ref{chainres}(2), the modulus of each weight at $M_{\max}$ is at most $1$. Hence, the weights at $M_{\max}$ are $\{-1,-1,0\}$ or $\{-1,0,0\}$ depending whether $\dim(M_{\max}) = 2 $ or $4$, and so $1 \leq H_{\max} \leq 2$ by the weight sum formula \ref{relcho}(\ref{weightcho}). Since by the same reasoning $H_{\min} = -1$, we deduce that $0 \in (H_{\min},H_{\max})$. 

Suppose for a contradiction that there exists a fixed submanifold $F$ with a weight $w$ such that $|w| \geq 2$. Then by Lemma \ref{chainres}(2) there is a pair of isolated fixed points $p_{1},p_{2} \in M$ with $H(p_{1}) \geq 2 \geq H_{\max}$, contradicting the assumption that $\dim(M_{\max}) \geq 2$. Hence all weights have modulus at most $1$, i.e. the action is semi-free.

To see that the only fixed submanifolds in the level set $H^{-1}(0)$ are surfaces, note that any isolated fixed point $p$ has weights $\{-1,1,1\}$ or $\{-1,-1,1\}$, so satisfies $H(p) = \pm 1$ by the weight sum formula \ref{relcho}(\ref{weightcho}). Therefore, the only possible fixed submanifolds in the level set $H^{-1}(0)$ are fixed surfaces with weights $\{-1,1,0\}$ as claimed.
\end{proof}

\subsection{Theorem \ref{fourtheo} and its proof}

\begin{theorem} \label{fourtheo}Let $M$ be a symplectic Fano $6$-manifold with a Hamiltonian $S^1$-action. Suppose that $\dim(M_{\min}) =4$ and $\dim(M_{\max}) \geq  2$. Then $M_{\min} $ is diffeomorphic to a del Pezzo surface.

\end{theorem}
\begin{proof}
We will proceed by applying Lemma \ref{DH} to the reduced spaces $M_x$  of $M$. By the symplectic Fano condition $[\omega] = c_{1}(M) $ we get
$$[\omega|_{M_{\min}}]=c_{1}(TM|_{M_{\min}})=c_{1}(TM_{\min})+c_{1}(N),$$
where $N $ is the normal bundle of $M_{\min}$.

For $x \in (-1,0)$ we have that $e(H^{-1}(x)) = c_{1}(N)$ since for small $\varepsilon$, $H^{-1}(-1+\varepsilon)$ is isomorphic to the unit circle bundle of $N$ and all values of $H$ in $(-1,0)$ are regular. Applying Lemma \ref{DH} for $x \in (-1,0)$ we have \begin{displaymath}
[\omega_{x}] = c_{1}(TM_{\min})  - x c_{1}(N). 
\end{displaymath}
Hence $[\omega_{x}]$ tends to $[c_{1}(TM_{\min})]$ as $x$ tends to $0$.

Since by Corollary \ref{newref} all fixed submanifolds on level $0$ are fixed surfaces of weights $\{-1,1,0\}$, by \cite[Theorem 10.1]{GS} we may construct a Hamiltonian $S^{1}$-manifold  $(\tilde{M},\tilde{\omega},\tilde{H})$ and find $\varepsilon >0$ with the following properties: \begin{enumerate}
\item There is a smooth map $\Psi: \tilde{M} \rightarrow H^{-1}(-\varepsilon,\varepsilon) $, such that $ \Psi^{*}H = \tilde{H}$.
\item $\Psi|_{\tilde{H}^{-1}(-\varepsilon,0)} $ is an equivariant symplectomorphism onto $H^{-1}(-\varepsilon,0)$.
\item The $S^{1}$-action on $\tilde{M}$ has no fixed points.
\item $\tilde{M}_{0}$ is homeomorphic to $M_{0}$.
\end{enumerate}

Applying Lemma \ref{DH} as above, we see that $\tilde{M}_{0}$ is a symplectic Fano $4$-manifold, and so it is diffeomorphic to a del Pezzo surface \cite[Theorem 1.3]{OhtaOno}. Note that for any $x \in (-\varepsilon,0)$, $\tilde{M_{x}}$ is diffeomorphic to $\tilde{M}_{0}$ since $\tilde{M}$ has no fixed points. By conditions 1. and 2. above $\tilde{M}_{x}$ is diffeomorphic to $M_{x}$, and so $M_{x}$ is diffeomorphic to a del Pezzo surface. Lastly, note that the gradient map $gr_{-1}^{x}$ provides a diffeomorphism between $M_{x}$ and $M_{\min}$. \end{proof}

\subsection{Restrictions on fixed point data}\label{restictions04}

Finally, we give restrictions on the fixed point data, in the case when $\dim(M_{\min}) = 4$ and $\dim(M_{\max}) = 0$. The results of this section are used only in Section \ref{theorem04sec}.

\begin{proposition}\label{listofweights} Let $M$ be a symplectic Fano $6$-manifold with a Hamiltonian $S^1$-action such that $\dim(M_{\min}) =4$  and $M_{\max}$ is a point. Then the weights at $M_{\max}$ are either $\{-1,-1,-1\}$ or $\{-1,-1,-2\}$. Furthermore, non-extremal isolated fixed points in $M$ must have weights  $\{1,-1,-2\}$, $\{2,-1,-1\}$ or $ \{1,-1,-1\}$. 
\end{proposition}
\begin{proof}
Let $p$ be an isolated fixed point in $M$ such that there is a weight $w$ at $p$ with $|w|>1$. We will show that the weights at $p$ are either $\{-1,-1,2\}$, or $\{-1,1,2\}$, or $\{-1,-1,-2\}$. From this statement it clearly follows that the weights at $M_{\max}$ are either $\{-1,-1,-1\}$ or $\{-1,-1,-2\}$.

By Lemma \ref{chainres}(2), $p $ is equal to one of the fixed points in a maximal downward chain $p_{1},p_{2}$ of isolated fixed points such that the weights at $p_{2}$ are $\{-1,-1,2\}$ and $H(p_{1}) \geq 2$. Hence we may assume that $p=p_{1}$. By \cite[Lemma 2.6]{To} the weights at $p$ are $\{-2,w_{1}',w_{2}'\}$ where $w_{i}'$ are odd integers. Hence, by Lemma \ref{chainres}(3) $w_{i}' = \pm 1$ for each $i$. Consider now two cases.

\textbf{Case one} $p \neq M_{\max}$. The weights at $p$ are $\{1,-2,w\}$, where $w=1$ or $-1$ .  By the weight sum formula $H(p) = 1-w \geq 2$, so we must have $w = -1$.

\textbf{Case two} $p = M_{\max}$. In this case clearly the weights at $p$ are $\{-1,-1,-2\}$. 

Hence the weights at $p$ fall into the three possibilities claimed.

To finish the proof it remains to show that non-extremal isolated fixed points with weights of modulus $1$, must have weights $\{-1,-1,1\}$. The only other possibility is a fixed point $p $ with weights $\{-1,1,1\}$. In this case $H(p)=-1 = H_{\min}$ by the weight sum formula \ref{relcho}(\ref{weightcho}), which is a contradiction.
\end{proof}

\begin{definition} Let $M$ be a symplectic Fano $6$-manifold with a Hamiltonian $S^{1}$-action such that $\dim(M_{\min}) =4$ and $M_{\max}$ is a point. Fixed points in $M$ with weights $\{1,-1,-2\}$, $\{2,-1,-1\}$ and $ \{1,-1,-1\}$ are referred to as fixed points of types $A$, $B$ and $C$  respectively. We define $n_{A}$, $n_{B}$ and $n_{C}$ to be the number of fixed points in $M$ of type $A$, $B$ and $C$ respectively.
\end{definition}
 
\begin{lemma}\label{nAnB}
Let $M$ be a symplectic Fano $6$-manifold with a Hamiltonian $S^{1}$-action such that $\dim(M_{\min}) =4$ and $M_{\max}$ is a point.
\begin{enumerate}
\item In the case when the weights at $M_{\max}$ are $\{ -1,-1,-1 \}$, we have $n_{A} =  n_{B}$.

\item In the case when the weights at $M_{\max}$ are $\{ -1,-1,-2 \}$, we have $n_A+1=n_B$. 

\item Any $2$-dimensional component of $M^{S^1}$ is contained in the level set $H^{-1}(0)$.
\end{enumerate}
\end{lemma}

\begin{proof}

Statements 1 and 2 follow from applying Lemma \ref{chainres}(1) to a maximal downward chain with $p_{1}$ equal to a fixed point with weights $\{-1,1,-2\}$ or $\{-1,-1,-2\}$. 

By Lemma \ref{chainres}(2) any fixed surface $\Sigma$ has weights of modulus at most one. Hence the only possibility for the weights is $\{-1,1,0\}$, and Statement 3 follows from applying the weight sum formula \ref{relcho}(\ref{weightcho}). \end{proof}

\section{Hamiltonian $6$-manifolds with $2$-dimensional extrema: topology}\label{sectopology}

For the remainder of the article the main object of our consideration will be  $6$-dimensional Hamiltonian $S^1$-manifolds such that $M_{\min}$ and $M_{\max}$ are surfaces of positive genus $g>0$. 

Here, we begin to lay foundations for the construction of a \textit{symplectic fibre}, which is central to the proof of our main results. To be more specific, we will construct the symplectic analogue of an invariant K\"ahler hypersurface that always exists in the K\"ahler case, which we will now recall.

{\bf K\"ahler case.} Assume that the $S^1$-action on $M$ preserves a K\"ahler metric $g$ compatible with $\omega$. In this case one can show that the following holds.

\begin{enumerate}

\item The $S^1$-action on $M$ extends to a holomorphic $\mathbb C^*$-action on $M$.

\item There is a holomorphic map $\phi: M\to M_{\min}$ sending any $\mathbb C^*$-orbit to a point in $M_{\min}$. A generic $\mathbb C^*$-orbit is sent by $\phi$ to the unique point of $M_{\min}$ in its closure. 

\item The map $\phi$ induces a holomorphic map $\phi_c$ from each reduced space $M_c$ to $M_{\min}$. For a generic $x\in M_{\min}$ the fibre $\phi_c^{-1}(x)$ is a smooth rational curve.  $M_c$ is a $2$-dimensional complex orbifold for regular $c$  and $\phi_c^{-1}(x)$ is a sub-orbifold.

\item For a generic point $x\in M_{\min}$ the preimage $\phi^{-1}(x)$ in $M$ is a smooth complex surface with a $\mathbb C^*$-action with isolated fixed points. Hence it is a toric surface.
\end{enumerate} 

{\bf Symplectic case.} Our strategy will be to show that even in the case when a compatible $S^1$-invariant K\"ahler metric does not exist on $M$, many of the above properties have  topological and symplectic analogues\footnote{It would be  interesting to learn if a compatible $S^1$-invariant K\"ahler metric does exist, but we expect it is not always the case.}. All these analogues will be worked out in Sections \ref{sectopology}-\ref{sympfibresec}, we briefly summarize this work below.

\begin{itemize}

\item In the present section we introduce a topological analogue of the above holomorphic map $\phi : M\to M_{\min}$. We give a cohomological characterization of the fibre of the induced map $\phi_c:M_{c}\to M_{\min}$ (see Theorem \ref{Thefibreclass}).

\item Section \ref{locKahlsec} serves as a bridge between symplectic and complex situations. In particular, we introduce K\"ahler structures compatible with the symplectic one on an open neighbourhood of the set of isotropy spheres and fixed points in $M$.

\item Section \ref{orbispheresection} deals with existence problems for symplectic orbi-spheres in symplectic orbifolds of dimension $4$ with cyclic stabilizers. We prove there Theorem \ref{twoinoneOrbisphere}.

\item In Section \ref{sympspheresec} we use results of  Section \ref{orbispheresection} to prove  Theorem \ref{fibreinregular}, stating that for all regular values of $c$ the reduced space $M_c$ contains a symplectic orbi-sphere in the cohomology class of the fibre of the map $\phi_c: M_c\to M_{\min}$. An analogue of this statement for critical values is proven in Theorem \ref{critfibres}.

\item Finally, in Section \ref{sympfibresec} we establish the existence of symplectic fibre - a symplectic hypersurface in $M$ that sits in the same homology class as the preimage of a generic point in $M_{\min}$ under the map $\phi$.
\end{itemize}

\subsection{Retraction of $M$ to $M_{\min}$}

In this subsection we  introduce an $S^1$-invariant retraction $\phi$ from $M$ to $M_{\min}$.

\begin{lemma} \label{retractmin} Let $M$ be a compact Hamiltonian $S^1$-manifold such that $M_{\min}$ is a surface of genus $g>0$. Then there exists an $S^1$-invariant retraction $\phi: M\to M_{\min}$. 
\end{lemma}

\begin{proof}\label{retractionlemma} Recall that by Theorem \ref{Li} the homomorphism $i_{*}:\pi_1(M_{\min})\to \pi_1(M)$ induced by the inclusion is an isomorphism. By Theorem \ref{Li2} the quotient map $Q:M\rightarrow M/S^{1}$ induces an isomorphism $Q_{*}: \pi_{1}(M) \rightarrow \pi_{1}(M/S^{1})$.

 Since $M_{\min}$ is a $K(\pi,1)$ space, any homomorphism $\varphi:\pi_1(M/S^{1})\to \pi_1(M_{\min})$ can be induced by a continuous map $\pi: M/S^{1}\to M_{\min}$. It follows that the isomorphism $i_{*}^{-1} \circ Q_{*}^{-1}: \pi_1(M/S^{1})\to \pi_1(M_{\min})$ is induced by a continuous map $\pi$. By the homotopy extension property $\pi$ can be homotoped to a retraction $\Phi: M/S^{1} \rightarrow M_{\min}$. Finally, the $S^{1}$-invariant retraction is given by $\phi = \Phi \circ Q$. \end{proof}

\begin{corollary} Let $x$ be a value of $H$ and let $M_x$ be the reduced space. Let 
$\phi_{x} : M_{x} \rightarrow M_{\min}$, be the map induced by  the $S^1$-invariant retraction $\phi$. Then $\phi_{x}$ induces an isomorphism on $\pi_{1}$.
\end{corollary}
\begin{proof} This corollary follows from the proof \cite[Theorem 0.1]{Li}. In \cite{Li} a natural isomorphism of fundamental groups of all reduced spaces $M_x$ is given and this isomorphism  commutes with the homomorphisms induced by maps $\phi_{x}$.
\end{proof}

We will need one more technical statement in the setting of Lemma  \ref{retractionlemma}.

\begin{corollary}\label{gradphicomute} Let $c\in (H_{\min}, H_{\max})$ be a critical value of $H$, and $d>c$ (or $d<c$) be such that the interval $[c,d]$ (or $[d,c]$) does not contain critical values of $H$ apart from $c$. Consider the gradient map $gr_c^d: M_d\to M_c$. Then the maps $\phi_c\circ gr_c^d$ and $\phi_d$ from $M_d$ to $M_{\min}$ are homotopic.   
\end{corollary}
\begin{proof} Consider the family of maps $\phi_t\circ gr_t^d: M_d\to M_{\min}$ parametrised by $t\in [c,d]$. This family of maps  provides a homotopy between $\phi_c\circ gr_c^d$ and $\phi_d$. \end{proof}

\subsection{The fibre class}

\begin{definition} Let $M$ be a Hamiltonian $S^1$-manifold of dimension $6$ with $M_{\min}$ a surface of positive genus and let $A\in H^2(M_{\min})$ be the class with $\int_{M_{\min}} A=1$. Let $\phi: M\to M_{\min}$ be an $S^1$-invariant retraction and $\phi_c: M_c\to M_{\min}$ be the induced map. The class $\phi_c^*(A)\in H^2(M_c)$ and its Poincar\'e dual are called the \emph{fibre classes} of $M_c$. We will denote $\phi_c^*(A)$ by ${\cal F}^*_c$ and the Poincar\'e dual class by ${\cal F}_c$.
\end{definition}

Recall from Theorem \ref{Mcsingature} the intersection form on $H^{2}(M_{c})$ has signature $(1,n)$. The following lemma from linear algebra will be useful for studying cohomology classes in $M_{c}$.

\begin{lemma}\label{quadralgebra} Let $Q$ be a symmetric bilinear form on $\mathbb R^{n+1}$ of signature $(1,n)$. Suppose that $v,w\in \mathbb R^{n+1}$ are non-zero vectors such that $Q(w,w)>0$ and $Q(v,v)=0$. Then $Q(w,v)\ne 0$.
\end{lemma}

\begin{proof} Consider  $\mathbb R^2\subset \mathbb R^{n+1}$ spanned by $w$ and $v$. Since $Q(w,w)>0$ the signature of $Q$ restricted to $\mathbb R^2$ is $(1,1)$. Now, since $Q(v,v)=0$ we see that $v^{\perp}\subset \mathbb R^2$ is spanned by $v$. Hence $Q(w,v)\ne 0$.
\end{proof}

\begin{theorem}\label{Thefibreclass} For $c\in (H_{\min}, H_{\max})$  the fibre class $\phi_c^*(A)={\cal F}^*_c$ satisfies the following properties 
\begin{enumerate}
\item $\phi_c^*(A)^2=0$.

\item $\phi_c^*(A)\ne 0\in H^2(M_c)$.

\item $\int_{M_c}\phi_c^*(A)\cdot \omega_c>0$.
\end{enumerate}

\end{theorem}

\begin{proof}   1) The equality $\phi_c^*(A)^2=0$ is clear.

2) For small $\varepsilon$ and $c_{\varepsilon}=H_{\min}+\varepsilon$, the space $M_{c_{\varepsilon}}$ is an $S^2$-bundle over $M_{\min}$. The class $\phi_{c_{\varepsilon}}^*(A)$ is Poincar\'e dual to the class of a symplectic $S^2$-fibre of the fibration $M_{c_{\varepsilon}}\to M_{\min}$. In particular, $\phi_{c_{\varepsilon}}^*(A)\ne 0$.
To show that $\phi_c^*(A)\ne 0\in H^2(M_c)$ for all $c\in (H_{\min}, H_{\max})$, we will increase $c$ and prove that   $\phi_c^*(A)$ does not vanish when $c$ passes a critical value of $H$. 

Let $c$ be a critical value of $H$ and let $\varepsilon$ be such that $c$ is the only critical value of $H$ in the interval $[c-\varepsilon, c+\varepsilon]$. It follows from Corollary \ref{gradphicomute} that 
$$i)\; \phi_{c-\varepsilon}^*(A)=(gr^{c-\varepsilon}_c)^*(\phi_c^*(A)),\;\;\;ii)\; \phi_{c+\varepsilon}^*(A)=(gr^{c+\varepsilon}_c)^*(\phi_c^*(A)).$$ 
Now, since by our assumptions $\phi_{c-\varepsilon}^*(A)\ne 0$ we deduce from {\it i}) that $\phi_c^*(A)\ne 0$. On the other hand the map $(gr^{c+\varepsilon}_c)^*$ is an injection, and  we deduce from {\it ii}) that $\phi_{c+\varepsilon}^*(A)\ne 0$.

3) Note that the function $\mu(c)=\int_{M_c}\phi_c^*(A)\wedge \omega_c$ is a continuous function of $c$. For $c_{\varepsilon}=H_{\min}+\varepsilon$ we have $\mu(c_{\varepsilon})>0$ since we integrate a symplectic form over a symplectic sphere. At the same time, $\omega_c\ne 0$ and by 2) $\phi_c^*(A)\ne 0$. So, applying Theorem \ref{Mcsingature} and Lemma \ref{quadralgebra} to the intersection form on $H^2(M_c)$   we see that $\mu(c)$ can not vanish for $c\in (H_{\min},H_{\max})$. This finishes the proof of 3).  
\end{proof}

\begin{corollary}\label{uniquefibre}Let $M_c$ be the reduced space at regular level $c$ and let $\phi_c: M_c\to M_{\min}$ be the projection. The fibre class ${\cal F}_c\in H_2(M_c)$ satisfies the following properties 
\begin{enumerate}
\item $\int_{{\cal F}_c} \omega_c>0$.
\item ${\cal F}_c^2=0$.
\item $\phi_{c*}({\cal F}_c)=0\in H_2(M_{\min})$.
\end{enumerate}
Moreover, ${\cal F}_c$ is the unique such class in $H_2(M_c)$ up to  multiplication by a positive constant.
\end{corollary}
\begin{proof} Property 1) for ${\cal F}_c$  corresponds to Property 3) in Theorem  \ref{Thefibreclass}. Property 2) corresponds to  Property 1) in Theorem  \ref{Thefibreclass}. Property 3) is obvious.

Let us show now the uniqueness of ${\cal F}_c$  up to a positive constant.  Let ${\cal G}_c$ be a class satisfying properties 1) -- 3). According to 3) ${\cal G}_c$ belongs to the kernel of ${\phi_c}_*: H_2(M_c)\to H_2(M_{\min})$. Note at the same time that $\ker {\phi_c}_*=(\phi_c^*A)^{\perp}$. Since the signature of the intersection form in $H_2(M_c)$ is $(1,n)$, there is a unique up to a multiplicative constant vector in $\ker {\phi_c}_*=(\phi_c^*A)^{\perp}$ with zero square. We deduce that ${\cal F}_c$  and ${\cal G}_c$ are proportional, and Property 1) tells us that the constant of proportionality is positive.
\end{proof}

We finish this section with the following lemma.

\begin{lemma} \label{Isotropy fibre}
Suppose that $N \subset M$ is an isotropy $4$-manifold such that $\dim(N_{\min}) = \dim(N_{\max}) = 2$. For any $c \in H(N)$ denote by $N_{c} \subset M_{c}$ the trace of $N$.  Then the following holds
\begin{itemize}
\item For $c \in H(N)$, $\langle \phi^{*}_{c}(A),N_{c} \rangle$ does not depend on $c$.

\item In particular, $N_{\min} \cdot \mathcal{F}_{H(N_{\min})} = N_{\max} \cdot \mathcal{F}_{H(N_{\max})}$.
\end{itemize}
\end{lemma}

\begin{proof}
The restriction $\phi|_{N}$ is $S^{1}$-invariant an so it descends to a continuous map $\phi_{S^{1}}: N/S^{1} \rightarrow M_{\min}$. The quotient space $N/S^{1}$ is homeomorphic to the product of $N_{\min}$ with an interval parametrised by $H$. Clearly, the degree of $\phi_{S^{1}}$ restricted to $N_{c} = N_{\min} \times \{c\}$ does not depend on $c$. Since this degree is equal to $\langle \phi^{*}_{c}(A),N_{c} \rangle$, the result follows.
\end{proof}

\section{Local K\"ahler structures}\label{locKahlsec}
We prove here three results that permit one to find a K\"ahler structure close to a specific collection of symplectic curves in a symplectic manifold or an orbifold. Theorem \ref{semilocalKahler} and Proposition \ref{kahleratfix} deal with symplectic manifolds with a Hamiltonian $S^1$-action, Theorem \ref{semilocal4orb} deals with $4$-dimensional symplectic orbifolds.

\begin{theorem}\label{semilocalKahler} Let $(M^{2n},\omega)$ be a Hamiltonian $S^1$-manifold and let $C_1,\ldots, C_k$ be all isotropy spheres in $M$. Then there exists a collection of open neighbourhoods $U_i$ of the spheres $C_i$, and an $S^1$-invariant K\"ahler metric $g$ on $U_1\cup\ldots \cup U_k$ compatible with $\omega$, satisfying the following  properties 
\begin{enumerate}
\item For each $(U_i,g)$ there is an effective, isometric and Hamiltonian action of $T^{n-1}$ commuting with the $S^1$-action and fixing $C_i$ point-wise. This $T^{n-1}$-action and the $S^1$-action generate together an effective $T^n$-action on $U_i$. 

\item Let $p\in M^{S^1}$ be a point of intersection of $C_i$ and $C_j$. Then the $T^n$-actions on $U_i$ and $U_j$ preserve $U_i\cap U_j$ and induce on it the same action. 
\end{enumerate}
\end{theorem}

To formulate the next result we need to give a definition.
\begin{definition}\label{semitoricdef} Let $(S,g)$ be a $2$-dimensional K\"ahler orbifold with cyclic stabilizers. Let $D_1,\ldots, D_n$ be the curves in the orbifold locus of $S$ and let $p_1,\ldots, p_k$ be the maximal orbi-points. 
We say that the metric $g$ is \emph{semi-toric} at the orbifold locus of $S$ if the following properties hold.

1) For each $D_i$ there is an isometric $S^1$-action on a neighbourhood of $D_i$ that fixes $D_i$ point-wise and preserves $g$. 

2) For each maximal orbi-point $p_m$ there is an effective isometric $T^2$-action on a neighbourhood of $p_m$ that fixes $p_m$ and preserves $g$.

3)  If $p_m$ belongs to $D_i$ then the action of $S^1(D_i)$ on a neighbourhood of $p_m$ is induced by a subgroup of $T^2(p_m)$. In particular, whenever $D_i$ and $D_j$ intersect, the two $S^1$-actions commute.
\end{definition}

\begin{theorem}\label{semilocal4orb} Let $(M^4,\omega)$ be a symplectic orbifold with cyclic stabilizers. Suppose  we are given a symplectic $T^2$-action in  a neighbourhood $U$ of the union of all maximal orbi-points in $M^4$.
Let $g$ be a $T^2$-invariant K\"ahler metric defined in $U$ and compatible with $\omega$. Then, after shrinking $U$, $g$ can be extended to a K\"ahler semi-toric metric compatible with $\omega$ and defined in a neighbourhood of $\Sigma_{\cal O}(M^4)$.
  
\end{theorem}

The proofs of both theorems are similar and rely on equivariant Darboux-Weinstein Theorem \ref{relative_local_neighbourhood} and the following gluing statement.

\begin{proposition}\label{gluinginKahler} Let $(U,\omega)$ be a symplectic orbifold with a Hamiltonian action of a compact group $G$. Suppose  $(U,\omega)$ admits a $G$-invariant K\"ahler metric, compatible with $\omega$. Let $\bar x=\{x_1,\ldots,x_m\}\subset U^G$ be a finite $G$-fixed subset. Suppose that for some neighbourhood $U_{\bar x}$ of $\bar x$ a $G$-invariant K\"ahler metric $g_x$ compatible with $\omega$ is chosen. Then after  shrinking $U_{\bar x}$ to a smaller neighbourhood $U_{\bar x}'$ of $\bar x$ one can extend $g_{\bar x}$ from $U_{\bar x}'$ to a $G$-invariant K\"ahler metric on $U$, compatible with $\omega$.
\end{proposition}

\begin{proof}

In the partial case when  $U$ is a manifold, $G$ is trivial and the set $\bar x$ consists of a  unique point $x$ this proposition is Corollary \ref{symplant}.  Working $G$-equivariantly with orbifolds instead of manifolds and using Theorem \ref{relative_local_neighbourhood} one can prove the full proposition without essential changes. \end{proof}

\subsection{Auxiliary torus actions}

To prove Theorem \ref{semilocalKahler} we need to show in particular, that any isotropy sphere in a Hamiltonian $S^1$-manifold $M^{2n}$ has a neighbourhood with a K\"ahler metric admitting a $T^n$-symmetry. We explain here how this is done.

Let $V=L_1\oplus\ldots\oplus L_{n-1}$ be a rank $n-1$ holomorphic vector bundle over $\mathbb CP^1$ and let $\cal V$ be the total space. $\cal V$ has the structure of a toric variety where the $(\mathbb C^*)^{n}$-action can be defined as follows. We write  $(\mathbb C^*)^{n}=\mathbb C^*\times (\mathbb C^*)^{n-1}$ and let the $(\mathbb C^*)^{n-1}$-factor act diagonally, i.e., fixing the base $\mathbb CP^1$ and preserving the splitting of $V$. To define the action of the first $\mathbb C^*$-factor consider the $\mathbb C^*$-action on $\mathbb C^2$ given by $t:(x,y)\to (tx,t^{-1}y)$ and identify $\mathbb CP^1$ with $\mathbb P(\mathbb C^2)$. This action can be lifted naturally to $V$ and commutes there with the diagonal $(\mathbb C^*)^{n-1}$-action.

Let $T^n\subset (\mathbb C^*)^n$ be the real torus acting on the total space $\cal V$ of $V$. Let $U\subset \cal V$ be a $T^n$-invariant neighbourhood of the zero section of $V$ and let $g$ be a $T^n$-invariant K\"ahler metric on $U$. We say that $(U,g)$ is  {\it a toric K\"ahler neighbourhood}. 

\begin{proposition}\label{toricKahlernei} Let $C$ be an isotropy sphere of a  Hamiltonian $S^1$-action on a symplectic manifold $M^{2n}$. Then there exits a holomorphic rank $n-1$ bundle $V$ over $\mathbb CP^1$,  a toric K\"ahler neighbourhood  $(U,g)$ of its zero section,  and a symplectic embedding $\varphi: (U,\omega_g)\to (M^{2n},\omega)$ satisfying the following properties.

\begin{enumerate}

\item $ \varphi$ sends $\mathbb CP^1$ to $C$.
\item For some subgroup $S^1$ of $T^n$ the map $\varphi$ is $S^1$-equivariant.
\end{enumerate}

\end{proposition}

\begin{proof} First, we explain how to construct a triple $(\mathbb CP^1,  V, g)$ consisting of a holomorphic bundle $V$ over $\mathbb CP^1$ and a $T^n$-invariant metric $g$ defined close to the zero section in the total space $\cal V$ of $V$. 

Introduce  an $S^1$-invariant holomorphic structure on $C$ compatible with $\omega$, making it biholomorphic to $\mathbb CP^1$, and let  $I:\mathbb CP^1\to C$ be  the identification map. Let $N_C$ be the normal bundle to $C$ in $M^{2n}$, i.e. the subbundle of $TM^{2n}|_C$ $\omega$-orthogonal to $TC$. Denote by $V$ the pull-back $I^*(N_C)$ on $\mathbb CP^1$. This bundle has a natural $S^1$-action. Moreover, $V$ is symplectic and we denote the corresponding symplectic form by $\omega_V$.

Let us split $V$ into a sum of $\omega_V$-orthogonal $S^1$-invariant rank $2$ symplectic subbundles, $V=L_1\oplus\ldots \oplus L_{n-1}$. Next, introduce  an $S^1$-invariant structure of holomorphic line bundle on each $L_i$ and denote by $\cal V$ the total space of $L_1\oplus\ldots \oplus L_{n-1}$.

By construction, we have a symplectic form on the restriction of $T\cal V$ to $\mathbb CP^1$, compatible with the holomorphic structure of $\cal V$. It is not hard to see, that this symplectic form is induces by  a $T^n$-invariant K\"ahler metric $g$ defined on a neighbourhood of $\mathbb CP^1$. Again, by construction the identification $I:\mathbb CP^1\to C$ naturally extends to a symplectic linear isomorphism from the normal bundle $V$ of $\mathbb CP^1$ to the normal bundle $N_C$ of $C$. To finish the proof it remains to apply Theorem \ref{relative_local_neighbourhood} 1). This permits us to extend the above isomorphism of normal bundles to a symplectic $S^1$-equivariant embedding from some $T^n$-invariant neighbourhood $(U,\omega_g)$ of $\mathbb CP^1$ to $(M^{2n},\omega)$.



\end{proof}

The following proposition will be used later and we state it without proof, since it is very similar to Proposition \ref{toricKahlernei}.

\begin{proposition}\label{kahleratfix} Let $(M^{2n},\omega)$ be a Hamiltonian $S^1$-manifold and let $\Sigma_h$ be a genus $h$ fixed surface. Then there exists a complex curve $C_h$, a collection of line bundles $L_1,\ldots, L_{n-1}$ and a K\"ahler metric $g$ on the projectivised bundle $\mathbb P(L_1\oplus \ldots \oplus L_{n-1}\oplus {\cal O})$, such that 

\begin{enumerate}

\item The metric $g$ is invariant under the diagonal $T^{n-1}$-action on the projectivised bundle. 

\item Denote by $C_{h,n}$ the embedding of $C_h$ in the projectivised bundle via $n$th summand.  Then there is a subgroup $S^1\subset T^{n-1}$ and an $S^1$-equivariant symplectomorphism from  a neighbourhood of $C_{h,n}$ to a neighbourhood of $\Sigma_h$.

\end{enumerate}

\end{proposition}

\subsection{Proofs of  Theorems  \ref{semilocalKahler} and \ref{semilocal4orb}}
Let us first prove Theorem \ref{semilocalKahler}.

\begin{proof} Let $x_1,\ldots,x_m$ be all the $S^1$-fixed points in $M^{2n}$ that belong to a least two isotropy spheres. By the equivariant Darboux theorem we can choose a small (possibly disconnected) neighbourhood $U_{\bar x}$ of the set $\{x_1,\ldots, x_m\}$ on which the $S^1$-action extends to an effective Hamiltonian $T^n$-action, and there exists a $T^n$-invariant K\"ahler metric $g$  compatible with  $\omega$. We will show that such a metric $g$ can be extended to a metric on neighbourhood of $C_1\cup\ldots\cup C_k$ that has all the desired properties.

For each isotropy sphere $C_i$ we apply Proposition \ref{toricKahlernei}
to get a toric K\"ahler neighbourhood $U_i$ of the zero section in a vector bundle over $\mathbb CP^1$. Neighbourhood $U_i$ admits a symplectic $S^1$-equivariant embedding $\varphi_i:U_i \to M$, sending $U_i$ to a neighbourhood of $C_i$. Using Theorem \ref{relative_local_neighbourhood} 2), after possibly shrinking $U_{\bar x}$ and $U_i$ we can assume that the embedding $\varphi_i$ is $T^n$-equivariant in a small neighbourhoods of the two $T^n$-fixed points in $U_i$.
 
We have now a K\"ahler metric $\varphi_{i_*}(g_i)$, compatible with $\omega$ defined on each $\varphi_i(U_i)$. Let $p$ and $q$ be the two $T^n$-invariant points in $U_i$ and let $\varphi(p)=x_r$ and $\varphi(q)=x_s$. Both metrics $\varphi_{i_*}(g_i)$ and $g$ are compatible with $\omega$ and $T^n$-invariant in a small neighbourhood of $x_r$ and $x_s$, but they need not be equal there. In order to cure this, we use Proposition \ref{gluinginKahler}. This proposition permits us to replace $\varphi_{i_*}(g_i)$ defined on $\varphi_i(U_i)$ by a K\"ahler $T^n$-invariant metric $g_i'$, defined on $\varphi_i(U_i)$, compatible with $\omega$, and equal to $g$ close points $x_r$ and $x_s$. To finish, we just need to shrink further the neighbourhoods $\varphi_i(U_i)$ to assure that whenever two such neighbourhoods $\varphi_i(U_i)$ and $\varphi_j(U_j)$ intersect, the metrics $g_i'$ and  $g_j'$ are equal to $g$ on $\varphi_i(U_i)\cap\varphi_j(U_j)$. \end{proof}

\begin{remark}\label{localKahler6} Let $M$ be a Hamiltonian $S^1$-manifold of dimension $6$, such that $M_{\min}$ and $M_{\max}$ are surfaces. Let $C_1,\ldots, C_k$ be all the isotropy spheres in $M$.  Theorem \ref{semilocalKahler} and Proposition \ref{kahleratfix} define for us a particular $S^1$-invariant K\"ahler structure in a neighbourhood of $M^{S^1}\cup C_1\cup \ldots\cup C_k$, compatible with $\omega$. We will call such a structure {\it adjusted}.

\end{remark}

The proof of  Theorem \ref{semilocal4orb} is identical to the proof of Theorem \ref{semilocalKahler}, the only difference is that it uses  Proposition \ref{kahlerdivisor} instead of Proposition \ref{toricKahlernei}.

\begin{proof}[Proof of Theorem  \ref{semilocal4orb}] Let $x_1,\ldots, x_m$ be all the maximal orbi-points of $M^4$ and $U=U_{\bar x}$ be a neighbourhood of $\bar x=\{x_1,\ldots, x_m\}$ with a $T^2$-invariant  K\"ahler metric compatible with $\omega$. As in the proof of Theorem \ref{semilocalKahler}, using Proposition \ref{kahlerdivisor}, for each divisor $D_j$ in the orbifold locus we can find an $S^1$-invariant K\"ahler metric on its neighbourhood, compatible with $\omega$. Moreover, using Theorem \ref{relative_local_neighbourhood} 2) we can assume that the metric is $T^2$-invariant at each 
maximal orbi-point on $D_j$. Now again as in the proof of Theorem \ref{semilocalKahler} we modify this metric using Proposition \ref{gluinginKahler} so that it coincides with the preassigned metric close to all maximal orbi-points on $D_j$.
\end{proof}

\subsection{Regularization of symplectic $4$-orbifolds }

In this  section we associate  two symplectic spaces $(\widetilde M^4,\widetilde\omega)$ and $(\overline M^4, \overline \omega)$ to each symplectic orbifold $(M^4,\omega)$ with cyclic stabilizers. The space $\widetilde M^4$ is a smoothing of $M^4$ along  $\Sigma_{\cal O}(M^4)$. $\widetilde M^4$  is a  symplectic orbifold with isolated singularities and is homeomorphic to $M^4$. The space $\overline M^4$ is a smooth symplectic manifold which can be seen as a resolution of  isolated singularities of $\widetilde M^4$. The construction of $\widetilde M^4$ and $\overline M^4$ relies on the semi-toric  K\"ahler structure  in a neighbourhood of $\Sigma_{\cal O}(M^4)$ provided by Theorem \ref{semilocal4orb} and on smoothing results from Appendix \ref{smoothingSection}.

\begin{definition} To define $(\widetilde M^4, \widetilde \omega)$, pick a K\"ahler semi-toric metric in a neighbourhood of $\Sigma_{\cal O}(M^4)$ provided by Theorem \ref{semilocal4orb}. Note  that the corresponding complex orbi-structure endows a neighbourhood of $\Sigma_{\cal O}(M^4)$ with a structure of a complex surface with quotient singularities (Lemma \ref{complexquotient}). Hence, $\widetilde M^4$ is also a smooth $4$-dimensional orbifold with isolated orbi-points. Clearly, $\widetilde M^4$ is  diffeomorphic to $M^4$ on the complement to $\Sigma_{\cal O}(M^4)$. 

Next, Theorem \ref{smoothingdisivors} allows us to smooth the singular K\"ahler structure on $\widetilde M^4$ along all the orbi-divisors of $M^4$. As a result we get a K\"ahler metric $\widetilde \omega$ that has quotient singularities at the isolated orbi-points of $\widetilde M^4$.
\end{definition}

\begin{definition}\label{resolutionDef} To define $(\overline M^4, \overline \omega)$, consider a holomorphic  resolution $\pi: \overline M^4\to \widetilde M^4$ of isolated orbi-singularities of $\widetilde M^4$ and adjust the form $\pi^*\widetilde\omega$ in a small neighbourhood of the exceptional divisors, using Lemma \ref{kahleronisolated}.
\end{definition}

\begin{remark}\label{holoneigh} For the remainder of the article it will be important for us, that the spaces $M^4$, $\widetilde M^4$ and $\overline M^4$ come with open subsets that we will call $U_o$, $\widetilde U_o$, and $\overline U_o$ that all have a complex structure. With respect to this complex structure $(U_o,\omega)$ is a K\"ahler orbifold, $(\widetilde U_o,\widetilde\omega)$ is a K\"ahler orbifold with isolated orbi-singularities and $(\overline U_o,\overline \omega)$ is a smooth K\"ahler surface. We have a map $i: \widetilde M\to M$ that restricts to a biholomorphism on $\widetilde U_o$ and a map $\pi: \overline  M^4\to \widetilde M^4$ that is a holomorphic contraction on $\overline U_o$.

\end{remark}

\section{Orbi-spheres in symplectic $4$-orbifolds}\label{orbispheresection}

In this section we will prove Theorem \ref{twoinoneOrbisphere}. For convenience this theorem will be split into two results, Theorem \ref{symplorbispheres} and Theorem \ref{orbisotopy}. The first theorem generalises Theorem \ref{ruledcriterion} to orbifolds with cyclic stabilizers. The second theorem generalises a classical result on isotopy of symplectic spheres with zero self-intersection.

\subsection{Existence of orbi-spheres in symplectic $4$-orbifolds}

Once we know how to get a smooth symplectic manifold $\overline M^4$ from a symplectic orbifold $M^4$ with cyclic stabilizers (see Definition \ref{resolutionDef}), we can generalize Theorem \ref{ruledcriterion}, namely we have.
 
\begin{theorem}\label{symplorbispheres} Let $M^4$ be a $4$-dimensional symplectic orbifold with cyclic stabilizers and with $\pi_1(M^4)\ne 0$. Let $F$ be a smooth sub-orbifold in $M^4$ transversal to $\Sigma_{\cal O}(M^4)$ and satisfying the following: 

i) $F$ is a two-sphere, ii) $F\cdot F=0$, iii) $\int_F\omega>0$. 

Then the following statements hold.
\begin{enumerate}
\item The desingularisation $\overline M^4$ of $M^4$ is an irrational ruled manifold.

\item $M^4$ contains a symplectic sub-orbifold sphere $F'$ in the same homology class as $F$, that is transversal to $\Sigma_{\cal O}(M^4)$.
\end{enumerate}

\end{theorem}

\begin{proof} 1) Recall that there is an open neighbourhood $U_o$ of $\Sigma_{\cal O}(M^4)$ in $M^4$ with a holomorphic structure (see Remark \ref{holoneigh}). We can perturb $F$ slightly  in an even smaller neighbourhood  of  $\Sigma_{\cal O}(M^4)$ to make it holomorphic and smooth there with respect to the holomorphic structure of $U_o$. Call the perturbed sphere $\widetilde F$, it is clearly smooth in $\widetilde M^4$ by construction of the smooth structure on $\widetilde M^4$. In addition, $\widetilde F\cdot \widetilde F=0$ and by construction of $\widetilde \omega$ (see Theorem \ref{smoothingdisivors} 3)) we have $\int_{\widetilde F}\widetilde \omega=\int_F\omega>0$. Denote now by $\overline F$ the preimage $\pi^{-1}(\widetilde F)$. We have again $\int_{\overline F}\overline \omega=\int_F\omega>0$ by construction of $\overline \omega$, and again $\overline F\cdot \overline F=0$. Applying Theorem \ref{ruledcriterion}, we see that $\overline M^4$ is rational or irrational ruled. At the same time $\pi_1(\overline M^4)\cong \pi_1 (M^4)\ne 0$, since the singularities of $M^4$ are quotient singularities. We deduce that $\overline M^4$ is irrational ruled.

2) Choose an almost complex structure $\overline J$ on $\overline M^4$ that is compatible with $\overline \omega$ and that coincides with the integrable complex structure defined  on $\overline U_o$. By Zhang's Theorem  \ref{zhang} we have a singular fibration of almost complex curves on $(\overline M^4, \overline J)$ such that each fibre is in the homology class $[\overline F]$ and only a finite number of fibres are reducible. All the irreducible fibres represent spheres symplectically embedded in $\overline M^4$. Note that the intersection of any such fibre with $\overline U_o$ is holomorphic. It follows that irreducible fibres do not intersect exceptional curves contracted by the map $\pi: \overline U_o\to \widetilde U_o$. 

From Corollary \ref{finitangence} it follows that all but a finite  number of irreducible  $\overline J$-holomorphic curves in class $[\overline F]$ are transversal to $(i\circ \pi)^{-1}(\Sigma_{\cal O}(M^4))$.  Take such a curve $\overline F'$ and then the sub-orbifold $F'$ in $M^4$ can be given by $i\circ\pi(\overline F')$.
\end{proof}

We will need later a slightly stronger technical version of Theorem \ref{symplorbispheres}, that follows immediately from the proof of  Theorem \ref{symplorbispheres}.

\begin{corollary}\label{adoptedweiyi} Let $(M^4, \omega)$, $F$ be as in Theorem \ref{symplorbispheres}. Suppose $J$ is an almost complex orbi-structure integrable close to $\Sigma_{\cal O}(M^4)$, tamed by $\omega$, and such that $(J,\omega)$ defines a semi-toric K\"ahler metric close to $\Sigma_{\cal O}(M^4)$. Then there is a finite collection of almost-complex spheres in $M^4$ such that any point of their complement is contained in a smooth almost complex orbi-sphere homologous to $F$.
\end{corollary}

\subsection{Isotopy of orbi-spheres in symplectic $4$-orbifolds}

The next result generalises to orbifolds the classical result on irrational ruled symplectic $4$-manifolds, stating that any two symplectic spheres with zero self-intersection in such a manifold are symplectically isotopic.

\begin{theorem}\label{orbisotopy} Let $M$ be a symplectic orbifold with cyclic stabilizers and with $\pi_1(M)\ne 0$. Suppose that  $M$ contains two symplectic sub-orbifold spheres $F_0$ and $F_1$ that are transversal to the orbifold locus  and satisfy $F_0^2=F_1^2=0$. Then there is a symplectic isotopy from $F_0$ to $F_1$ in the class of symplectic sub-orbifolds transversal to the orbifold locus.
\end{theorem}

The proof of this theorem will use the same idea as in Theorem \ref{symplorbispheres} and will rely additionally on the following two statements. The first result follows from {\it automatic transversality} (see \cite{HLS} and \cite[Theorem 4.5]{Wen}).

\begin{theorem}\label{smoothdef} Let $(M,\omega)$ be an irrational ruled surface and $J_t$ be a smooth family of compatible almost complex structures, $t\in [0,1]$. Suppose that for some $t_0\in (0,1)$, $\Sigma_{t_0}$ is a smooth almost complex sphere in $(M,\omega, J_{t_0})$ with $\Sigma_{t_0}^2=0$. Then there is a small open neighbourhood $U(\Sigma_{t_0})\subset M\times (0,1)$ and a diffeomorphism $\varphi: B^3\times S^2\to U(\Sigma_{t_0})$ such that for any $x\in B^3$ the image $\varphi(x, S^2)$ lies in some fibre $M\times t'$ and is $J_{t'}$-holomorphic. 
\end{theorem}

\begin{corollary}\label{transisotopy} Let $(M,\omega)$ be an irrational ruled surface, $\Sigma_1,\ldots, \Sigma_n\subset M$ be a collection of symplectic surfaces and $S_0, S_1$ be two symplectic spheres in $M$, transversal to $\Sigma_i$ and satisfying $S_0^2=S_1^2=0$. Let $J_t$ be a smooth family of $\omega$-compatible almost complex structures, $t\in [0,1]$ such that 1) $\Sigma_i$ are $J_t$-holomorphic for all $t$, 2) $S_0$ is $J_0$-holomorphic, 3) $S_1$ is $J_1$-holomorphic. 

Then there is a smooth isotopy $S_t$, from $S_0$ to $S_1$ in $M$, such that  $S_t$ is a smooth $J_t$-holomorphic sphere  transversal to all $\Sigma_i$ for any $t\in [0,1]$.

\end{corollary}

\begin{proof} Let $U$ be the subset of $M\times [0,1]$ whose intersection with each fibre $M\times t$ is the union of all smooth $J_t$-holomorphic spheres in $M$ that are in homology class $\cal F$ and  transversal to $\Sigma_i$.
It follows directly from Theorem \ref{smoothdef} that $U$ is an open submanifold in $M\times [0,1]$. $U$ is connected, since by Theorem \ref{zhang} 4) and Corollary \ref{finitangence} for each $t$ the intersection $U\cap (M,t)$ is a complement in $M$ to a finite number of $J_t$-almost complex curves. It follows easily that there exists a smooth path $\psi: [0,1]\to M$ such that the path $(\psi(t),t)$ lies in $U$. Let now $S_t$ be  the unique $J_t$-holomorphic sphere in the class $\cal F$ containing the point $\psi(t)$ in $M$. It follows from Theorem \ref{smoothdef} that the family of spheres $S_t$ provides a smooth symplectic isotopy between $S_0$ and $S_1$.
\end{proof} 

We are now ready to prove Theorem \ref{orbisotopy}.

\begin{proof}[Proof of Theorem \ref{orbisotopy}] By perturbing slightly $F_1$ we may assure that the intersections $F_0\cap \Sigma_{\mathcal O}(M)$ and $F_1\cap \Sigma_{\cal O}(M)$ are disjoint. In such a case, we can introduce a semi-toric structure on a small neighbourhood $U_o$ of $\Sigma_{\cal O}(M)$ so that both $F_0$ and $F_1$ are holomorphic in $U_o$. Take now the regularisation $(\overline M, \overline J)$ of $(M,J)$ and extend the holomorphic structure $\overline J$ from $\overline U$ to the whole $\overline M$ in the following two ways. We choose $\overline J_0$ and $\overline J_1$ compatible with $\overline \omega$ to be equal to $\overline J$ on $\overline U_o$ and such that $\overline F_0$ is  $\overline J_0$-almost complex in $\overline M$ while $\overline F_1$ is  $\overline J_1$-almost complex in $\overline M$. 

Let $\overline J_t$, $t\in [0,1]$ be a smooth family of $\overline \omega$-compatible almost complex structures  that connects $\overline J_0$ with $\overline J_1$ restricting to $\overline J$ on $\overline U_o$. By Corollary \ref{transisotopy} there is a smooth family of spheres $\overline F_t$ such that each $\overline F_t$ is $\overline J_t$-holomorphic and transversal to the preimage of $\Sigma_{\cal O}(M)$ in $\overline M$. Clearly, spheres $\overline F_t$ are disjoint from the exceptional curves of the map $\pi: \overline M\to \widetilde M$. Hence,  the family of orbi-spheres $F_t=i\circ\pi(\overline F_t)$ in $M$ provides the desired isotopy from $F_0$ to $F_1$. 
\end{proof}

\section{Symplectic spheres in reduced spaces}\label{sympspheresec}

In this section we work with Hamiltonian $S^1$-manifold of dimension $6$ such that $M_{\min}$ is a surface of positive genus. Our goal is to prove Theorem \ref{fibreinregular} and its analogue for critical values of $H$, namely Theorem \ref{critfibres}.

\begin{theorem}\label{fibreinregular} Let $M$ be a Hamiltonian $S^1$-manifold of dimension $6$ such that $M_{\min}$ is a surface of positive genus. Let $c$ be a regular value of the Hamiltonian. Then in the reduced space $M_c$ there is a symplectic sub-orbifold $2$-sphere $F$ transversal to $\Sigma_{\cal O}(M_c)$  and satisfying $F\cdot F=0$.  

\end{theorem}

\subsection{Proof of Theorem \ref{fibreinregular}}

\begin{proof}   Note, that the statement holds for $c$ close to $H_{\min}$, so the proof can be split into the following two statements.

{\bf i)} Let $c_1$ and $c_2$ be two regular values of $H$ such that $c_1<c_2$. Suppose that all values of $H$ in $[c_1,c_2]$ are regular. If the statement of the theorem holds for $M_{c_1}$ then it holds for $M_{c_2}$.

{\bf ii)} If $c$ is a critical value of $H$ and the statement of the theorem holds for $M_{c-\varepsilon}$ then it holds for $M_{c+\varepsilon}$.

To see that {\bf i)} holds, note that the gradient flow produces a diffeomorphism between smooth orbifolds $M_{c_1}$ and $M_{c_2}$. Let $F_{c_1}$ be a symplectic orbi-sphere in $M_{c_1}$  with $[F_{c_1}]={\cal F}_{c_1}$. $F_{c_1}$ is sent by the gradient flow to a smooth orbi-sphere $F_{c_2}$ in $M_{c_2}$ with $[F_{c_2}]={\cal F}_{c_2}$. By Corollary \ref{uniquefibre} 1) we have $\int_{F_{c_2}}\omega_{c_2}>0$. Hence we can apply Theorem  \ref{symplorbispheres} to produce a symplectic orbi-sphere in $M_{c_2}$ in the fibre class.

Let us now establish {\bf ii)}. We will make use of {\it adjusted} K\"ahler structure defined on  an $S^1$-invariant neighbourhood $U\subset M$ of the union of $M^{S^1}$ with all the isotropy spheres, see Remark \ref{localKahler6}. It will be useful for us to assume that $U$ has the following {\it cylindrical} property. There exists an $\varepsilon>0$ such that a point $p$ with $H(p)\in [c-\varepsilon,c+\varepsilon]$ belongs to $U$ if an only if the gradient trajectory passing through $p$ has a point of $U\cap H^{-1}(c)$ in its closure. Such a property can be always achieved by shrinking $U$ if necessary; we also assume that $c$ is the only critical value of $H$ in the interval $[c-\varepsilon, c+\varepsilon]$.

Let $\Sigma_1,\ldots,\Sigma_l$ be all the fixed surfaces in the level set $H^{-1}(c)$ and let $x_1,\ldots, x_m$ be all the isolated fixed points in this level set. 

It follows from Theorem \ref{analyticstr} that the K\"ahler structure on $U$ induces a  complex structure on the reduced neighbourhood $U_{c+\delta}$ for any $\delta\in [-\varepsilon, \varepsilon]$. It follows from Theorem \ref{bimeromorphic} that the map $gr^{c+\delta}_{c}: U_{c+\delta}\to U_c$ is a regular bi-meromorphic map, invertible close to $\Sigma_1,\ldots,\Sigma_l\subset U_c$. Hence the preimage $\Sigma_{i,\delta}$ of $\Sigma_i$ in $M_{c+\delta}$ under the map $gr^{c+\delta}_{c}$ is a smooth complex curve in $U_{c+\delta}$. Note as well, that the preimage of $x_j$ in $U_{c+\delta}$ under the map $gr^{c+\delta}_{c}: U_{c+\delta}\to U_c$ is either a point or a $\mathbb CP^1$.

Let us now apply Corollary \ref{adoptedweiyi} to $M_{c-\varepsilon}$. For this, we should choose an appropriate almost complex structure $J_{-\varepsilon}$ on $M_{c-\varepsilon}$, semi-toric at the orbifold locus of $M_{c-\varepsilon}$. Note that the reduced K\"ahler metric on $U_{c-\varepsilon}$ is already semi-toric thanks to additional symmetries coming from Remark \ref{localKahler6}. We only need to extend this K\"ahler metric to a neighbourhood of the whole orbifold locus of $M_{c-\varepsilon}$. To do this we keep the metric intact on all connected components of $U_{c-\varepsilon}$  containing surfaces $\Sigma_{i,-\varepsilon}$. As for components  of $U_{c-\varepsilon}$ containing the preimages   of points $p_i$, after shrinking them if necessary, we use Theorem \ref{semilocal4orb} to extend this semi-toric K\"ahler  structure to a neighbourhood of the whole orbi-locus of $M_{c-\varepsilon}$. Finally, we extend this (integrable) almost complex structure from a neighbourhood of the orbifold locus to an almost complex structure on whole $M_{c-\varepsilon}$. In exactly the same way we define an almost complex structure $J_{\varepsilon}$ on $M_{c+\varepsilon}$.

We apply now Corollary \ref{adoptedweiyi} to $(M_{c-\varepsilon}, J_{-\varepsilon})$ to get  a smooth $J_{-\varepsilon}$-holomorphic orbi-sphere $F_{c-\varepsilon}$  in the fibre class of $M_{c-\varepsilon}$ that does not pass through maximal orbi-points in $M_{c-\varepsilon}$.  Note that the map  
$$(gr_{c}^{c+\varepsilon})^{-1}\circ gr_{c}^{c-\varepsilon}:M_{c-\varepsilon}\to M_{c+\varepsilon}$$
is bimeromorphic on $U_{c-\varepsilon}$ and regular along the union of $\Sigma_{i,-\varepsilon}$. Since $F_{c-\varepsilon}$  is  disjoint from holomorphic exceptional curves in $M_{c-\varepsilon}$ contracted by the gradient map $gr_{c}^{c-\varepsilon}$, the above map is smooth along $F_{c-\varepsilon}$. We deduce that the image of $F_{c-\varepsilon}$ in $M_{c+\varepsilon}$ under this map is a smooth orbi-sphere. Hence we can again apply Theorem \ref{symplorbispheres} to produce a symplectic orbi-sphere in $M_{c+\varepsilon}$ in the fibre class. \end{proof}

\subsection{Symplectic fibres at critical levels}

Our second task is to understand what should play the role of symplectic fibre in $M_c$ for a critical value $c$ of $H$, and prove an analogue of Theorem \ref{fibreinregular} for critical levels. We choose again an adjusted K\"ahler structure on $M$ close to the union of $M^{S^1}$ with all the isotropy spheres. As in the previous section, we denote by $\Sigma_1,\ldots,\Sigma_l$  all the fixed surfaces in the level set $H^{-1}(c)$ and by $x_1,\ldots, x_m$ all the isolated fixed points in this level set. The reduced space $M_c$ is a topological orbifold, though the quotient K\"ahler metric on $M_c$ close to the image of $\cup_ix_i\cup_j \Sigma_j$ in  $ M_c$ is not an orbifold metric. However, one can deduce from  Theorem \ref{analyticstr} 3), that the neighbourhood of the image of $\cup_ix_i\cup_j \Sigma_j$ in $M_c$ acquires the structure of a complex surface with quotient singularities. 

\begin{theorem}\label{critfibres} Let $M$ be a Hamiltonian $S^1$-manifold of dimension $6$ with $M_{\min}$ and $M_{\max}$ surfaces of positive genus. Let $c$ be a critical value of $H$ in the interval $(H_{\min}, H_{\max})$. Then there is a sphere $F$ in $M_c$  with $F^2=0$, satisfying the following conditions.
\begin{enumerate}
\item $F$ is a symplectic sub-orbifold transversal to $\Sigma_{\cal O}(M_c)$
in the complement to the image of $\cup_ix_i\cup_j \Sigma_j$  in $M_c$. 

\item $F$ is holomorphic close to the images in $M_c$  of $\Sigma_1,\ldots, \Sigma_l$, and is transversal to these images.

\item Each $\Sigma_{j}$ of positive genus intersects $F$. 
\end{enumerate}

\end{theorem} 

\begin{proof} Using Lemma  \ref{smoothblowup}, we can perform a small  K\"ahler $S^1$-equivariant  blow-up of $M$ at the union of points $x_i$ and surfaces $\Sigma_j$ to get a new Hamiltonian $S^1$-manifold $M'$. By Lemma \ref{Kahlerblowup}, $c$ is a regular level of $M'$ and there is a map $\varphi : M'_{c} \rightarrow M_{c}$ that is a biholomorphism over a neighbourhood of each $\Sigma_{j}$, holomorphic over a neighbourhood of each $x_{j}$ and a symplectomorphism of orbifolds elsewhere.

Applying Theorem  \ref{fibreinregular} to $M_c'$ we get a symplectic orbi-sphere $F'$ in $M_c'$ transversal to $\Sigma_{\cal O}(M_c')$. Let us take its image $\varphi(F')$ in $M_c$. It is symplectic outside of $U$ and holomorphic inside $U$. Hence it satisfies both Properties 1 and 2.
 
To prove that each $\Sigma_{j}$ of positive genus intersects $F$, we apply Corollary \ref{fibreintersection} on the symplectic resolution $\overline M'_{c}$.
\end{proof}

\subsubsection{Constructing a holomorphic slice at critical levels}

The main result of this section is Lemma \ref{holoslicelemma}, which can be summarised as follows. Let $F \subseteq M_{c}$ be the sphere constructed in Theorem \ref{critfibres}, and let $Q: H^{-1}(c) \rightarrow M_{c}$ be the quotient map. Then on a K\"ahler neighbourhood (in $M$)  of each fixed surface $\Sigma_i$, which we denote $U_i$, the union of gradient spheres that intersect $Q^{-1}(F) \cap U_i$ is a smooth complex surface in $U_i$, which we refer to as a holomorphic slice. These slices will be used in the proof of Lemma \ref{stripexisence}, to establish the existence of the symplectic fibre.

We start with the following self-evident lemma.
\begin{lemma} \label{equihol} Let $L_1\oplus L_2$ be a sum of two line bundles over a curve $C$. Suppose we have a $\mathbb C^*$-action on the total space of weight $p>0$ on $L_1$ and of weight $-q<0$ on $L_2$. Then the natural (categorical) quotient map $\psi: L_1\oplus L_2\to L_1^{q}\otimes L_2^p$ sends $\mathbb C^*$-orbits to points.
\end{lemma} 

We would now like to apply Theorem \ref{analyticstr} to a neighbourhood of the zero-section in $L_{1} \oplus L_{2}$. In Proposition \ref{kahleratfix} we constructed the K\"ahler neighbourhood of the zero section in $L_{1} \oplus L_{2}$, so that it admits an equivariant K\"ahler embedding into the projectivised bundle $\mathbb{P}(L_{1} \oplus L_{2} \oplus \mathcal{O})$.

\begin{lemma}\label{complexfixed}
Consider the compatible K\"ahler form $\omega$ on a neighbourhood of the zero section in $L_{1} \oplus L_{2}$ constructed in Proposition \ref{kahleratfix}, and let $H$ be the corresponding Hamiltonian (normalised so that the zero-section is on level $0$). Then  $H^{-1}(0)/S^{1} $ may be identified  biholomorphically with a neighbourhood of the zero section in $L_{1}^{q} \otimes L_{2}^{p}$.
\end{lemma}
\begin{proof}
On a neighbourhood of the zero section, the closure of every $\C^{*}$-orbit intersects the level $H = 0$. Hence, by the second statement of Theorem \ref{analyticstr}, $H^{-1}(0)/S^{1}$ may be identified holomorphically with $(L_{1} \oplus L_{2})//\C^{*} \cong L_{1}^{q} \otimes L_{2}^{p}$. 
\end{proof}

\begin{lemma}\label{holoslicelemma} Let $C$ be a complex curve and $L_1$, $L_2$ be two line bundles on it. Consider a $\mathbb C^*$-action on the total space of $L_1\oplus L_2$ such that the weight on $L_1$ is $p>0$ and the weight on $L_2$ is $-q<0$. Let $\omega$ be a K\"ahler form defined in an $S^1$-invariant neighbourhood $U$ of the zero section and $H$ be the corresponding Hamiltonian. Let $U_{0}$  denote $(U \cap H^{-1}(0))/S^{1}$ which is naturally a complex manifold by Corollary \ref{complexfixed}. Let $D\subset U_0$ be a holomorphic disk transversal to $C\subset U_0$. Let finally $V\subset U$ be the union of the $\mathbb C^*$-orbits which have closure intersecting $H^{-1}(0)$ in a point that projects to $D\subset U_0$. Then $V$
is a smooth holomorphic submanifold in $U$.
\end{lemma}
\begin{proof} Consider the map $\psi : L_{1} \oplus L_{2} \rightarrow  L_{1}^{q} \otimes L_{2}^{p}$ from Lemma \ref{equihol}. In a local trivialisation $(x,z_{1},z_{2})$ we may write $$\psi(x,z_{1},z_{2}) = (x,z_{1}^{p}z_{2}^{q}).$$ 
Since $D$ is transversal to the zero-section, in a local trivialisation $(x,z)$ of $L_{1}^{q} \otimes L_{2}^{p}$ we may write $D = \{f = 0 \}$, for some holomorphic $f : \C^{2} \rightarrow \C$  with $\frac{\partial f}{\partial x} \neq 0$. To prove the result note that $V = \{  \psi^{*}f = 0 \}$ and $\frac{\partial ( \psi^{*}f)}{\partial x} \neq 0$, hence $V$ is locally a smooth hypersurface.
\end{proof}

\section{The symplectic fibre}\label{sympfibresec}

In this section we construct the {\it symplectic fibre}.

\begin{definition} Let $M$ be a Hamiltonian $S^1$-manifold of dimension $6$ with $M_{\min}$ and $M_{\max}$ surfaces of positive genus. A connected, $S^1$-invariant, symplectic submanifold ${\cal F}(M)\subset M$ is called a {\it symplectic fibre} if it is $4$-dimensional and intersects $M^{S^1}$ transversally. 
\end{definition}

We prove here the existence part of the following theorem.

\begin{theorem}\label{fibreconsturction} Let $M$ be a Hamiltonian $S^1$-manifold of dimension $6$ with $M_{\min}$ and $M_{\max}$ surfaces of positive genus. Then the symplectic fibre ${\cal F}(M)\subset M$ exists. Moreover, any two symplectic fibres in $M$ are $S^1$-equivariantly symplectomorphic.

\end{theorem}

\begin{remark} Theorem \ref{fibreconsturction}  a slightly weaker version of Theorem \ref{symfibmain}, but this is the result that we need to prove Theorem \ref{theo} and  Theorem \ref{maintheorem}. Note however, that the case of  Theorem \ref{symfibmain} when $\dim(M_{\min})=2$ and $\dim(M_{\max})=4$  can be treated in an almost identical way. One  needs to observe that in the case when $M_{\min}$ is a genus $g$ surface and $M_{\max}$ has dimension $4$, $M_{\max}$ is an irrational ruled surface over a genus $g$ surface. Then all the steps of the proof can be repeated. For this reason we don't present the full proof of Theorem \ref{symfibmain} in this article.
\end{remark}

\subsection{Basic properties of symplectic fibres}

Let us first establish some simple properties of the symplectic fibre.
 
\begin{lemma} \label{reduceduniq} Let $M$ be as in Theorem \ref{fibreconsturction}. A symplectic fibre ${\cal F}(M) \subset M$ intersects both $M_{\min}$ and $M_{\max}$ in a unique point. For each $c\in (H_{\min}, H_{\max})$ the trace ${{\cal F}(M)}_c\subset M_c$ is a $2$-sphere with $[{{\cal F}(M)}_c]={\cal F}_c\in H_2(M_c)$.
\end{lemma}
\begin{proof}
Let $x\in {\cal F}(M)$ be a point where $H$ attains its minimum and let $C$ be the connected component of $M^{S^1}$ that contains $x$. Since ${\cal F}(M)$ is transversal to  $C$ at $x$ and $\mathcal{F}(M)$ is $4$-dimensional, we see that $C = M_{\min}$. Since all level sets of $H$ on ${\cal F}(M)$ are connected, we see that ${\cal F}(M)$ intersects $M_{\min}$ in a unique point. 

It is clear now that for any $c\in (H_{\min}, H_{\max})$ the surface ${{\cal F}(M)}_c$ is a sphere in $M_c$  whose cohomology class satisfies all the properties of Corollary \ref{uniquefibre}. It follows that  $[{{\cal F}(M)}_c]={\cal F}_c$.
\end{proof}

\begin{proposition}\label{fibercharacterisation} Let $M$ be as in Theorem \ref{fibreconsturction} and let $N\subset M$ be a smooth, $4$-dimensional, $S^1$-invariant submanifold. $N$ is a symplectic fibre if and only if it satisfies the following properties. 
\begin{enumerate}

\item $N$ intersects $M^{S^1}$ transversally.
\item  The set of critical points of the restriction $H|_N$ coincides with $N\cap M^{S^1}$.
\item   In the complement $M\setminus M^{S^1}$ the intersection of $N$ with each level set $H=c$ is pre-symplectic.
\end{enumerate}  
\end{proposition}
\begin{proof}{\it Only if.} Suppose that $N$ is a symplectic fibre. Then property 1) holds by definition. Since $N$ is symplectic, each critical points of $H$ on $N$ is fixed by $S^1$ and so belongs to $M^{S^1}$. Property 3) holds since $N$ is  symplectic.

{\it If.} Let us assume that properties 1)-3) hold and  deduce that $N$ is a symplectic submanifold of $M$. Consider two cases.

i) Let $x\in N$ be a point in the complement $M\setminus M^{S^1}$ and let $H(x)=c$. Then the level set $H^{-1}(c)$ is a submanifold in a neighbourhood of $x$. Denote by $T_xH^{-1}(c)$ the ($5$-dimensional) tangent space to $H^{-1}(c)$ at $x$ and let $v_x\in T_xH^{-1}(c)$ be a vector tangent to the $S^1$-orbit through $x$. Note that $\omega^{\perp}(v_x)=T_xH^{-1}(c)$, and so for any non-zero vector $v'\in T_xM$ that doesn't lie in $T_xH^{-1}(c)$ we have $\omega(v_x,v')\ne 0$. In particular, since $H|_N$ is non-critical at $x$, it follows that $T_xN$ is transversal to $T_xH^{-1}(c)$, and we can choose a vector $v\in T_xN$, such that $\omega(v_x,v)\ne 0$.

Assume now by contradiction that $\omega$ is degenerate on $T_xN$. Then  the restriction of $\omega$ to  $T_xN$ has a $2$-dimensional kernel, which we denote by $K_{\omega}$. Since $N\cap H^{-1}(c)$ is pre-symplectic at $x$, the kernel of $\omega$ restricted to $T_xN\cap T_xH^{-1}(c)$ is generated by $v_x$. And so, the intersection of $K_{\omega}$ with $T_xH^{-1}(c)$ must contain $v_x$. But this is impossible since $\omega(v_x,v)\ne 0$. 

ii) Let now $x\in N$ be a point that lies on a surface $\Sigma\subset M$ fixed by the $S^1$-action. It follows from simple linear algebra that there is a unique $4$-dimensional $S^1$-invariant subspace in $T_xM$ transversal to $T_x\Sigma$, namely $\omega^{\perp}(T_x\Sigma)$. Hence, $T_xN$ coincides with this space and therefore is a symplectic subspace of $T_xM$.
\end{proof}

The following statement is an immediate corollary of Proposition \ref{fibercharacterisation}. We will use it to prove that the symplectic fibre exists.
 
\begin{corollary}
Let $N$ be an $S^1$-invariant smooth $4$-dimensional submanifold of $M$ transversal to $M^{S^1}$. Suppose that for any  $c\in (H_{\min}, H_{\max})$ the surface $N_c\subset M_c$ is a $2$-sphere in the fibre class and is a symplectic sub-orbifold at all points in $(M\setminus M^{S^1})_c$. Then $N$ is a symplectic fibre in $M$. 
\end{corollary} 
\begin{proof} We only need to show that property 3) of Proposition  \ref{fibercharacterisation} holds. This follows from Lemma \ref{presymth}.
\end{proof}

\subsection{Proof of Theorem \ref{fibreconsturction}: Existence}
We will construct the symplectic fibre  by patching together a finite number of strips. 

\begin{definition} Let $M$ be as in Theorem \ref{fibreconsturction} and choose $a,b\in [M_{\min}, M_{\max}]$ with $a<b$. A symplectic $S^1$-invariant submanifold  $N_a^b\subset M$ with boundary $\partial N_a\cup \partial N_b$ is called an {\it $ab$-strip} if the following holds
\begin{itemize}
\item $N$ is transversal to $M^{S^1}$. 
\item $H$ is constant on $\partial N_a$ and on $\partial N_b$ and $H(\partial N_a)=a<b=H(\partial N_b)$. 
\item For any $c\in [a,b]\cap (H_{\min}, H_{\max})$ the trace $N_c\subset M_c$ is a sphere in homology class ${\cal F}_c$.
\end{itemize}
\end{definition}
By a slight abuse of notations in this definition we allow $a=H_{\min}$ (or $b=H_{\max}$), in which case the strip has only one boundary component. 
\begin{lemma}\label{stripexisence} Let $M$ be as in Theorem \ref{fibreconsturction}, then for any $c\in (H_{\min}, H_{\max})$ there exists an $ab$-strip in $M$ with $a<c<b$.
\end{lemma} 

\begin{proof} Suppose first that $c$ is not a critical value of $H$. Take in $M_c$ a sub-orbifold sphere $F$ given by Theorem \ref{symplorbispheres} and let $N^3$ be the pre-symplectic $3$-dimensional manifold in $H^{-1}(c)$ that projects to $F$. Consider the union of all gradient lines in $M$ that intersect $N^3$. This space is an open $4$-dimensional submanifold of $M$ and it is symplectic close to $N^3$. It is easy to see that for $\varepsilon $ small enough the intersection of this space with $H^{-1}([c-\varepsilon, c+\varepsilon])$ is a $(c-\varepsilon, c+\varepsilon)$-strip.

Suppose now $c$ is a critical value. If $c=H_{\min}$ (or $c=H_{\max}$) we can take as a strip a small $4$-dimensional $S^1$-invariant symplectic disk transversal to $M_{\min}$ (respectively to $M_{\max}$).

If $c$ is a non-extremal critical value, by  Theorem \ref{critfibres} there exists a sphere $F\subseteq M_{c}$ and a K\"ahler neighbourhood $U$ of the fixed set, such that $F$ is a symplectic sub-orbifold on the complement of $U$ and a holomorphic curve on $U$, which is transversal to the fixed set. Take now the union of all gradient spheres in $M$ whose trace in $M_c$ belongs to $F$. Call the intersection of this set with  $H^{-1}(c -\varepsilon,c+\varepsilon)$ by $N$. Over $F \cap (M_{c} \setminus U)$, $N$ satisfies all the required properties to be a strip by the arguments used in the case when $c$ is regular. Over $F \cap U$ this follows from Lemma \ref{holoslicelemma}.
\end{proof}

The next lemma permits one to glue two overlapping $ab$-strips.

\begin{lemma} Let $N_a^c$ and $N_b^d$ be two strips in $M$ with $a<b<c<d$. Then there exists an $ad$-strip $N_a^d$ in $M$.
\end{lemma}

\begin{proof}\label{gluingstrips} Let $e\in (b,c)$ be any regular value of $H$. By  applying small symplectic isotopies supported in a neighbourhood of the level set $H=e$  to $N_a^c$ and $N_b^d$, we may assure that intersections of $N_a^c$ and $N_b^d$ with $H^{-1}[e, e+\varepsilon]$ are tangent to $\nabla H$. 

Denote by $N_0$ and $N_1$ the intersections of  the strips $N_a^c$ and $N_b^d$ with the level set $H=e$. Let $\Sigma_0$ and $\Sigma_1$ be the projections of  $N_0$ and $N_1$ to $M_e$. These surfaces in $M_{e}$ are orbifold spheres  in the class ${\cal F}_e$ and are transversal to the orbifold locus $\Sigma_{\cal O}M_e$. Hence, we can apply Theorem \ref{orbisotopy} to get a smooth symplectic isotopy $\Sigma_t$ in $M_e$ between $\Sigma_0$ and $\Sigma_1$, $t\in [0,1]$. We may assume that  $\Sigma_t=\Sigma_0$ for small $t$ and that $\Sigma_t=\Sigma_1$ for $t$ close to $1$. Finally, let $N_t$ be the corresponding isotopy of pre-symplectic manifolds in the level set $H=e$.

For any $e'\in [e, e+\varepsilon]$ let us denote by $gr_e^{e'}$ the map from the level set $H=e$ to $H=e'$ given by the normalized gradient flow. We will construct now the $ad$-strip $N_a^d$ by gluing it from three strips. The first strip is the intersection of $N_a^c$ with the subset $H\le e$ of $M$.  The third strip is the intersection of $N_b^d$ with the subset $H\ge e+\varepsilon$ of $M$. To get the second strip that connects these two, consider the subset of $M$ whose intersection with the level set $H=e+t\varepsilon$ is $gr_e^{e+\varepsilon t}(N_t)$. By decreasing $\varepsilon$ if necessary we can insure that submanifolds $gr_e^{e+\varepsilon t}(N_t)$ are pre-symplectic for all $t\in [0,1]$. It is clear that these three strips glue together to one $ad$-strip.
\end{proof}

Now, the proof of existence part of Theorem \ref{fibreconsturction} is clear.
\begin{proof}[Proof of Theorem \ref{fibreconsturction}, existence] Using Lemma \ref{stripexisence} we can find a cover of the interval $[H_{\min}, H_{\max}]$ by open sub-intervals $(a,b)$ for each of which there is an $ab$-strip. Choose a finite sub-cover and  then, using Lemma \ref{gluingstrips}, glue these strips consecutively to get a $H_{\min}H_{\max}$-strip,  i.e., a symplectic fibre.    
\end{proof}

The uniqueness part of Theorem \ref{fibreconsturction} will be established in Section \ref{uniquesec}.

\subsection{Symplectic fibres of relative symplectic Fano $6$-manifolds}

\begin{proposition}\label{fibredelpez}
Let $(M,\omega)$ be a relative symplectic Fano $6$-manifold with a Hamiltonian $S^{1}$-action such that $M_{\min} $ and $ M_{\max} $ are surfaces of genus $g>0$. Then the symplectic fibre $\mathcal{F}(M)$ is symplectomorphic to a toric del Pezzo surface.
\end{proposition}
\begin{proof}
Recall that $\mathcal{F}(M)$ is a symplectic $4$-manifold with a Hamiltonian $S^{1}$-action with isolated fixed points. If $\Sigma$ is a surface in $M$ with weights $\{w_{1},w_{2},0\}$ then $H(\Sigma) =-w_{1}-w_{2}$. Hence for any fixed point $p \in \mathcal{F}(M)$ with weights $\{w_{1},w_{2}\}$ we have $H(p)=-w_{1}-w_{2}$. It follows from Proposition \ref{converse} that $\mathcal{F}(M)$ is a symplectic Fano manifold and hence a toric del Pezzo by \cite[Theorem 5.1]{Ka}.
\end{proof}

\subsection{Symplectic fibre and restrictions on fixed spheres}

In this section we give a first application of existence of the symplectic fibre $\mathcal{F}(M)$. We give an upper bound on the area of fixed spheres in $6$-dimensional Hamiltonian $S^1$-manifolds with $M_{\min}$ a surface of positive genus and then apply it to relative symplectic Fanos in Corollary \ref{nosphere}.

\begin{proposition}\label{boundSphereAreas} Let $M$ be a  Hamiltonian  $S^1$-manifold of dimension $6$ with  $M_{\min}$ a surface of positive genus $g$. Suppose $\Sigma\subset M^{S^1}$ is a fixed sphere and let $H(\Sigma)=c$. Let ${\cal F}_c$ be the fibre class in $H_2(M_c)$. Then we have

1) $\omega_c({\cal F}_c)\ge \omega(\Sigma)$. 2) Equality holds if and only if $[\Sigma]={\cal F}_c$.  
\end{proposition}
\begin{proof}1) Assume by contradiction $\omega_c({\cal F}_c)< \omega(\Sigma)$. Then for $\varepsilon$ small enough there is a symplectic orbi-sphere $\Sigma_{\varepsilon}\subset M_{c-\varepsilon}$  that is sent by the gradient flow to $\Sigma$. For small $\varepsilon$  we still have $\omega_c({\cal F}_{c-\varepsilon})< \omega_{c-\varepsilon}(\Sigma_{\varepsilon})$ and the reduced level set $M_{c-\varepsilon}$ is an orbifold. Hence we can take the symplectic resolution $\overline M_{c-\varepsilon}$ which is an irrational ruled surface. Since the resolution does not change the area of the fibre class and the area of $\Sigma_{\varepsilon}$, we get a contradiction with Corollary \ref{sphereareabound}.

2) Assume now $\omega_c({\cal F}_c)= \omega(\Sigma)$ and suppose by contradiction that $\Sigma$ is not the fibre class in $M_c$. Observe then that $[\Sigma]^2<0$. Indeed, this can be deduced from the following three facts: i) $[\Sigma]$ and ${\cal F}_c$ span a two-plane in $H_2(M_c)$, ii) $[\Sigma]\cdot {\cal F}_c=0$, iii) $b_+(M_c)=1$. 

Consider a symplectic $S^1$-invariant blow-up $M_{\Sigma}$  of $M$ in a small neighbourhood of $\Sigma$ (constructed in Lemma \ref{Kahlerblowup}) and let $E\subset M_{\Sigma}$ be the exceptional divisor. The trace of $E$ in the reduced space  of $M_{\Sigma}$ on level $c$ gives us a sphere of area larger than $\omega(\Sigma)$ (one uses here $[\Sigma]^2<0$). This again leads us to a contradiction with Corollary \ref{sphereareabound} as in part 1).
\end{proof}

\begin{corollary}\label{nosphere} Let $(M,\omega)$ be a relative symplectic Fano $6$-manifold with a Hamiltonian $S^{1}$-action. Suppose that $M_{\min}$ is a surface of positive genus $g$. If $\Sigma \subset M$ is a fixed sphere then we have $H(\Sigma) \geq 0$.
\end{corollary}

\begin{proof} Let $a,b$ be the weights of $S^1$-action on $M_{\min}$ so that $a,b\ge 1$. Suppose first that $a+b \geq 3$. By the weight sum formula \ref{relcho} (\ref{weightcho}) $H_{\min}=-a-b$.

Let ${\cal F}(M)$ be the symplectic fibre of $M$.   
Applying Lemma \ref{DuisHeckSphere} 2) to ${\cal F}(M)$ at the level $H=-c$ we get the following inequality:
$$\omega({\cal F}_{-c})\le \frac{a+b-c}{a\cdot b}.$$ 

Let us assume by contradiction that there is a fixed sphere $\Sigma$ in the level $H=-c$. We know that $\int_{\Sigma}\omega$ is a positive integer and moreover from Proposition \ref{boundSphereAreas} 1) we have $\int_{\Sigma}\omega\le \frac{a+b-c}{a\cdot b}$. So, unless $a=1$ or $b=1$ and $\int_{\Sigma}\omega=1=c$  we get a contradiction. 
 
Suppose first $a=1$, $b>1$, $c=1$, and  $\int_{\Sigma}\omega=1$. This means that $H(\Sigma)=-1$ and we will denote by $\Sigma$ as well the image of $\Sigma$ in $M_{-1}$. We will show that $\Sigma$ is not in the fibre class of $M_{-1}$. The contradiction then will follow from Proposition \ref{boundSphereAreas} 2). 

Denote by $N$ the isotropy submanifold of weight $b$ that contains $M_{\min}$.  Since the maximum of $H$ on $N$ is at least $H_{\min}+b=-1$, the trace $N_{-1} \subset M_{-1}$ is a certain genus $g$ surface $\Sigma_b$. Note now that $N$ and $\Sigma$ are disjoint in $M$, and so the image of $\Sigma$ in the reduced space $M_{-1}$ is disjoint from $\Sigma_b$. Clearly this means that $[\Sigma_b]\cdot \Sigma=0$. Finally, the fibre class ${\cal F}_{-1}$ of  $M_{-1}$ has intersection $1$ with $\Sigma_b$. We see that  ${\cal F}_{-1}\ne [\Sigma]\in H_2(M_{-1})$ and this gives us a contradiction.
 
Finally, we need to show that in case $a=b=1$ there can be no fixed sphere $\Sigma$ in $M$ with $H(\Sigma)=-1$. Assume by contradiction that such a sphere exists. Since $M_{-1}$ is an $S^2$-bundle, $\Sigma$ represents the fibre class in it. It follows that ${\cal F}(M)$ does not have fixed points at level $-1$, since such a point would correspond to a fixed surface in $M$ that intersects positively the fibre class. Since the weights at the minimum of $\mathcal{F}(M)$ are $\{1,1\}$, there must be a fixed point in $\mathcal{F}(M)$ with weights $\{-1,n\}$ ($n > 0$) and since there are no  fixed points on level $-1$, by the weight sum Formula \ref{relcho} (\ref{weightcho}), we must have $n = 1$. It follows that there is a fixed surface of positive genus in $M$ on level $0$, we denote it by $\Sigma'$. 

Note now, that the weights at $\Sigma$ must be $\{-1,2\}$. Let $N_2$ be the isotropy $4$-manifold with $(N_{2})_{\min} = \Sigma$. Since the weight of $N_{2}$ is $2$, we have $H((N_{2})_{\max})\ge  H(\Sigma)+2 = 1$ by Lemma \ref{gradsphereint}. Hence, $\Sigma'$ must intersect in $M_0$ the sphere traced by $N_2$, since the latter one is in the fibre class. This is a contradiction, since $\Sigma'$ is composed of fixed points. \end{proof}

\section{Fixed surfaces of positive genus in Hamiltonian $6$-manifolds}\label{fixsurfsec}

In this section we study Hamiltonian $S^{1}$-actions on a symplectic $6$-manifold $(M,\omega)$ such that $M_{\min}$ and $M_{\max}$ are surfaces. Our main goal is to understand the fixed surfaces in $M$, both in terms of the genus and the level sets which contain them.

The main result of the section is Theorem \ref{struc}, which relates the fixed points of the symplectic fibre $\mathcal{F}(M)$ to the fixed surfaces of positive genus in $M$. In particular, this provides strong restrictions on the possible genus of fixed surfaces. We will show that the graph of fixed points (defined in Definition \ref{Ka}) associated to any symplectic fibre is the same. In conjunction with Theorem \ref{unigraph} this will imply the uniqueness of the symplectic fibre.


\subsection{The graph of fixed surfaces}
 We start by introducing the {\it graph of fixed surfaces}. This is an analogue of the graph of fixed points originally defined by Karshon \cite{Ka} for symplectic $4$-manifolds, which we recall. 

\begin{definition}\label{Ka} \cite[Section 2.1]{Ka}
Let $X$ be a $4$-dimensional symplectic manifold with a Hamiltonian $S^{1}$-action with isolated fixed points. The graph of fixed points associated to the action, which we denote by $\mathcal{Q} = (V,E)$, is defined as follows
\begin{itemize}
\item The vertices $V$ are in bijection with the isolated fixed points. We denote the vertex associated to an isolated fixed point $p$ by $v_{p} $.
\item $v_{p_{1}}$  and $v_{p_{2}}$ are joined by an edge $e \in E$ if and only if there is an isotropy sphere containing both $p_{1}$ and $p_{2}$.

\end{itemize}
\end{definition}
Define a function $H: V \rightarrow \R$ by setting $H(v_{p}) = H(p)$. An edge corresponding to a gradient sphere $S$ is labelled by the weight of $S$. We orient the edges so that $H$ is strictly increasing with respect to this orientation. In symbols, we write $e = [S_{\min},S_{\max}]$ when $e$ corresponds to a gradient sphere $S$. In \cite{Ka} it was shown that this graph determines $(X,\omega)$ up to equivariant symplectomorphism.

\begin{theorem} \label{unigraph} \cite[Theorem 4.1]{Ka}
Let $X_{1}$,$X_{2}$ be symplectic $4$-manifolds with Hamiltonian $S^{1}$-actions with isolated fixed points. Then any isomorphism of the associated graphs (respecting the Hamiltonian and labelling of edges by weights) induces an equivariant symplectomorphism $\tau : X_{1} \rightarrow X_{2}$.
\end{theorem}


Let us now adapt Definition \ref{Ka} to define a graph associated to a $6$-dimensional symplectic manifold with a Hamiltonian $S^{1}$-action.  

\begin{definition}
Let $(M,\omega)$ be a $6$-dimensional symplectic manifold with a Hamiltonian $S^{1}$-action such that $\dim(M_{\min}) = \dim(M_{\max}) = 2$. We define a graph $G = (V,E)$ associated to $M$ as follows: \begin{itemize}
\item The vertices of $\mathcal{G}$ are in bijection with the fixed surfaces in $M$. The vertex associated to a fixed surface $\Sigma$ is denoted $v_{\Sigma} $. 

\item Two vertices $v_{ \Sigma_{1} }$ and $v_{\Sigma_{2}}$ are joined by a unique edge $e \in E$ if and only if there exists an isotropy $4$-manifold $W$ such that $\Sigma_{1},\Sigma_{2} \subseteq W$.

\item Define $v_{\min},v_{\max} \in V$ to be the vertices corresponding to $M_{\min},M_{\max}$. These are called the \emph{extremal vertices}.  

\end{itemize}
\end{definition}

\textbf{Orientation and labelling associated to $\mathbf{\mathcal{G}}$}. Define a function $H: V \rightarrow \R$ by setting $H(v_{\Sigma}) = H(\Sigma)$. Note that $\mathcal{G}$ is naturally a directed graph, indeed if $\Sigma_{1},\Sigma_{2}$ are contained in an isotropy submanifold then $H(\Sigma_{1}) \neq H(\Sigma_{2})$, we direct the edge to the vertex with the higher value of $H$. In symbols, we write $e = [v_{\Sigma_{1}},v_{\Sigma_{2}}]$ when $H(\Sigma_{2}) > H(\Sigma_{1})$.

We also define a label on an edge $e =[v_{\Sigma_{1}},v_{\Sigma_{2}}]$ to be the weight of the isotropy $4$-manifold $W$ such that $\Sigma_{1},\Sigma_{2} \subseteq W$. We refer to this label as the weight of the edge $e$ or just $w(e)$. By definition we have $w(e) \geq 2$ for all $e \in E$.

If $e =[v_{\Sigma_{1}},v_{\Sigma_{2}}]$ and $w(e) = n$ then one of the weights at $\Sigma_{1}$ must be $n$ and one of the weights at $\Sigma_{2}$ must be $-n$. 

Note that since at most two isotropy $4$-manifolds can contain the same fixed surface, the number of edges incident to a vertex $v$ is at most two. We denote this number $\deg(v)$.

\begin{lemma}\label{Liapp}
Let $(M,\omega)$ be a symplectic $6$-manifold with a Hamiltonian $S^{1}$-action such that $\dim(M_{\min}) = \dim(M_{\max}) = 2$. Let $\mathcal{G}$ be the graph of fixed surfaces associated to $M$.
\begin{enumerate}

\item The genus of fixed surfaces is constant along connected components of $\mathcal{G}$.

\item Suppose that $\Sigma$ is a fixed surface with genus $g > 0$ and weights $\{w_{i}\}$, then $\deg(v_{\Sigma}) = \# \{w_{i}: |w_{i}| > 1\}$. In particular, if $\deg(v_{\Sigma}) = 1$, then $\Sigma$ must have a weight equal to $1$ or $-1$.

\end{enumerate}
\end{lemma}
\begin{proof}

1. Suppose that $e = [v_{\Sigma_{1}},v_{\Sigma_{2}}]$ is an edge in $\mathcal{G}$. Then there exists an isotropy $4$-manifold $W$ such that $\Sigma_{1},\Sigma_{2} \subseteq W$. Hence, $\Sigma_{1}$ and $\Sigma_{2}$ have the same genus by Corollary \ref{genus}.

2. Let $\{w_{i}\}$ denote the weights of the action along $\Sigma$. For each weight $w_{i}$ such that $|w_{i}| > 1$ there is an isotropy $4$-manifold containing $\Sigma$ with weight $w_{i}$. By Corollary \ref{genus} these isotropy $4$-manifolds contain two fixed surfaces of genus $g$ so correspond to edges in $\mathcal{G}$.
\end{proof}

\begin{definition}\label{ggg} \begin{enumerate}
\item Denote by $\mathcal{G}_{+}$ the sub-graph of $\cal G$ consisting of fixed surfaces of positive genus.

\item Denote by $\mathcal{G}_{g}$ the sub-graph consisting of fixed surfaces of genus $g = g(M_{\min})$.

\item Denote by $\mathcal{G}_{0}$ the sub-graph consisting of fixed surfaces of genus $0$.
\end{enumerate}
\end{definition}

\subsection{The graph of fixed points of the symplectic fibre}
Let $(M,\omega)$ be a $6$-dimensional symplectic manifold with a Hamiltonian $S^1$-action such that $M_{\min}, M_{\max}$ are surfaces of genus $g > 0$. Let ${\cal F}(M)$ be a symplectic fibre associated to $M$. Recall that $\mathcal{F}(M) \subset M$ is a $4$-dimensional symplectic submanifold inheriting a Hamiltonian $S^{1}$-action with isolated fixed points. 

Here we show an explicit relation (in most cases an isomorphism) between the graph $\mathcal{G}_{+}$ defined in the previous section and $\mathcal{Q}$, the graph of fixed points of $\mathcal{F}(M)$. In particular, we will gain an understanding of the possible genus of fixed surfaces in $M$.

\begin{definition} We will say that ${\cal F}(M)$ is {\it reflective} if there is a symplectic involution $\tau :\mathcal{F}( M ) \rightarrow \mathcal{ F}(M)$ that commutes with the $S^1$-action and induces a non-trivial permutation on ${\cal F}(M)^{S^1}$.
\end{definition}

\begin{lemma}\label{reflch}
$\mathcal{F}(M)$ is reflective if and only if its graph of fixed points $\mathcal{Q}$ has an order two endomorphism (preserving the weights and Hamiltonian).
\end{lemma}
\begin{proof}
If such an endomorphism exists, then the corresponding $S^{1}$-equivariant symplectic involution of $\mathcal{F}(M)$ is given by Theorem \ref{unigraph}. The other direction is obvious.
\end{proof}

\begin{corollary} \label{nonref}
Suppose that $\mathcal{F}(M)$ is reflexive. Then the weights at $M_{\min}$ (resp. $M_{\max}$) are $\{1,1,0\}$ (resp. $\{-1,-1,0\}$).
\end{corollary}
\begin{proof}
The existence of an order 2 endomorphism of the graph of fixed points of $\mathcal{F}(M)$ implies the weights at $M_{\min}$ must be $\{n,n,0\}$ for some $n>0$. By the effectiveness of the action we must have $n=1$.
\end{proof}

\begin{lemma} \label{strucjist}
Let $M$ be a Hamiltonian $S^1$-manifold of dimension $6$ with $M_{\min}$ and $M_{\max}$ surfaces of genus $g>0$. Suppose there exists a fixed surface $\Sigma$ in $M$ such that $\Sigma \cdot \mathcal{F}_{H(\Sigma)}=2$.  Then all non-extremal fixed surfaces of positive genus in $M$ lie in the same connected component of $\mathcal{G}$. Furthermore, each such surface is of genus $g(\Sigma)$ and has intersection 2 with the symplectic fibre.
\end{lemma}
\begin{proof}

Suppose that $\Sigma'$ is a fixed surface such that $\Sigma$ and $\Sigma'$ are contained in an isotropy submanifold $N$ and without loss of generality that $H(\Sigma') > H(\Sigma)$. By Lemma \ref{Isotropy fibre} \begin{displaymath}
\Sigma \cdot \mathcal{F}_{H(\Sigma)} = \Sigma' \cdot \mathcal{F}_{H(\Sigma')} =2.
\end{displaymath} 
Suppose $\Sigma$ has weights $\{w_{1},w_{2},0  \}$ and $\Sigma'$ has weights $\{-w_{1},w_{2}',0  \}$. Then there are two fixed points $p_{1},p_{2} \in \mathcal{F}(M)$ on level $H(\Sigma)$ with weights $\{w_{1},w_{2}\}$ and two fixed points $p_{1}',p_{2}' \in \mathcal{F}(M)$ on level $H(\Sigma')$ with weights $\{-w_{1},w_{2}'\}$. On the other hand, there are boundary divisors $S_{1}$,$S_{2} \subseteq \mathcal{F}(M)$ such that $p_{1},p_{1}' \in S_{1}$ and $p_{2},p_{2}' \in S_{2}$. Hence, there is no fixed point $q \in \mathcal{F}(M)$ with $H(q) \in (H(\Sigma),H(\Sigma'))$. This in turn implies that $M$ has no fixed surfaces of positive genus in this range (by Theorem \ref{critfibres}(3)).

Let the component of $\cal G$ containing $v_{\Sigma}$ be denoted by $A \subseteq \cal G$. Then by the above, all surfaces corresponding to vertices of $A$ have intersection $2$ with the symplectic fibre. We claim that all non-extremal surfaces of positive genus in $M$ correspond to vertices of $A$.

Let $\Sigma_{+},\Sigma_{-}$ be the fixed surfaces corresponding to the vertices of $A$ where $H$ achieves its $\max$/$\min$ respectively. By the above all fixed surfaces of positive genus in  $H^{-1}([H(\Sigma_{-}),H(\Sigma_{+})])$ correspond to vertices in $A$.

We have that $\deg(\phi|_{M_{\min}}) = 1$ since the (transversal) intersection point of $M_{\min}$ with the symplectic fibre is the unique point $\mathcal{F}(M)_{\min}$. Hence $\Sigma_{-} \neq M_{\min}$, so the weights at $\Sigma_{-}$ must be $\{-1,N,0\}$ where $N>0$. Hence, there must be two fixed points on level $H(\Sigma_{-})$ in $\mathcal{F}(M)$, both having weights $\{-1,N,0\}$. By Corollary \ref{neededcor} there is no fixed point $q  \in \mathcal{F}(M)$ such that $H(q) \in (H_{\min},H(\Sigma_{-})))$. Hence, there are no fixed surfaces of positive genus on these levels (by Theorem \ref{critfibres}(3)). Applying the same argument to $\Sigma_{+}$ we see that all non-extremal fixed surfaces of positive genus correspond to vertices of $A$. The lemma follows.
\end{proof}

\begin{theorem} \label{struc} Let $M$ be a Hamiltonian $S^1$-manifold of dimension $6$ with $M_{\min}$ and $M_{\max}$ surfaces of genus $g>0$. Let $\phi: M\to M_{\min}$ be the $S^1$-equivariant retraction constructed in Lemma \ref{retractmin}. 

Then a surface $\Sigma \subseteq M^{S^1}$ is a sphere if and only if the map $\phi|_{\Sigma}$ has degree $0$. If the genus of $\Sigma$ is positive, then $\phi|_{\Sigma}$ has degree $1$ or $2$. In addition, precisely one of the two following possibilities occur.

\begin{enumerate}

\item Suppose $M$ contains a fixed surface  $\Sigma$ such that $\Sigma \cdot \mathcal{F}_{H(\Sigma)} = 2$. Then  $M$ contains exactly $\frac{1}{2}\chi({\cal F}(M)) - 1$ non-extremal fixed surfaces of positive genus and the restriction of $\phi$ to each such surface has degree $2$. Furthermore, $\mathcal{F}(M)$ is reflexive. 

\item Otherwise, $M^{S^1}$ contains exactly $\chi({\cal F}(M))$ fixed surfaces of positive genus. Any such surface $\Sigma$ has  genus $g$ and the map $\phi: \Sigma\to M_{\min}$ has degree $1$.

\end{enumerate}
\end{theorem}

\begin{proof}
Throughout this proof we will use the fact that for any fixed surface $\Sigma$, its intersection with the symplectic fibre, $\Sigma \cdot \mathcal{F}_{H(\Sigma)}$ is equal to the degree of $\phi|_{\Sigma}$. This follows from Lemma \ref{reduceduniq}, where we proved for each $c$ that $\mathcal{F}_{c}$ is Poincar\'e dual to $\phi_{c}^{*}(A)$, $A$ being the positive generator of $H^{2}(M_{\min},\Z)$ and $\phi_{c} : M_{c} \rightarrow M_{\min}$ is the map induced by $\phi$.

 If $\Sigma \subseteq M^{S^{1}}$ is a sphere then $\phi|_{\Sigma}: \Sigma \rightarrow M_{\min}$ has degree $0$ since $S^2$ is simply connected.

 Suppose that $\Sigma$ is a fixed surface of positive genus, then $\deg (\phi|_{\Sigma}) >0$.  This follows from the third statement of Theorem \ref{critfibres}. On the other hand, $\Sigma \cdot \mathcal{F}_{H(\Sigma)} \leq 2$ since there are at most two fixed points in $H^{-1}(c) \cap \mathcal{F}(M)$ for each $c$. We conclude that for any fixed surface of positive genus we must have $1 \leq \deg (\phi|_{\Sigma}) \leq 2$. Finally, we will show that precisely one of the two conditions {\it 1.} and {\it 2.} is satisfied.


Suppose first that $\Sigma$ is a fixed surface in $M$ such that $\Sigma \cdot \mathcal{F}_{H(\Sigma)} = 2$. By Lemma \ref{strucjist} each non-extremal fixed surfaces with positive genus has intersection 2 with the symplectic fibre and is of genus $g(\Sigma)$. It follows that there are exactly $\frac{1}{2}\chi({\cal F}(M)) - 1$ such surfaces.

Furthermore it follows that on each non-extremal critical level of $\mathcal{F}(M)$, there are two fixed points with the same weights. Hence there is a pairing between non-extremal fixed points of $\mathcal{F}(M)$ preserving the Hamiltonian and the weights. This shows that the graph of fixed points of $\mathcal{F}(M)$ has an endomorphism of order $2$, preserving the Hamiltonian and the weights. By Lemma \ref{reflch}, we see that $\mathcal{F}(M)$ is reflective. 

Otherwise the restriction of $\phi$ to each non-extremal fixed surface of positive genus is 1, implying that each such surface is of genus $g = g(M_{\min})$. This is condition {\it 2}. \end{proof}

\begin{corollary}\label{genusg}
Suppose $\mathcal{F}(M)$ is not reflective. Then $\mathcal{G} = \mathcal{G}_{g} \cup \mathcal{G}_{0}$, i.e. fixed surfaces in $M$ have genus $0$ or $g$.
\end{corollary}
\begin{proof}
If $\mathcal{F}(M)$ is not reflexive then we must be in case 2 of Theorem \ref{struc}.
\end{proof}

The following example shows that in the case 1 of Theorem \ref{struc},  one cannot bound the genus of fixed surfaces in terms of the genus $g$ of the base.

\begin{example} Let $C_g$ be a curve of genus $g$ and let $C_{g'}$ be a curve that admits a (possibly ramified) double cover of $C_g$. Let $\sigma$ be the involution on $C_{g'}$ such that $C_g\cong C_{g'}/\sigma$.

Consider $M=\mathbb CP^1\times \mathbb CP^1\times C_{g'}$. Choose a $\mathbb C^*$-action on $M$ that preserves all $\mathbb CP^1\times \mathbb CP^1$ fibres and acts on them diagonally via, 
$$((z_1:w_1),(z_2:w_2))\to ((tz_1:w_1),(tz_2:w_2)).$$

Consider the involution on $M$ given by the formula 

$$\sigma_M:=((z_1:w_1),(z_2:w_2),x)\to ((z_2:w_2),(z_1:w_1),\sigma(x)).$$

Then $M/\sigma_M$ is a complex projective orbifold. The orbifold locus comes from the diagonals in the $\mathbb CP^1\times \mathbb CP^1$- fibres over fixed points of the involution $\sigma$. Making a simple blow-up of this collection of curves we obtain a smooth projective $3$-fold with a unique non-extremal fixed curve isomorphic to $C_{g'}$  whose projection to the base $C_g$ has degree $2$.  
\end{example}

We now return to our general discussion. Let $\mathcal{Q}$ be the graph of fixed points associated to the symplectic fibre $\mathcal{F}(M)$. 

\begin{corollary} \label{iso}
Let $(M,\omega)$ be a symplectic $6$-manifold with a Hamiltonian $S^{1}$-action such that $M_{\min} $ and $ M_{\max} $ are surfaces of genus $g>0$. 

\begin{enumerate}
\item Suppose that $\Sigma \cdot \mathcal{F}_{H(\Sigma)}= 1$ for all fixed surfaces of positive genus. Then there is an isomorphism $I: \mathcal{Q} \rightarrow \mathcal{G}_{g}$, preserving the Hamiltonian and labelling of edges by weights.

\item Suppose $M$ contains a fixed surface $\Sigma$ such that $\Sigma \cdot \mathcal{F}_{H(\Sigma)} = 2$. Then there exists a map $I: \mathcal{Q} \rightarrow \mathcal{G}_{+}$, preserving $H$ and the labelling of edges. Furthermore, the sub-graph of $\cal Q$ corresponding to non-extremal fixed points splits into two connected components $\mathcal{C}_{1}$ and $\mathcal{C}_{2}$. On each $\mathcal{C}_{i}$, the restriction $I : \mathcal{C}_{i} \rightarrow \mathcal{G}_{+}$ is an isomorphism onto the sub-graph of $\mathcal{G}_{+}$ consisting of non-extremal fixed surfaces.

In both cases $I$ has the property that if $\Sigma$ is a  surface of positive genus with weights $\{w_{1},w_{2},0\}$, then $I(v_{\Sigma})$ represents a fixed point in $\mathcal{F}(M)$ with weights $\{w_{1},w_{2}\}$.

\end{enumerate}
\end{corollary}
\begin{proof}
By Theorem \ref{struc}, a fixed point $p \in \mathcal{F}(M)$ is equal to the intersection of $\mathcal{F}(M)$ with a fixed surface of positive genus $\Sigma$. We simply set $I(v_{p}) = v_{\Sigma}$. The required properties of $I$ follow from Theorem \ref{struc}.
\end{proof}

\subsection{Proof of Theorem \ref{fibreconsturction}: uniqueness}\label{uniquesec}

Using the results proven so far in this section, we will now show the uniqueness of the symplectic fibre up to an equivariant symplectomorphism.

\begin{proof}[ Proof of Theorem \ref{fibreconsturction}: uniqueness] 
In both cases of Corollary \ref{iso}, the graph of fixed points of any symplectic fibre is determined up to isomorphism by the graph of fixed surfaces of $M$ (including the data of the Hamiltonian and the weights of edges). By Theorem \ref{unigraph} this uniquely determines $\mathcal{F}(M)$ up to $S^{1}$-equivariant symplectomorphism.
\end{proof}

\subsection{The relative symplectic Fano case}

Finally, using results of this section we deduce a useful restriction on the range of relative symplectic Fano manifolds with a Hamiltonian $S^{1}$-action, which contain a fixed surface with weights  $\{1,1,0\}$,$\{-1,-1,0\}$ or $\{1,-1,0\}$. This will be needed in the next section.

\begin{lemma} \label{class}
Let $(M,\omega)$ be a relative symplectic Fano $6$-manifold with a Hamiltonian $S^{1}$-action such that $M_{\min} $ and $ M_{\max} $ are surfaces of genus $g>0$. Suppose that $M$ contains a fixed surface of genus $g$ with weights equal to $\{1,1,0\}$,$\{-1,-1,0\}$ or $\{1,-1,0\}$. Then $H(M) \subseteq [-3,3]$.
\end{lemma}
\begin{proof}

If $\mathcal{F}(M)$ is reflective then the weights at $M_{\min}$ and $M_{\max}$ are $\{1,1\}$ and $\{-1,-1\}$ respectively so $H(M) = [-2,2]$. Hence, we may assume that $\mathcal{F}(M)$ is non-reflective. By Corollary \ref{iso}, $\mathcal{G}_{g} \cong \mathcal{Q}$ where $\mathcal{Q}$ is the graph of fixed points associated to $\mathcal{F}(M)$.  By Proposition \ref{fibredelpez}, $\mathcal{F}(M)$ is symplectomorphic to a toric del Pezzo surface. Applying Lemma \ref{calc}, we see that $H(\mathcal{F}(M)) \subseteq [-3,3]$ and since $\mathcal{G}_{g} \cong \mathcal{Q}$ the same holds for $H(M)$.
\end{proof}

\section{Proof of Theorem \ref{theo}}\label{sectionproftheo}



In this section we prove Theorem \ref{theo}. In the first three subsections we give the main component of the argument. As explained in the introduction, the main idea is to push through the proof of Lemma \ref{nonunitweights}. In practice this amounts to proving the inequality of Theorem \ref{isotropy inequality}. We view this inequality as a substitute to Lemma \ref{fourcor}, which applies to the normal bundles of fixed surface $\Sigma$ with weights equal to $\{-1,n,0\}$ ($n\geq 2$) and $M_{\min}$. This is what allows us to ``complete'' the cycle of isotropy $4$-manifolds and conclude the argument.

We briefly describe the main steps towards proving Theorem \ref{isotropy inequality}. Firstly, we prove restrictions on fixed submanifolds on levels close to $H_{\min}$. This allows us to flow down $\Sigma$ continuously between levels (see the beginning of Subsection \ref{defcyclesSc}), forming a continuous family of $2$-cycles $S_{c} \subseteq M_{c}$. Then we show that the quantity $\langle e(H^{-1}(c)), S_{c} \rangle$  is decreasing in terms of $c$ (see the preliminaries for the definition of $e(H^{-1}(c))$). We achieve this by proving that $\Sigma$ has positive intersection with the exceptional divisors emanating from isolated fixed points. The argument concludes in Subsection \ref{proofTheorem112} by showing that in the end $\Sigma$ flows to a particular section of the $S^{2}$-bundle $M_{c}$ for $c$ sufficiently close to $H_{\min}$.
 
In the final subsection, we conclude the proof by dealing with the remaining case where the range of the Hamiltonian is contained in the interval $[-3,3]$.

\begin{theorem} \label{aim}
Suppose $(M,\omega)$ is a relative symplectic Fano $6$-manifold with a Hamiltonian $S^1$-action such that $M_{\min} $ and $ M_{\max} $ are surfaces of genus $g>0$. Suppose additionally that there is no fixed surface of genus $g$ with weights $\{1,1,0\}$, $\{-1,-1,0\}$ or $\{1,-1,0\}$ in $M$. Then there exists a fixed surface $\Sigma \subseteq M$ of genus $g$ such that $\langle c_{1}(M),\Sigma \rangle \leq 2-2g$.
\end{theorem}

Firstly we give a proof of Theorem \ref{aim} assuming Theorem \ref{isotropy inequality}. This slight abuse of ordering is to give the reader a better global understanding of the proof and give motivation for statements that come later in the section.

\begin{theorem}\label{isotropy inequality}
Let $(M,\omega)$ be as in Theorem \ref{aim}. Suppose that $\Sigma$ is a fixed surface with weights $\{-1,n,0\}$ $n \geq 2$. Then the weights at $M_{\min}$ are $\{1,m,0\}$ where $m >1$. Let $L_{1}$ be a sub-bundle of the normal bundle $N(M_{\min})$ with weight $1$ and  let $L_{2}$ be a sub-bundle of the normal bundle $N(\Sigma)$ with weight $-1$. Then $$c_{1}(L_{1})+c_{1}(L_{2}) \leq 0.$$
\end{theorem}

\begin{proof}[Proof of Theorem \ref{aim} assuming Theorem \ref{isotropy inequality}] By assumption the weights at $M_{\min}$ are not $\{1,1,0\}$, so in particular $\mathcal{F}(M)$ is not reflexive by Corollary \ref{nonref}. Hence we are in case 2 of Theorem \ref{struc} and so by Corollary \ref{iso} we have that $\mathcal{G}_{g} \cong \mathcal{Q}$ where $\mathcal{Q}$ is the graph of fixed points corresponding to $\mathcal{F}(M)$. Also note that $\mathcal{F}(M)$ is a toric del Pezzo surface by Proposition \ref{fibredelpez}.

Let $v_{\min},v_{\max}$ denote the extremal vertices of $\mathcal{G}$, by our assumptions $\deg(v_{\min}) \geq 1$ and $\deg(v_{\max}) \geq 1$, hence the same holds true for the extremal vertices of $\mathcal{Q}$. By Lemma \ref{fourbound} the boundary divisors in $\mathcal{F}(M)$ with weight $1$ must contain one of the two extremal vertices. Hence  $\mathcal{F}(M)$ contains at most two boundary divisors with weight $1$.

Consider the ordering on the vertices of $\mathcal{Q}$ defined by traversing the edges of the moment polygon of $\mathcal{F}(M)$ cyclically. We label them $p_{1},\ldots,p_{n }$ where $p_{n+1} = p_{1}$ etc. so that $p_{i},p_{i+1} \in S_{i}$ for some boundary divisor $S_{i}$. This gives us an ordering of the fixed surfaces of genus $g$ in $M$, say $\Sigma_{1},\ldots,\Sigma_{n}$ where $\Sigma_{n+1} = \Sigma_{1}$ etc. 

Note that if $w(S_{i}) > 1$ then there is an isotropy $4$-manifold $N_{i}$ with weight $w(S_{i})$ such that $\Sigma_{i},\Sigma_{i+1} \subset N_{i}$. The normal bundle of $\Sigma_{i}$ has a unique equivariant splitting as $L_{i,1} \oplus L_{i,2}$ so that $L_{i,1}$ is the normal bundle of $\Sigma_{i}$ in $N_{i-1}$ and $L_{i,2}$ is the normal bundle of $\Sigma_{i}$ in $N_{i}$. Note that by assumption for each $\Sigma_{i}$, one of its weights has modulus greater than $1$, hence either $N_{i-1}$ or $N_{i}$ exists, which permits us to distinguish $L_{i,1}$ and $L_{i,2}$. Define $n_{i,j} = c_{1}(L_{i,j})$.

Let us sum $\langle c_{1}(M),\Sigma_{i} \rangle$ over all $\Sigma_{i}$,
\begin{displaymath}
\sum_{1 \leq i \leq n} \langle c_{1}(M),\Sigma_{i} \rangle = \sum_{1 \leq i \leq n} ((2-2g)+n_{i,1}+n_{i,2}). 
\end{displaymath}
By applying Lemma \ref{fourcor} when $w(S_{i}) > 1$ and Theorem \ref{isotropy inequality} (applied to the Hamiltonians $H$ and $-H$) when $w(S_{i}) = 1$, we have that $n_{i,2}+n_{i+1,1} \leq 0$ for all $i$. These inequalities imply that
\begin{displaymath}
\sum_{1 \leq i \leq n} \langle c_{1}(M),\Sigma_{i} \rangle \leq n(2-2g).
\end{displaymath} 
Hence, there exists a fixed surface $\Sigma$ such that $\langle c_{1}(M), \Sigma \rangle \leq 2-2g$. By Corollary \ref{genusg} we have that $g(\Sigma) = g$.
\end{proof}

\subsection{Fixed surface $\Sigma$ of positive genus with weights $\{-1,n,0\}$}
We now begin to give the proof of Theorem \ref{isotropy inequality}. In this subsection  $(M,\omega)$ is a relative symplectic Fano $6$-manifold with a Hamiltonian $S^1$-action such that $M_{\min} $ and $ M_{\max} $ are surfaces of genus $g>0$. We study the case when  $M$ contains  a fixed surface $\Sigma$ of positive genus  with the weights  $\{-1,n,0\}$ with $n \geq 2$. Our goal here is to give restrictions on fixed submanifolds, on levels below $H(\Sigma)$.

\begin{lemma} \label{threesix} 
Let $\Sigma \subseteq M$ be a fixed surface of positive genus with weights $\{-1,n,0\}$ where $n \geq 2$. Then the following statements  hold.
\begin{enumerate}

\item $H(\Sigma) -  H_{\min} \leq 3$.
 
\item The weights at $M_{\min}$ are $\{1,m,0\}$ where $m \geq (H(\Sigma) -  H_{\min} )$.

\item There is no fixed surface $\Sigma' \subseteq M$ with positive genus such that $H(\Sigma') \in (H_{\min},H(\Sigma))$.

\end{enumerate}
\end{lemma}
\begin{proof}  The result follows from applying Corollary \ref{us} to $\mathcal{F}(M)$. \end{proof}

Next, we will show that all fixed submanifolds on levels $k \in ( H_{\min},H(\Sigma))$ are isolated points and establish restrictions on the weights at these fixed points. First we will need to make the following definition.

\begin{definition}\label{flowdefi} Suppose that $p \in M$ is an isolated fixed point with weights $\{a,b,-c\}$ with $a,b,c>0$. Denote by $U(p) \subseteq M$ the union of all gradient spheres with minimum $p$. Denote by $E_{p,x}$ the trace of $U_{p}$ in $M_{x}$. If the weights at $p$ are $\{-a,-b,c\}$ with $a,b,c>0$, we define $E_{p,x}$ to be the trace of $D_{p}$, where $D_{p}$ is defined to be the union of gradient spheres with maximum $p$. 
\end{definition}

\begin{remark}\label{arearemark}
For $\alpha \in (0,1)$, $E_{p,H(p)+\alpha } \subseteq M_{H(p)+\alpha}$ is a smooth sub-orbifold, homeomorphic to $S^{2}$. Let us note some properties of $E_{p,H(p)+\alpha }$. \begin{enumerate}
\item By Lemma \ref{DuisHeckSphere} (and recalling the definition of Euler class for orbi-$S^{1}$-bundles given in Section \ref{prelimperlim}): \begin{displaymath}
\langle e(H^{-1}(H(p)+\alpha),[E_{p,H(p)+\alpha}] \rangle = -\frac{1}{ab}.
\end{displaymath}

\item Lemma \ref{DH} states that \begin{displaymath}
\frac{d}{d \alpha}( \omega( E_{H(p)+\alpha} ) ) = \langle -e(H^{-1}(H(p)+\alpha),[E_{p,H(p)+\alpha}] \rangle. 
\end{displaymath}
Hence,
\begin{displaymath}
\omega( E_{H(p)+\alpha} ) = \frac{\alpha}{ab}. 
\end{displaymath}

\end{enumerate}
\end{remark}

The following Lemma is well-known so we omit its proof.
\begin{lemma} \label{selfint}
Let $E_{p,x} \subseteq M_{x}$ be as in Definition \ref{flowdefi} and $x \in (H(p),H(p) +1)$. Then $E_{p,x} \cdot E_{p,x} <0$. 
\end{lemma}

Now we give the main result of this subsection.

\begin{lemma}\label{twofixed}
Let $\Sigma \subseteq M$ be a fixed surface of positive genus such that the weights at $\Sigma$ are $\{-1,n,0\}$ with $n \geq 2$. Then for each integer $k \in (H_{\min},H(\Sigma))$, any fixed submanifold contained in $H^{-1}(k)$ must be an isolated fixed point with weights $\{ -1,a,b \}$ for some $a,b>0$.
\end{lemma}

\begin{proof}
Statement 3 of Lemma \ref{threesix} shows that there is no fixed surface of genus $g$ in $H^{-1}(k)$. On the other hand, since $k < H(\Sigma) < 0$ there is no fixed sphere on level $k$ by Corollary \ref{nosphere}. Hence all fixed points on level $k \in (H_{\min},H(\Sigma))$ are isolated.

We observe that for an isolated fixed $p$ with $H(p) \in (H_{\min},H(\Sigma))$ any negative weight at $p$ must be equal to $-1$. This can be deduced from Lemma \ref{missinglemma} and Lemma \ref{resweight}.2) once one has that the weights at $M_{\min}$ are $\{1,m,0\}$ where $m > H(p)-H_{\min}$. This was shown in Lemma \ref{threesix}. Hence to prove the lemma it remains to exclude fixed points of the form $\{-1,-1,w\}$ where $w >0$. We consider now two cases:

Suppose first that  $H(p) = H_{\min}+1$. The sphere $E_{p,H(p) -\varepsilon} \subseteq M_{H(p) -\varepsilon}$ constructed in Definition \ref{flowdefi} has strictly negative self-intersection by Lemma \ref{selfint}. However, $M_{H(p) -\varepsilon}$ is an $S^{2}$-bundle over $M_{\min}$ hence contains no sphere with non-zero self intersection. This yields a contradiction.

Suppose now that $H(p) = H_{\min}+2 $, so that  $H(\Sigma) - H_{\min} = 3$. Applying Definition \ref{flowdefi}, we get a smoothly embedded sphere $E_{p,H(p) - \varepsilon} \subseteq M_{H(p) -\varepsilon}$ for $\varepsilon \in (0,1)$.  By Remark \ref{arearemark}.2, \begin{displaymath}
\omega(E_{p,H(p) - \varepsilon}) = \varepsilon.
\end{displaymath}

Note that $E_{p,H(p)-\varepsilon} $ is disjoint from the trace of the isotropy $4$-manifold $N$ such that $N_{\min} = M_{\min}$. Also, if $\phi : M_{H(p) - \varepsilon} \rightarrow M_{\min}$ is the map induced by the retraction constructed in Lemma \ref{retractmin}, then $\phi|_{E_{p,H(p)-\varepsilon}}$ has degree $0$. Hence, as a class in $H_{2}(M_{H(p) -\varepsilon})$ we have 
$$[E_{p,H(p)-\varepsilon}] = \sum n_{i} [E_{p_{i},H(p)-\varepsilon }]$$ 
where $\{p_{i}\}$ is the set of isolated fixed points on level $H_{\min}+1$. On the other hand note that for each $p_{i}$ \begin{displaymath}
\lim_{\varepsilon \rightarrow 1}  \omega( E_{p_{i},H(p)-\varepsilon }) = 0.
\end{displaymath} However, $\omega (S_{\varepsilon})$ tends to $1$ as $\varepsilon \rightarrow 1$ by Remark \ref{arearemark}.2. This yields a contradiction.
\end{proof}

\subsection{The downward gradient flow close to $M_{\min}$ and exceptional spheres}
In this subsection, we again denote by $(M,\omega)$ a relative symplectic Fano $6$-manifold with a Hamiltonian $S^{1}$-action such that $M_{\min},M_{\max}$ are surfaces of genus $g>0$, such that $M$ contains no fixed surface of genus $g$ with weights $\{1,1\}$,$\{-1,-1\}$ or $\{1,-1\}$. As before, we suppose that $M$ contains a fixed surface $\Sigma$ of genus $g$ with weights $\{-1,n,0\}$ and $n \geq 2$.

First, we will define continuous maps between reduced spaces of $M$, which extends the usual partially defined gradient maps $gr^{c_{1}}_{c_{2}}$. Then we study the cycles that are contracted by these maps, namely $E_{p,x}$. In particular, we will show that these spheres lift to exceptional spheres in the resolutions of the reduced spaces which were constructed in Section \ref{locKahlsec}.

Let $k \in (H_{\min},H(\Sigma))$ be an integer and $\varepsilon > 0$ be small enough. From Definition \ref{gradmapdef} we have maps \begin{displaymath}
gr^{k-\varepsilon}_{k} : M_{k - \varepsilon} \rightarrow M_{k}
\end{displaymath}
and
\begin{displaymath}
gr^{k+\varepsilon}_{k} : M_{k+\varepsilon} \rightarrow M_{k}.
\end{displaymath}

By Lemma \ref{twofixed} all fixed points with $k = H(p)$ are isolated with weights $\{-1,a,b\}$ where $a,b>0$. This implies that $gr^{k-\varepsilon}_{k}$ is a homeomorphism. Hence 
\begin{displaymath}
(gr^{k-\varepsilon}_{k})^{-1} \circ gr^{k+\varepsilon}_{k}  : M_{k+\varepsilon} \rightarrow M_{k - \varepsilon}
\end{displaymath}
is continuous and extends the usual partially defined gradient map.

\begin{definition}\label{flowmap}
Let $c_{1},c_{2} \in (H_{\min},H(\Sigma))$ with $c_{1}<c_{2}$. Denote by $gr_{c_{1}}^{c_{2}} : M_{c_{2}} \rightarrow M_{c_{1}}$ the map extending the partially defined gradient map as was explained above.
\end{definition}

\begin{remark}
Let $k \in  (H_{\min},H(\Sigma))$ be an integer and let $\varepsilon$ be a constant such that $0<\varepsilon<1$. Let $p_{1},\ldots,p_{n}$ be the isolated fixed points in the level set $H^{-1}(k)$. Then $gr_{k-\varepsilon}^{k+\varepsilon}$ contracts each of the $E_{p_{i},k+\varepsilon}$ (see Definition \ref{flowdefi}) and is a homeomorphism on their complement.
\end{remark}

Let $p$ be an isolated fixed point such that $H(p) \in (H_{\min},H(\Sigma))$. In the following lemma we construct an exceptional sphere  in the symplectic resolution ${\overline M}_{H(p)+\varepsilon}$ of $M_{ H(p)+\varepsilon}$ (see Definition \ref{resolutionDef}) for small $\varepsilon>0$.


\begin{lemma} \label{nicesphere}
Let $p$ be an isolated fixed point with $H(p) \in (H_{\min},H(\Sigma))$. Let $\pi : \overline{M}_{H(p)+\varepsilon} \rightarrow M_{H(p)+\varepsilon}$ denote the resolution. Then for $\varepsilon > 0$ small enough there is a exceptional sphere $\tilde{E} \subseteq \overline{M}_{H(p)+\varepsilon} $ such that in homology $\pi_{*}([\tilde{E}]) = [E_{p,H(p)+\varepsilon}]$ (where $E_{p,x}$ is as in Definition \ref{flowdefi} ).
\end{lemma}
\begin{proof}

We may (equivariantly) identify a neighbourhood $U$ of $p$ with a neighbourhood of the origin in $\C^{3}$ so that the $S^{1}$-action is linear. Denote by $U_{x}$ the reduced space of $U$ at level $x$.

For $\varepsilon> 0$ small enough, $U_{H(p) - \varepsilon}$ is biholomorphic to a neighbourhood of the origin in $\C^{2}$, since the negative weight at $p$ equals $-1$. The complex analytic map 
$$F = gr_{H(p)-\varepsilon}^{H(p)+\varepsilon}:U_{ H(p)+\varepsilon} \rightarrow U_{H(p) -\varepsilon}$$ 
is a weighted blow-up map that contracts $E_{H(p)+\varepsilon}$ (see Definition \ref{flowmap}). 

Consider now a resolution of singularities $\pi:\overline{M}_{ H(p)+\varepsilon} \rightarrow M_{ H(p)+\varepsilon}$ that is holomorphic on $U$. Where the composition $F \circ \pi$ is defined, it is a surjective morphism of smooth complex surfaces, hence a composition of simple blow-ups. The result now follows from applying Lemma \ref{meromorphic} to $F \circ \pi$.
\end{proof}

\begin{lemma} \label{sphereclose}
Let $p$ be an isolated fixed point such that $H(p) \in (H_{\min},H(\Sigma))$. Assume the weights at $p$ are $\{-1,a,b\}$ with $a,b>0$ and $\gcd(a,b) =1$. Let $c_{1},c_{2} \in (H(p),H(\Sigma))$ be such that $c_{1}<c_{2}$ and $c_{1}-H(p) <1$. Then there is an exceptional sphere $\tilde{E} \subset \overline{M}_{c_{2}}$ such that in homology \begin{displaymath}
(gr^{c_{2}}_{c_{1}} \circ \pi)_{*} [\tilde{E}] = [E_{p,c_{1}}].
\end{displaymath}
\end{lemma}
\begin{proof}
The weights at $p$ are $\{-1,a,b\}$ where $a,b>0$. We suppose that $a,b>1$ otherwise the proof is easier. Let $S_{1},S_{2}$ be the gradient spheres with minimum at $p$, with weights $a$ and $b$ respectively. Then there is an $S^{1}$-invariant K\"ahler structure on a neighbourhood $U$ of $S_{1} \cup S_{2}$, by Theorem \ref{semilocalKahler}. We may construct a resolution $\pi$ of reduced spaces, such that $\pi|_{U} : \overline{ U}_{z} \rightarrow  U_{z}$ is holomorphic for any $z \in (H(p),H(\Sigma))$. The biholomorphism \begin{displaymath}
 gr^{c_{2}}_{c_{1}}: U_{c_{2}} \rightarrow U_{c_{1}} 
\end{displaymath} lifts to a diffeomorphism of the resolutions \begin{displaymath}
{\overline{gr}}^{c_{2}}_{c_{1}} : \overline{ U}_{c_{2}} \rightarrow \overline{ U}_{c_{1}}.
\end{displaymath}

Recall that all fixed points on levels contained in $(H_{\min},H(\Sigma))$ are isolated and have weights $\{-1,a',b'\}$ where $a',b'>0$. The traces in $M_{c_{1}}$ of gradient spheres with maximum at such isolated fixed points give us a finite collection of non-orbifold points $\{q_{1},\ldots,q_{k}\} \subseteq M_{c_{1}}$. Then $(gr^{c_{2}}_{c_{1}})^{-1}$ is well defined and a homeomorphism on $M_{c_{1}} \setminus \{q_{1},\ldots,q_{k}\}$. Since each $q_{i}$ is a non-orbifold point, $\{q_{1},\ldots,q_{k}\}$ corresponds to a finite collection of points in the resolution $\overline{M}_{c_{1}}$.

By Lemma \ref{nicesphere}, there is an exceptional sphere in $\tilde{E} \subseteq M_{c_{1}}$ such that $\pi_{*}[\tilde{E}] = [E_{p,c_{1}}]$. We may perturb $\tilde{E}$ smoothly to be disjoint from the $q_{i}'s$. Hence using $(gr_{c_{1}}^{c_{2}})^{-1}$, we get an exceptional sphere in $\overline{M}_{c_{2}}$ with the required properties. \end{proof}

\subsection{The $2$-cycles $S_{c}$ and Euler numbers of their associated orbi-bundles} \label{defcyclesSc}

We will proceed to study the family of $2$-cycles $S_{c} \subseteq M_{c}$, defined by mapping $\Sigma$ via the gradient maps $gr^{c_{2}}_{c_{1}}$. We prove in Lemma \ref{fibre wiggly} that the cycles $S_{c}$ have intersection $1$ with the symplectic fibre. Then, in Lemma \ref{excint} we show that they have non-negative intersection with the spheres $E_{p,x}$ contracted by the gradient maps. Eventually, we conclude the section by proving the inequality stated in Corollary \ref{chernin}, from which it is straightforward to deduce Theorem \ref{isotropy inequality}.

\begin{definition} Let $\Sigma$ be as in Theorem \ref{isotropy inequality}. Denote by $D(\Sigma) \subseteq M$ the union of all gradient spheres with maximum in $\Sigma$. For $c \in (H(\Sigma)-1,H(\Sigma)]$ let $S_{c}$ denote the trace of $D(\Sigma)$ in $M_{c}$. More generally for $c \in (H_{\min},H(\Sigma)]$, define $S_{c} = gr_{c}^{c'}(S_{c'})$ where $c'$ is any value in $(H(\Sigma)-1,H(\Sigma))$. It is easy to check that the $2$-cycle $S_{c}$ is independent of the choice of $c'$.

\end{definition}
\begin{remark}
For $c \in (H(\Sigma)-1,H(\Sigma)]$, $S_{c}$ is a smooth surface diffeomorphic to $\Sigma$. For $c = H(\Sigma)- \varepsilon$, with $\varepsilon>0$ small enough, $S_{c}$ is a symplectic surface.
\end{remark}

\begin{lemma}\label{fibre wiggly}
Let $c \in (H_{\min},H(\Sigma)]$ and $\mathcal{F}_{c} \subseteq M_{c}$ be the symplectic fibre then \begin{displaymath}
\mathcal{F}_{c} \cdot S_{c} = 1.
\end{displaymath}
\end{lemma}
\begin{proof}
Let $\phi : M \rightarrow M_{\min}$ be the equivariant retraction (see Lemma \ref{retractmin}). There exists a continuous map $S: \Sigma \times (H_{\min},H(\Sigma)] \rightarrow M/S^{1}$, such that the restriction of this map to $\Sigma \times \{c\}$ represents the $2$-cycle $S_{c} \subseteq M_{c}$.

Hence, the degree of $\phi \circ S $ restricted to $\Sigma \times \{c\}$ is equal to $deg (\phi|_{\Sigma})$ for all $c \in (H_{\min},H(\Sigma)]$. By Theorem \ref{struc} $deg (\phi|_{\Sigma}) = 1$. Now the result follows since \begin{displaymath}
deg((\phi \circ S)  |_{\Sigma \times \{c\}}) = \mathcal{F}_{c} \cdot S_{c}.
\end{displaymath} 
\end{proof}

\begin{lemma} \label{excint} Let $p\in M$ be an isolated fixed point such that $H(p) \in (H_{\min},H(\Sigma))$. Then for $\alpha \in (0,1)$, we have that \begin{displaymath}
E_{p,H(p)+\alpha}  \cdot S_{H(p)+\alpha} \geq 0.
\end{displaymath}
\end{lemma}
\begin{proof}
If $p$ is not contained in $S_{H(p)}$ then clearly  $E_{p,H(p)+\alpha}  \cdot S_{H(p)+\alpha} = 0$. So we will assume $p \in S_{H(p)}$.

By Lemma \ref{twofixed} the weights at $p$ are $\{-1,a,b\}$ where $a,b>0$. Since $p \in S_{H(p)}$, there is a gradient sphere $S'$ of weight $1$ with $S'_{\min} = p$ so  $\gcd(a,b)=1$.

Let $\pi : \overline{M}_{H(\Sigma) - \varepsilon} \rightarrow M_{H(\Sigma) - \varepsilon}$ be the resolution for $\varepsilon\in (0,1)$. Since $\gcd(a,b)=1$, we can apply Lemma \ref{sphereclose} for the values $c_{2} = H(\Sigma) -\varepsilon$ and $c_{1} = H(p)+\alpha$. We see that there is an exceptional sphere $\tilde{E} \subseteq \overline{M}_{c_{2}}$ such that \begin{displaymath}(gr^{c_{2}}_{c_{1}} \circ \pi)_{*}[\tilde{E}] = [E_{p,c_{1}}]. \end{displaymath}

Let now $\varepsilon > 0$ be small enough so that $S_{H(\Sigma) - \varepsilon} \subseteq M_{H(\Sigma) - \varepsilon}$ is symplectic.  Since $S_{H(\Sigma) - \varepsilon}$ does not intersect the orbifold locus of $M_{H(\Sigma) - \varepsilon}$, it lifts to a symplectic surface in $\tilde{S} \subseteq \overline{M}_{H(\Sigma) - \varepsilon}$. By Theorem \ref{inttheo} \begin{displaymath}
\tilde{E} \cdot \tilde{S} \geq 0.
\end{displaymath}
Now the inequality follows since 
$$(gr^{c_{2}}_{c_{1}} \circ \pi)_{*}[\tilde{E}] = [E_{p,H(p)+\alpha}]\;\;\; {\rm and}\;\;\; (gr^{c_{2}}_{c_{1}} \circ \pi)_{*}[\tilde{S}] = [S_{H(p)+\alpha}].$$
 \end{proof}

\begin{corollary}\label{chernin}
 Let $k \in (H_{\min},H(\Sigma))$ be an integer and let $\varepsilon \in (0,1)$. Then \begin{displaymath}
 \langle e(H^{-1}(k+\varepsilon)),S_{k+\varepsilon} \rangle \geq \langle e(H^{-1}(k-\varepsilon)),S_{k-\varepsilon} \rangle.
 \end{displaymath}
\end{corollary}
\begin{proof}
Let \begin{displaymath}F = gr^{k+\varepsilon}_{k-\varepsilon}: M_{k+\varepsilon} \rightarrow M_{k- \varepsilon},\end{displaymath} be the downward flowing gradient map. Then, since $F$ only contracts the spheres $E_{p,k+ \varepsilon}$ and is a homeomorphism on their complement, we have
$$
[S_{k+\varepsilon}] = F^{*}([S_{k - \varepsilon}])+\sum_{p \in \mathcal{P}_{k}} n_{p} [E_{p,k+ \varepsilon}].
$$
Here $\mathcal{P}_{k}$ is the set of isolated fixed points such that $H(p) = k$. 

For each $p\in \mathcal{P}_{k} $,
Lemmas \ref{excint} and \ref{selfint} tell us
$$E_{p,k+\varepsilon}  \cdot S_{k+\varepsilon} \geq 0,\;\;\;
 E_{p,k+ \varepsilon} \cdot E_{p,k+ \varepsilon} <0.$$ 
 Combining these two inequalities, we have that $n_{p}\leq 0$ for each $p \in \mathcal{P}_{k}$. 
 
Note that since $S_{k-\varepsilon}$ can be perturbed to be disjoint from the images of $E_{p,k+\varepsilon}$ under $gr^{k + \varepsilon}_{k-\varepsilon}$, we have that
$$
\langle e(H^{-1}(k - \varepsilon)),F^{*}([S_{k - \varepsilon}]) \rangle  = \langle e(H^{-1}(k+\varepsilon)),[S_{k+\varepsilon}] \rangle. 
$$

Therefore it is sufficient to show that 
$$\langle  e(H^{-1}(k+\varepsilon)) , E_{p,k+\varepsilon} \rangle < 0$$ 
for each $p \in \mathcal{P}_{k}$. To show this, recall that each $p \in \mathcal{P}_{k}$ has weights $\{-1,a,b\}$ where $a,b>0$. Hence,  
$$\langle e(H^{-1}(k+\varepsilon)) , E_{p,k+\varepsilon} \rangle = \frac{-1}{ab} < 0$$ 
by Remark \ref{arearemark}.
\end{proof}

\subsection{Proof of Theorem \ref{isotropy inequality}} \label{proofTheorem112}

\begin{definition} \label{theotherdefinition}

 Consider the sub-bundle $L_{1}$ of $N(M_{\min})$ with weight $1$, where $N(M_{\min})$ denotes the normal bundle of $M_{\min}$ in $M$. Define a subset $U(M_{\min},L_{1}) \subseteq M$ to be the union of all gradient spheres that are tangent to $L_{1}$ at $M_{\min}$. For $x \in (H(\Sigma),H(\Sigma)+1)$ define $T_{x} \subseteq M_{x}$ to be the trace of $U(M_{\min},L_{1})$. 
\end{definition}

\begin{remark}
Note that $T_{x}$ is a section of the $S^{2}$-bundle $M_{x} \rightarrow M_{\min}$. For $(x-H_{\min}) = \varepsilon$ small enough, it is a symplectic surface.
\end{remark}

\begin{lemma} \label{bottomclass}Let $x \in (H_{\min},H_{\min}+1)$ then as homology classes \begin{displaymath}
[T_{x}] = [S_{x}].
\end{displaymath}
\end{lemma}
\begin{proof}
$M_{x}$ is homeomorphic to an $S^{2}$-bundle over $M_{\min}$. Note that $T_{x}$ is a section of this bundle, and so is the trace $N_{x}$ of the isotropy $4$-manifold $N$ containing $M_{\min}$. Note that the symplectic fibre $\mathcal{F}_{x} \subseteq M_{x}$ represents a fibre of this $S^{2}$-bundle. 

Note that $[S_{x}] \cdot [N_{x}] = 0$, since $S_{x}$ and $N_{x}$ are disjoint. On the other hand, by Lemma \ref{fibre wiggly}, $[S_{x}] \cdot \mathcal{F}_{x} = 1$. These two equalities uniquely characterise $[T_{x}] \in H_{2}(M_{x})$. Hence, $[S_{x}] = [T_{x}]$. 
\end{proof}

\begin{proof}[\textbf{Proof of Theorem \ref{isotropy inequality}}]

For $c_{1} \in (H_{\min},H_{\min}+1)$ we have by Lemma \ref{bottomclass} that  \begin{displaymath}
 e(H^{-1}(c_{1})),[S_{c_{1}}] \rangle =\langle e(H^{-1}(c_{1})),[T_{c_{1}}] \rangle = c_{1}(L_{1}).
\end{displaymath} For $
 c_{2} \in (H(\Sigma) - 1,H(\Sigma))$ we have that \begin{displaymath}
 e(H^{-1}(c_{2})),[S_{c_{2}}] \rangle = -c_{1}(L_{2}),
 \end{displaymath}
since for $H(\Sigma)-c_{2} = \varepsilon$ small enough, $S_{c_{2}}$ is the unit circle bundle of $L_{2}$ (the minus sign coming from the fact that $L_{2}$ has weight $-1$). Hence, by Corollary \ref{chernin} we conclude that $c_{1}(L_{1})+c_{1}(L_{2}) \leq 0$. 
\end{proof}

\subsection{Small Hamiltonian}\label{smallhamsec}

In this subsection we will prove the following:
\begin{theorem}\label{smallh}
Let $M$ be a relative symplectic Fano $6$-manifold with a Hamiltonian $S^{1}$-action such that $M_{\min},M_{\max}$ are surfaces of genus $g>0$. Suppose that $H(M) \subseteq [-3,3]$. Then there exists a fixed surface $\Sigma \subseteq M$ of genus $g(\Sigma) \geq g$ such that $\langle c_{1}(M),\Sigma \rangle \leq 2-2g$.
\end{theorem}

The proof will proceed using the localisation formula for $c_{1}^{S^{1}}(M)$ given in Theorem \ref{bigloc}. First, we prove some restrictions on contributions of fixed submanifolds to this formula. 

For the remainder of the section, we denote by $\mathcal{S}_{+}$ the set of fixed surfaces of positive genus and we denote by $\mathcal{S}_{0}$ the set of fixed surfaces of genus $0$.

\begin{lemma}\label{posgenus}

Let $M$ be a symplectic  $6$-manifold with a Hamiltonian $S^{1}$-action such that $M_{\min},M_{\max}$ are surfaces of genus $g > 0$. Then \begin{displaymath}
\sum_{\Sigma \in \mathcal{S}_{+}} \langle c_{1}(M),\Sigma \rangle  \leq -\sum_{\Sigma \in \mathcal{S}_{+}} \beta(\Sigma) .  
\end{displaymath} (with $\beta$ is as in Definition \ref{bigloc}).
\end{lemma}

\begin{proof}
By Lemma \ref{Liapp} the genus of fixed surfaces is constant along connected components of $\mathcal{G}$. Let $\Sigma_{i} $ for $i=1,\ldots,n$ be a chain of fixed surfaces corresponding to such a  component with $g(\Sigma_{i})=g'>0$ for each $i$. The calculation below proves the statement of the lemma for each such connected component. The case when the component is a cycle is easier so we just present the case when the component is an interval. 

There is a sequence of isotropy $4$-manifolds $N_{i}$ for $i =1,\ldots,n-1$ such that $\Sigma_{i},\Sigma_{i+1} \subset N_{i}$. Note that since $v_{\Sigma_{1}}$ and $v_{\Sigma_{n}}$ have degree $1$ in $\mathcal{G}$, by Lemma \ref{Liapp} $\Sigma_{1},\Sigma_{n}$ both have a weight with modulus $1$.

The normal bundle of $\Sigma_{i}$ may be split equivariantly as $L_{i,1} \oplus L_{i,2}$, where $L_{i,1}$ is the normal bundle of $\Sigma_{i}$ in $N_{i-1}$ and $L_{i,2}$ is the normal bundle of $\Sigma_{i}$ in $N_{i}$. Let $n_{i,j} = c_{1}(L_{i,j})$ and $w_{i,j}$ be the weight associated to $L_{i,j}$. With this convention we have that $|w_{1,1}| = |w_{n,2}| = 1$. 

Note that $w_{i,2} = -w_{i+1,1}$ and in particular $w_{i,2}^2 = w_{i+1,1}^2$. Applying this to the sum of $\beta(\Sigma_{i})$ we have
\begin{displaymath}
\sum_{i=1}^{n} \beta(\Sigma_{i}) = \sum_{i=1}^{n}\left( \frac{2-2g'}{w_{i,1} w_{i,2}} - \frac{n_{i,1}}{w_{i,1}^2} - \frac{n_{i,2}}{w_{i,2}^2}\right)
\end{displaymath}
\begin{displaymath}
=\sum_{i=1}^{n} \frac{2-2g'}{w_{i,1} w_{i,2}}+ \sum_{i=1}^{n-1} \frac{-1}{w_{i,2}^2} (n_{i,2}+n_{i+1,1}) - n_{1,1} - n_{n,2}.
\end{displaymath}

We sum $\langle c_{1}(M), \Sigma_{i} \rangle$ over $i = 1,\ldots,n$.
\begin{displaymath}
\sum_{i=1}^{n} \langle c_{1}(M), \Sigma_{i} \rangle
= \sum_{i=1}^{n} (2-2g') +\sum_{i=1}^{n-1} (n_{i,2}+n_{i+1,1})+n_{1,1}+n_{n,2}.
\end{displaymath}
For each $i$, $(n_{i,2} +n_{i+1,1}) \leq 0$ by Lemma \ref{fourcor}. Combining these inequalities with the above equations we deduce that the quantity
\begin{equation} \label{longeq}
\sum_{i=1}^{n} (\beta(\Sigma_{i})+\langle c_{1}(M),\Sigma_{i} \rangle) = \sum_{i=1}^{n}\left(1+\frac{1}{w_{i,1} w_{i,2}}\right)(2-2g')+\sum_{i=1}^{n-1} (n_{i,2}+n_{i+1,2})\left(1 - \frac{1}{w_{i,2}^{2}}\right)
\end{equation}
is non-positive.
\end{proof} 
 
\begin{lemma}\label{isosimp}
Suppose that $(M^{6},\omega)$ is a relative symplectic Fano $6$-manifold such that $M_{\min} $ and $ M_{\max} $ are surfaces of genus $g>0$. Assume additionally that $H(M) \subseteq [-3,3]$. Then if $w$ is a weight at an isolated fixed point in $M$ then $1 \leq |w| \leq 2$.
\end{lemma}\label{wtwo}
\begin{proof}
Any isolated fixed point $p$ with $H(p) = 2$ must have all positive weights equal to $1$ by Lemma \ref{missinglemma}. Hence the weights at $p$ are either $\{1,1,-4\}$ or $\{1,-2,-1\}$ by the weight sum formula \ref{relcho} (\ref{weightcho}). We claim that the combination $\{1,1,-4\}$ is impossible. Indeed, this would imply the existence of a sphere in  $M_{2+\varepsilon}$ with negative self-intersection (see Lemma \ref{selfint}), contradicting that $M_{2+\varepsilon}$ is homeomorphic to an $S^{2}$-bundle over a surface of positive genus. Hence the weights at $p$ must be $\{1,-2,-1\}$.

If $p$ satisfies $|H(p)| \leq 1$ then by the above it cannot have a weight greater than $2$, since this would imply the existence of a fixed point of level $H = 2$, by Lemma \ref{gradsphereint}, which does not have weights $\{1,-2,-1\}$. Repeating the same argument for $-H$ proves the result.
\end{proof}

\begin{lemma} \label{isolatedloc}
Suppose that $(M^{6},\omega)$ is a relative symplectic Fano $6$-manifold such that $M_{\min} $ and $ M_{\max} $ are surfaces of genus $g>0$. Assume additionally that $H(M) \subseteq [-3,3]$. Then for any isolated fixed point $p$ with weights $\{w_{1},w_{2},w_{3}\}$, we have that 
\begin{displaymath}
\alpha(p) = \frac{w_{1}+w_{2}+w_{3}}{w_{1}w_{2}w_{2}} \leq 0.
\end{displaymath}
\end{lemma}

\begin{proof}Let $p$ be an isolated fixed point with weights $\{w_{1},w_{2},w_{3}\}$, we have $1\le |w_{i}| \leq 2$ for each $i$ by Lemma \ref{isosimp}. These weights can not be of the same sign, so either two of them are positive and one negative, or otherwise. Consider the first case. Then the sum of the two positive $w_{i}$ is at least $2$, so we have $w_1+w_2+w_3\ge 0$. It follows that $\frac{w_1+w_2+w_3}{w_1w_2w_3}\le 0$. The second case is similar.  
\end{proof}

\begin{lemma} \label{betapos}Suppose that $(M^{6},\omega)$ is a relative symplectic Fano $6$-manifold such that $M_{\min} $ and $ M_{\max} $ are surfaces of genus $g>0$. Assume additionally that $H(M) \subseteq [-3,3]$. Then
\begin{displaymath}
\sum_{\Sigma \in \mathcal{S}_{+}} \beta(\Sigma) \geq 0.
\end{displaymath}
\end{lemma}
\begin{proof}
We start by calculating $\beta(\Sigma)$ for any $\Sigma \in \mathcal{S}_{0}$. By Corollary \ref{nosphere}, $H(\Sigma) = 0$. Hence, the weights at $\Sigma$ are $\{-1,1,0\}$ by the weight sum formula \ref{relcho} (\ref{weightcho}) and the fact that the $S^{1}$-action is effective. Therefore \begin{displaymath}
\beta(\Sigma) = -\langle c_{1}(M),\Sigma \rangle \leq 0,
\end{displaymath}
by the relative symplectic Fano condition applied to the sphere $\Sigma$.

By Lemma \ref{isolatedloc}, $\alpha(p) \leq 0$ for all $p \in \mathcal{P}$. Combining with the above we have that: 
\begin{displaymath}
\sum_{\Sigma \in \mathcal{S}_{0}} \beta(S)+\sum_{p \in \mathcal{P}} \alpha(p) \leq 0.
\end{displaymath}
Therefore by Equation (\ref{bigloceq}) of Lemma \ref{bigloc} we have that
\begin{displaymath} 
\sum_{\Sigma \in \mathcal{S}_{+}} \beta(\Sigma) \geq 0 .
\end{displaymath}
\end{proof}

\begin{proof}[Proof of Theorem \ref{smallh}] We will split the proof into two parts, depending on whether $\mathcal{F}(M)$ is reflective or not.

1. \textit{The case when $\mathcal{F}(M)$ is not reflective.} If $\mathcal{Q}$ is the graph of fixed points of $\mathcal{F}(M)$ then $\mathcal{Q} \cong \mathcal{G}_{g}$ by Corollary \ref{iso}. On the other hand, all $\Sigma \in \mathcal{S}_{+}$ are of genus $g$. 

By applying Lemma \ref{fourcor} to $\mathcal{F}(M)$ we have that \begin{displaymath}
\sum_{\Sigma_{i} \in \mathcal{S}_{+}} \frac{2-2g}{w_{i,1}w_{i,2}} = 0.
\end{displaymath} 
Therefore substituting this into Equation (\ref{longeq}) from the last line of the proof of Theorem \ref{smallh}, over each connected component $C_{j} \subset \mathcal{G}_{g}$ we obtain:
\begin{displaymath}
\sum_{\Sigma_{i} \in \mathcal{S}_{+}} (\beta(\Sigma_{i})+\langle c_{1}(M),\Sigma_{i} \rangle) = n(2-2g)+\sum_{\Sigma_{i} \in C_{j}}\sum_{i=1}^{n-1} (n_{i,2}+n_{i+1,2})\left(1 - \frac{1}{w_{i,2}^{2}}\right).  
\end{displaymath}
Hence,
\begin{displaymath} 
\sum_{\Sigma \in \mathcal{S}_{+}} (\beta(\Sigma)+\langle c_{1}(M),\Sigma \rangle) \leq  n(2-2g).   
\end{displaymath}
By Lemma \ref{betapos} the sum of $\beta(\Sigma)$ is non-negative, so we obtain:
\begin{displaymath}
\sum_{\Sigma \in \mathcal{S}_{+}}  \langle c_{1}(M),\Sigma \rangle \leq n(2-2g).
\end{displaymath}
This tell us that there exists a fixed surface $\Sigma'$ with $g(\Sigma') = g$ and such that $\langle c_{1}(M),\Sigma' \rangle \leq 2-2g. $

2. \textit{The case when $\mathcal{F}(M)$ is reflective.} Applying Corollary \ref{iso}.2
 for the reflective case, the number of fixed surfaces is equal to $\frac{\chi(\mathcal{F}(M))}{2}+1$.  Since $\mathcal{F}(M)$ is a toric del Pezzo surface $\chi(\mathcal{F}(M)) = 4$ or $6$, and so $\mathcal{S}_{+}$ contains either $3$ or $4$ surfaces.

Let us now show how the theorem follows from the following claim.
\begin{claim} \label{inclaim}
\begin{displaymath}
\sum_{\Sigma \in \mathcal{S}_{+}}\left(1+\frac{1}{w_{i,1} w_{i,2}}\right)(2-2g(\Sigma)) \leq 4(2-2g).
\end{displaymath}
\end{claim}

To see that Theorem \ref{smallh} follows from this claim, we substitute this inequality into Equation (\ref{longeq}) from the proof of Lemma \ref{posgenus}. This substitution shows  that  $$
\sum_{\Sigma \in \mathcal{S}_{+}} (\beta(\Sigma)+\langle c_{1}(M),\Sigma \rangle )\leq 4(2-2g).
$$
On the other hand, Lemma \ref{betapos} tells us that the sum of $\beta(\Sigma)$ is positive. Hence, 
$$
\sum_{\Sigma \in \mathcal{S}_{+}}  \langle c_{1}(M),\Sigma \rangle \leq 4(2-2g).
$$
Since there are most $4$ fixed surfaces, there exists at least one fixed surface $\Sigma$ with $g(\Sigma) \geq g$ such that $\langle c_{1}(M),\Sigma \rangle \leq 2-2g$.
 
\textit{Proof of Claim \ref{inclaim}}. Since $\mathcal{F}(M)$ is reflective, the weights at $M_{\min},M_{\max}$ are $\{1,1,0\}$, $\{-1,-1,0\}$ respectively. By a direct computation, the contribution of the extremal fixed surfaces to the left hand side is $4(2-2g)$. Hence, it is sufficient to show that the contribution of non-extremal fixed surfaces must be non-positive. This also follows from direct computation, once one notes that the weights at any such surface must be equal to $\{1,-1,0\}$, $\{-1,2,0\}$ or $\{-2,1,0\}$.  \end{proof}

\section{Proof of Theorem \ref{maintheorem} and Corollary \ref{chern}} \label{hirsection} 

In this section we finalise the proof of Theorem \ref{maintheorem} and deduce Corollary \ref{chern} from it. The only  remaining case of Theorem \ref{maintheorem} to prove is the following one.

\begin{theorem}\label{MmindlePezzo}
Let $M$ be a symplectic Fano $6$-manifold with a Hamiltonian $S^1$-action and such that $M_{\max}$ is a point and $M_{\min}$ has dimension $4$. Then $b_2(M_{\min})\le 9$, and moreover $M_{\min}$ is diffeomorphic to a del Pezzo surface.
\end{theorem}

Let us explain why this result indeed finishes the proof of Theorem \ref{maintheorem}.

\begin{proof}[Proof of Theorem \ref{maintheorem}] The manifold $M_{\min}$ can be of dimensions $0$, $2$, or $4$ and so we consider separately these three cases.

1) $\dim(M_{\min})=0$. In this case $M_{\min}$ is a point since it is connected.

2) $\dim(M_{\min})=2$.  In this case three possibilities can occur. First, $\dim(M_{\max})=0$, in which case $\pi_1(M_{\min})=\pi_1(M_{\max})=0$ by Theorem  \ref{Li} and so $M_{\min}$ is a $2$-sphere. Second case is when $\dim(M_{\max})=2$. This case is covered by Theorem \ref{theo}. The case when $\dim(M_{\max})=4$ follows from Theorem \ref{fourtheo}. Indeed, to apply the theorem invert the $S^1$-action, which changes the sign of the Hamiltonian and swaps $M_{\min}$ with $M_{\max}$.

3) $\dim(M_{\min})=4$. In this case we consider two possibilities. First is when $\dim(M_{\max})\ge 2$, this is treated in Theorem \ref{fourtheo}. Second case, when $\dim(M_{\max})=0$, is the content of Theorem \ref{MmindlePezzo}. 
\end{proof}

%
%
%
%
%
%
%
%
%

\subsection{Proof of Theorem \ref{MmindlePezzo}} \label{theorem04sec}

Recall briefly results of Section \ref{restictions04} which give restrictions on $M^{S^1}$ in the case when $\dim(M_{\min})=4$ and $\dim(M_{\max})=0$. We established in Proposition \ref{listofweights} that the weights at $M_{\max}$ can be either $\{-1,-1,-1\}$ or $\{-1,-1,-2\}$. 

Additionally to this, we explained that isolated non-extremal fixed points in $M$ can be of types $A$, $B$ and $C$, that have the following weights respectively.
$$\{1,-1,-2\}:\:{\rm type}\; A,\;\; \{2,-1,-1\}:\:{\rm type}\; B,\;\; \{1,-1,-1\}:\:{\rm type}\; C.$$
The number of fixed points of types $A$, $B$, and $C$ are denoted $n_A$, $n_B$, and $n_C$ respectively.

The proof of Theorem \ref{MmindlePezzo} uses the following lemma. 

\begin{lemma}\label{b2fixpoints} Let $M$ be as in Theorem  \ref{MmindlePezzo}. Then $b_2(M_{\min})$ is equal to the number of isolated fixed points in $M$.
\end{lemma}
\begin{proof} Let us see how the topology of the reduced space $M_c$ changes as $c$ decreases from $H_{\max}-\varepsilon$ to $H_{\min}$. For $c=H_{\max}-\varepsilon$, and $\varepsilon$ small enough, the orbifold $M_c$ is a weighted projective plane and so $b_2(M_c)=1$. According to Proposition \ref{listofweights} all non-extremal isolated fixed points on $M$ have index $4$. For this reason, when $c$ passes a non-extremal critical level, the space $M_c$ undergoes a weighted blow-up at all isolated fixed points at this level. At the same time, fixed surfaces in $M$ don't affect the topology of $M_c$. This proves the lemma. \end{proof}

\begin{proof}[Proof of Theorem \ref{MmindlePezzo}.] We will first deal with the topology of $M_{\min}$ and then will deal with its diffeomorphism type.

{\it 1) Topology of $M_{\min}$.} According to Lemma \ref{b2fixpoints}, in order to prove that $b_2(M_{\min})\le 9$ we need to show that $M$ has at most $8$ non-extremal isolated fixed points. In notations of Proposition \ref{listofweights} we need to prove $n_A+n_B+n_C\le 8$. To prove this bound we  will evaluate the volume of $M_0$ using Corollary \ref{isoDuisCor}.  Consider now two cases.

Suppose first $M_{\max}$ has weights $\{-1,-1,-1\}$. By Lemma \ref{nAnB} 3)  fixed points in the range $H>0$ are isolated. By Proposition \ref{listofweights} non-extremal fixed points in the range $H>0$ are of type $A$ or $C$. Hence, we can use Corollary \ref{isoDuisCor} (\ref{isoDuis}) to get
$${\rm Vol}(M_0)=\int_{M_0}\omega_0^2=9-2n_A-n_C.$$
Since the volume ${\rm Vol}(M_0)$ is positive and $n_B=n_A$ by Lemma \ref{nAnB} 1), the desired bound  $n_A+n_B+n_C\le 8$ follows from this equality.

Suppose next that $M_{\max}$ has weights $\{-1,-1,-2\}$.  Applying Equation (\ref{isoDuis}) as above we get
$$\int_{M_0}\omega_0^2=8-2n_A-n_C.$$
By Lemma \ref{nAnB} 2)  $n_B=n_A+1$ and so the desired bound  
$n_A+n_B+n_C\le 8$ is proven.

{\it 2) Diffeomorphism type.} Since $b_2(M_{\min})\le 9$, in order to prove that $M_{\min}$ is diffeomorphic to a del Pezzo surface, it is sufficient to show that it is a rational symplectic $4$-manifold. Let us explain how to deduce this from Theorem \ref{ruledcriterion}. 

According to Theorem \ref{ruledcriterion} it is enough to produce in $M_{\min}$ a smoothly embedded sphere with positive self-intersection. 

Assume first that $M_{\max}$ is a point with weights $\{-1,-1,-2\}$. In this case the reduced space $M_{4-\varepsilon}$ is symplectomorphic to a singular quadratic surface (a cone over a conic). Hence $M_{4-\varepsilon}$ contains a smooth sphere $S$ with self-intersection $2$. 

Note now that the partially defined map $gr^{4-\varepsilon}_{-1+\varepsilon}: M_{4-\varepsilon}\dashrightarrow M_{-1+\varepsilon}$ can be extended to a diffeomorphism on a complement to a finite number (at most $7$) of points in $M_{4-\varepsilon}$. This follows from the fact that  non-extremal isolated fixed points in $M$ have index $4$, and this can be achieved using   Theorem \ref{bimeromorphic} 2) as in the proof of Theorem \ref{fibreinregular}. Let us now perturb the sphere $S$ in $M_{4-\varepsilon}$ to disjoin it from the finite collection of indeterminacy  points of $gr^{4-\varepsilon}_{-1+\varepsilon}$  and take in $M_{-1+\varepsilon}$ the smooth sphere $S'=gr^{4-\varepsilon}_{-1+\varepsilon}(S)$. Clearly, $(S')^2=S^2=2$.

The case when $M_{\max}$ is a point with weights $\{-1, -1, -1\}$ is even easier. Instead of a conic we can start with a complex line in $M_{3-\varepsilon}\cong \mathbb CP^2$. This will give us a smooth sphere with self-intersection $1$ in $M_{-1+\varepsilon}$ in the same way as in the first case. \end{proof}

\subsection{Proof of Corollary \ref{chern}}
Now we deduce Corollary \ref{chern} from Theorem \ref{maintheorem}, i.e. we show that $c_{1}c_{2}(M) =24$.

\begin{proof}[Proof of Corollary \ref{chern}.]

Let $\chi_{y}(M) \in \Z[y]$ denote the  Hirzebruch genus of $M$. 
We have the following localisation formula \cite{F}
\begin{displaymath}
\chi_{y}(M) = \sum_{F \subset M^{S^{1}}} (-y)^{d_{F}} \chi_{y}(F),
\end{displaymath} 
where $d_{F}$ is the number of negative weights at the fixed submanifold $F$ (hence $\frac{1}{2}$ times the Morse-Bott index of $F$).
 
 The Todd genus of $M$ may be expressed in terms of Chern numbers $$td(TM)[M] = \frac{1}{24}c_{1}c_{2}(M).$$ The Todd genus of $M$ is also equal to the constant term of $\chi_{y}(M)$ \cite{F}. The only fixed submanifold that that can contribute to the constant term of $\chi_{y}(M)$ must have $d_{F} = 0$. In other words the Todd genus of $M$ is equal to the Todd genus of $M_{\min}$. Hence it is sufficient to show that the Todd genus of $M_{\min}$ is $1$. This is immediate since $M_{\min}$ is diffeomorphic to a del Pezzo surface, a $2$-sphere or a point by Theorem \ref{maintheorem}. \end{proof}

\appendix

\section{Appendix. Smoothing K\"ahler orbifolds}\label{smoothingSection}

The main goal of this appendix is to show how to smooth locally the metric on $2$-dimensional K\"ahler orbifolds with a semi-toric structure at the orbifold locus (see Definition \ref{semitoricdef}). The main result is Theorem \ref{smoothingdisivors}, it shows how to smooth an orbi-metric with singularities along a collection of divisors to a metric which has only isolated quotient singularities. The smoothing preserves the cohomology class of the metric and does not change it outside of a small neighbourhood of the orbifold locus. The main tool for proving   Theorem \ref{smoothingdisivors} is the technique for gluing K\"ahler metrics based on \emph{ regularised maximum}, which we recall first.

\subsection{Regularised maximum}

Using regularised maximum  one can in certain circumstances  glue together K\"ahler metrics defined on open sets, covering a complex manifold. After recalling the definition, we state  Lemma \ref{gluing} which follows from  \cite[Lemma 5.18]{DE}.

Let $\theta\in \mathbb C^{\infty}(\mathbb R,\mathbb R)$  be a non-negative function with support in $[-1,1]$ such that $\int_{\mathbb R}\theta(h)dh=1$ and  $\int_{\mathbb R}h\theta(h)dh=0$.

\begin{definition} For $\eta \in \mathbb{R}_{>0}$, the \emph{regularised maximum} $M_{\eta}$ of functions $t_1,\ldots, t_p$ (given on any space) is defined as follows
$$M_{\eta}(t_1,\dots, t_p)=\int_{\mathbb R^p}\max\{t_1+h_1,\dots,t_p+h_p\}\Pi_{1\le j\le p}\theta(h_j/\eta)dh_1\dots dh_p$$
\end{definition}

\begin{lemma}\label{gluing} Let $U$ be an open complex manifold covered by two open submanifolds, $U=U_1\cup U_2$ and let $\partial U_1, \partial U_2$ be the boundaries of the closures of $U_1$ and $U_2$.  Let $\omega'$ be a closed real $(1,1)$-form on $U$. Suppose that on each $U_i$ there is a K\"ahler form $\omega_i$ and a continuous function $\varphi_i$ extendible to $\partial U_i$, so that the following conditions hold.

1) $\omega'+\sqrt{-1}\partial\bar\partial \varphi_i=\omega_i$ on $U_i$.

2) There is a constant $c>0$ such that $\varphi_1-c>\varphi_2$ on $U_1\cap  \partial U_2$ and $\varphi_2-c>\varphi_1$ on $U_2\cap \partial U_1$.

Then there is a K\"ahler metric $\omega$ on $U$, that coincides with $\omega_1$ on $U_1\setminus U_2$ and with $\omega_2$ on $U_2\setminus U_1$. Such $\omega$ can be expressed as 
\begin{equation}
\omega=\omega'+\sqrt{-1}\partial\bar\partial M_{\eta}(\varphi_1, \varphi_2),
\end{equation} 
where $M_{\eta}$ is the regularised maximum with $\eta$ small enough. 
 
\end{lemma} 
\begin{proof}
This lemma follows directly from \cite[Lemma 5.18]{DE}.
\end{proof}

\subsection{Modifying  K\"ahler metrics in small neighbourhoods}

In this section we show how a K\"ahler metric can be modified in a small neighbourhood of a point to match a different metric. The following lemma is classical but we provide a proof for a lack of reference.

\begin{lemma}\label{implant} Let $U$ be a complex manifold, $x\in U$ be a point and $U_x\subset U$ be a neighbourhood of $x$. Suppose  that $g$ is a K\"ahler metric on $U$ and $g_x$ is a K\"ahler metric on $U_x$. Then there exists a K\"ahler metric $g'$ on $U$ that is equal to $g_x$ on a smaller neighbourhood $U_x'\subset U_x$ of $x$ and equal to $g$ on $U\setminus U_x$.
\end{lemma} 

\begin{proof}  Let $B \subseteq U_{x}$ be a standard coordinate ball, such that $x = 0$. We may find strictly plurisubharmonic functions $u$ and $v$ such that $\omega_{g} = i \partial \bar{\partial} u$ and   $\omega_{g_{x}} = i \partial \bar{\partial} v $. Next, we fix a cut-off function $\chi$ that is equal to $1$ near $x$, compactly supported on $B$ and vanishes close to $\partial B$. We may choose $\delta>0$ small enough so that the function \begin{displaymath}
\tilde{u}(z) = u(z)+\delta \chi(z) \log(|z|),
\end{displaymath} 
is strictly plurisubharmonic on $B \setminus \{x\}$. Next, we choose $C$ large enough so that $v - C < u$ in a neighbourhood of $\partial B$. The function $\varphi= M_{\eta}(\tilde{u},v-C)$ is smooth and strictly plurisubharmonic on $B$. Since on a neighbourhood of $\partial{B}$, $ u = \tilde{u}$ and $v - C < u$, the K\"ahler form associated to $g' = i \partial \bar{\partial} \varphi$, $\omega_{g'} $ is equal to $\omega_{g}$ there. Hence, we may define a K\"ahler metric on $U$, that is equal to $\omega_{g'}$ on $B$ and equal to $\omega_{g}$ on $U \setminus B$. One may check that this metric satisfies the required properties.
\end{proof}

\begin{remark}\label{sympiso} Let $g$ and $g'$ be K\"ahler metrics from Lemma \ref{implant} and consider the family of K\"ahler metrics $g_t=(1-t)g+tg'$, $t\in [0,1]$ on $U$. Let $\omega_t$ be the corresponding family of symplectic forms. Then symplectic manifolds $(U,\omega_t)$ are all symplectomorphic and moreover the symplectomorphism can be chosen to be equal to the identity on  $U\setminus U_x$.  
 
The symplectomorphism $(U,\omega)\to (U,\omega')$ we will be denoted $\Phi$.
\end{remark}

The following corollary is a symplectic variation of Lemma \ref{implant}.

\begin{corollary}\label{symplant} Let $(U,\omega)$ be a symplectic manifold, $x\in U$ be a point and $U_x\subset U$ be a neighbourhood of $x$. Suppose  that $g$ is a K\"ahler metric on $U$ and $g_x$ is a K\"ahler metric on $U_x$, both compatible with $\omega$. Then there exists a neighbourhood $U_x'\subset U_x$ of $x$ and a K\"ahler metric $g'$ on $U$ compatible with $\omega$,  such that $g=g_x$ on $U_x'$ and $g'=g$ on $U\setminus U_x$.
\end{corollary} 

\begin{proof}  Using Lemma \ref{implant}, we can deform $g$ to a K\"ahler metric $g''$ on $U$ satisfying two properties. 1) $g''$ coincides with $g$ outside of a small neighbourhood of $x$. 2) A small neighbourhood $U_{\varepsilon}(x)\subset U_{x}$ admits an isometric embedding $I: (U_{\varepsilon}(x),g'')\to (U_x,g_x)$. Using further the symplectomorphism  $\Phi:(U,\omega)\to (U,\omega'') $ mentioned in Remark \ref{sympiso} we obtain a K\"ahler metric $\Phi^*(g'')$ on $U$, compatible with $\omega$.

The metric $g'$ on $U$ will be obtained from $\Phi^*(g'')$ by a symplectic automorphism of $(U,\omega)$. By definition of $g''$ a small neighbourhood of $x$ endowed with the metric $\Phi^*(g'')$ admits an isometric embedding $I'$ into $(U_x,g_x)$  sending $x$ to $x$. Shrinking further this small neighbourhood and using Lemma \ref{sympextension}  we can extend $I'$ to a symplectic automorphism $\overline I'$ of $(U,\omega)$ that is equal to the identity on $U\setminus U_x$. The metric $g'$ we are looking for is then given by $g'=\overline I'(\Phi^*(g''))$.
\end{proof}

\begin{remark}\label{orbisymplant} Note that Corollary \ref{symplant} holds as well for orbifolds, where instead for Lemma \ref{sympextension} one has to use its orbifold version, see Remark \ref{orbisympextension}. 
\end{remark}

\subsection{$S^1$-invariant K\"ahler metrics on orbi-bundles}

The main result of this section is Corollary \ref{mainsmoothing}. It constructs a smoothing  of  an $S^1$-invariant, orbifold K\"ahler metric which is defined on the total space of a line bundle, with the orbifold locus along the zero section of the bundle.

Let $M$ be a K\"ahler manifold and $(L,h)$ be a Hermitian line bundle on it. Denote by $\cal L$ the total space of $L$, let $M_0$ be the zero section in $\cal L$ and let $\pi: {\cal L}\to M_0$ be the projection. A point of $\cal L$ will be denoted as $(x,z)$, where $x\in M_0$, and $z\in L_x$. By $|z|_h^2$ we denote the function equal to the square of the norm given by $h$. By $U_c \subset \cal L$ we denote the neighbourhood of $M_0$ consisting of points $(x,z)$ with $|z|_h\le c$.

For each integer $n>1$ one can turn $\cal L$ into   a complex orbifold ${\cal L}_n$ by declaring that all points of the zero section $M_0$ have stabilizer $\mathbb Z_n$. The next lemma and its corollary analyse K\"ahler and orbi-K\"ahler metrics on the neighbourhood $U_1$ of $M_0$.

\begin{lemma}\label{smoothlocal} Let $\omega$ be any $S^1$-invariant K\"ahler metric on the neighbourhood $U_1\subset \cal L$ of $M_0$, and denote the restriction of $\omega$ to $M_0$ by $\omega_0$. Then the metric $\omega$ can be presented as 
$$\pi^*(\omega_0)+i\partial \bar\partial f(x,z)|z|_h^2$$ 
where  $f$ is an $S^1$-invariant function satisfying the inequality $0<c_1<f(x,z)<c_2$ on $U_1$ for some positive constants $c_1$ and $c_2$.
\end{lemma}

\begin{proof}

Note that for every point $x\in M_0$ there is a unique $S^1$-invariant function $u_x$ defined on the unit disk $\mathbb D_1\subset L_x$ such that  $\omega|_{\mathbb D_1}=i\partial\bar\partial u_x$  and $u_x(0)=0$. This clearly defines to us a smooth function $u$ on whole $U_1$. Let us first show that $\omega=\pi^*(\omega_0)+i\partial \bar\partial u$.

Indeed, the form $\omega'=\omega-\pi^*(\omega_0)-i\partial \bar\partial u$ is $S^1$-invariant, it vanishes on $M_0$, and vanishes on all fibres of $L$ in $\cal L$. It is clear, that to prove that such $\omega'$ is identically zero it would be enough to prove this when $M_0$ is a complex ball $B$ and $U_1$ is $B\times \mathbb D_1$. Let us represent then $\omega'$ as $i\partial\bar\partial v$, where $v$  vanishes on $B\times 0$ and $S^1$-invariant. Since $\omega'$ vanishes on each $\mathbb D_1$-fibre, $v$ is constant on each such fibre and hence it is identically zero.

Finally, to see that $0<c_1< u(x,z)|z|_h^{-2}<c_2$ on $U_1$ for some $c_1,c_2>0$, note that such an inequality holds uniformly on each $\mathbb D_1$-fibre, since $i\partial\bar\partial u(x,z)$ is K\"ahler on each such $\mathbb D_1$-fibre, $u(x,0)=0$ and $u(x,z)$ is $S^1$-invariant.
\end{proof}

\begin{corollary}\label{orbilocal} Fix an integer $n>1$ and introduce on $\cal L$ the structure  
of orbifold ${\cal L}_n$. Let $\omega$ be any $S^1$-invariant  K\"ahler orbi-metric on the neighbourhood $U_1\subset {\cal L}_n$ of $M_0$. Then the metric $\omega$ can be presented as 
$$\pi^*(\omega_0)+i\partial \bar\partial f(x,z)|z|_h^{\frac{2}{n}}$$ 
where  $f$ is an $S^1$-invariant function that satisfies the inequality $0<c_1<f(x,z)<c_2$ on $U_1$ for some positive constants $c_1$ and $c_2$.
\end{corollary}

\begin{proof}The statement is local, and so it is enough to prove it when there exists a $n$-th  root $L'$ of $L$, $L\cong L'^{\otimes n}$. Consider the holomorphic map $\mu_n:{\cal L'}\to {\cal L}$ that is a composition of the $n$-th tensor power map $L'\to L'^{\otimes n}$ and the isomorphism. Let $h'$ be the unique Hermitian metric on $L'$ such that norm $1$ vectors in $L'$ are sent to norm $1$ vectors in $L$. Then $\mu_n^*(\omega)$ is a smooth K\"ahler form on $U_1'$ and by Lemma \ref{smoothlocal} we have 

$$\mu_n^*(\omega)=\pi^*(\omega_0)+i\partial \bar\partial f'(x,z)|z|_h'^2.$$ 

Since $f'(x,z)$ is $S^1$-invariant and satisfies   $0<c_1<f'(x,z)<c_2$ we see that the function $f(x,z)$ induced from $f'(x,z)$ on the quotient $U_1=U_1'/\mathbb Z_n$ satisfies $0<c_1<f(x,z)<c_2$. At the same time $\mu$ sends the function $|z|_{h'}^2$  to $|z|_h^{\frac{2}{n}}$.
\end{proof}

\begin{corollary}\label{mainsmoothing} Fix an integer $n>1$, introduce on $\cal L$ the structure of orbifold ${\cal L}_n$ and let $a\in (0,1)$. Let $\omega_n$ be any $S^1$-invariant  K\"ahler orbi-metric on the neighbourhood $U_1\subset {\cal L}_n$ of $M_0$. The metric $\omega_n$ can be smoothed in $U_1$ to an $S^1$-invariant K\"ahler metric $\omega$ so that the following holds
\begin{enumerate}
\item The smoothed metric $\omega$ can be presented as $\omega_n+i\partial \bar\partial \varphi$, where $\varphi$ is a continuous function with support in $U_a$. In particular $\omega$ coincides with $\omega_n$ in $U_1\setminus U_{a}$.

\item For any open subset $U\subset M_0$ the smoothed metric on $\pi^{-1}(U)$ only depends on restriction of $\omega_n$ on $\pi^{-1}(U)$.

\item Suppose that for some open $U\subset M_0$ the metric on $\pi^{-1}(U)$ is that of a direct product of $\omega_0$ on $U$ with an $S^1$-invariant K\"ahler orbifold metric on a unit disk. Then the smoothing preserves the direct product structure. 
\end{enumerate}
\end{corollary}

\begin{proof}

1) According to Corollary \ref{orbilocal} we have $\omega_n=\pi^*(\omega_0)+i\partial \bar\partial f(x,z)|z|_h^{\frac{2}{n}}$, where $0<c_1<f(x,z)<c_2$ on  $U_1$. Choose $b\in (0,a)$ such that the form  $\omega_1=\pi^*(\omega_0)+i\partial \bar\partial |z|_h^2$ is K\"ahler in $U_b$. It is not hard to see that there exist constants $d,e_1,e_2$ satisfying $0<d$ and $0<e_1<e_2<b$,  and such that 

\begin{equation}\label{ineqtwo}
e_1^2+d>c_2e_1^{\frac{2}{n}},\;\; e_2^2+d<c_1e_2^{\frac{2}{n}}.
\end{equation}

Now we can apply Lemma \ref{gluing} to $U_1= U_{e_2}\cup (U_1\setminus U_{e_1})$ and K\"ahler forms represented in the following form

$$\omega_1=\pi^*(\omega_0)+i\partial \bar\partial (|z|_h^2+d),\;\;\omega_n=\pi^*(\omega_0)+i\partial \bar\partial f(x,z)|z|_h^{\frac{2}{n}}. $$

From Inequalities (\ref{ineqtwo}) and the bounds on $f(x,z)$ it follows that the conditions of Lemma \ref{gluing} on potentials $|z|_h^2+d$ and $f(x,z)|z|_h^{\frac{2}{n}}$ are satisfied. So we can set 
$$\varphi=M_{\eta}(|z|_h^2+d, f(x,z)|z|_h^{\frac{2}{n}}),$$
to get a smooth  K\"ahler metric $\omega=\pi^*(\omega_0)+i\partial \bar\partial \varphi$ on $U_1$. The metric $\omega$ coincides with $\omega_1$ in $U_{e_1}$ and with $\omega_n$ in $U_1\setminus U_{e_2}$. Moreover, the function $f(x,z)|z|_h^{\frac{2}{n}}-\varphi$ is clearly continuous, and since  $e_2<a$, it is supported in $U_a$.  Hence condition 1) is established.

The validity of condition 2) follows from the proof of Corollary \ref{orbilocal}, indeed the values of the function $f(x,z)$ on $\pi^{-1}(U)$ only depend on the behaviour of the metric $\omega_n$ in $\pi^{-1}(U)$. Condition 3) holds, since in this case, the function $f(x,z)$ in $U$ only depends on $|z|_h$ and not on $x$.
\end{proof}

\subsection{Smoothing of K\"ahler orbi-metrics on surfaces}
In this section we will use Corollary \ref{mainsmoothing} to deduce a result on smoothing of K\"ahler metrics on complex orbifolds of dimension $2$. We will deal with the cases when the stabilizer of each point is cyclic. The underlying complex analytic surface of such an orbifold is a complex surface with quotient singularities of the type $\mathbb C^2/\mathbb Z_n$. It will be useful to us to write down the action of the generator of $\mathbb Z_n$ on $\mathbb C^2$ in the following form, 
\begin{equation}\label{quotientpresent}
 (x,z)\to (\mu_1^p\cdot x, \mu_2^q\cdot z),\; \mu_1=e^{2\pi i \cdot k_1/n},\;  \mu_2=e^{2\pi i\cdot k_2/n},
\end{equation} 
where $k_1$ and $k_2$ are coprime with $n$, while and $p$ and $q$ are coprime divisors of $n$.

\begin{theorem}\label{smoothingdisivors} Let $(S, g)$ be a K\"ahler orbifold of dimension $2$ with cyclic stabilizers and let $D_1,\ldots, D_k$ be the $1$-dimensional irreducible components of the orbifold locus of $S$. Suppose that $g$ is semi-toric at $D_1,\ldots, D_k$. Then for an arbitrary small neighbourhood $U$ of the orbifold locus of $S$ there is a smoothing $\omega'$ of $\omega$ such that

1) $\omega'=\omega$ on $S\setminus U$. 

2) $\omega'$ defines on $U$ a structure of a K\"ahler orbifold with isolated orbi-points.

3) $[\omega']=[\omega]\in H^2(S)$.

\end{theorem}
\begin{proof} To smooth the metric we will smooth it  consecutively along all divisors, starting from $D_1$. Suppose first that $D_1$ is disjoint from all other divisors and it does not contain singular points of $S$. In this case, all points of $D_1$ have the same stabilizer $\mathbb Z_p$. The complex  $S^1$-action  in a neighbourhood $U$ permits us to biholomorphically identify $U$ with a unit disk sub-bundle $U_1$ of a complex line bundle $L$ over $D_1$. This puts us in the setting of Corollary \ref{mainsmoothing}. Now the existence of a smoothing in arbitrary small neighbourhood of $D_1$ follows from this corollary.

Suppose now that  a generic point of $D_1$ has stabilizer $\mathbb Z_p$ 
while at a finite subset $\bar x\in D_1$ the stabilizers strictly contain $\mathbb{Z}_{p}$. Let us first use Lemma  \ref{implant} and deform slightly the metric in an $\varepsilon$-neighbourhood $U_{\varepsilon}(\bar x)$ of $\bar x$ to make it flat there in the orbifold sense. Then we can smooth the metric along the regular part of $D_1$ that avoids $U_{\varepsilon/2}(\bar x)$, applying Corollary \ref{mainsmoothing} as above. So we just need to explain how to extend this smoothing to $U_{\varepsilon}(\bar x)$. 

Let $x\in \bar x$ be a point with stabilizer $\mathbb Z_n$ where $n>p$. The local orbi-action is given by Equation (\ref{quotientpresent}). Consider the presentation of $U_{\varepsilon}(x)$ as $B_{\varepsilon}/\mathbb Z_n$ where $B_{\varepsilon}\subset \mathbb C^2$ is a flat complex ball. It is not hard to see that the subgroup $\mathbb Z_p\subset \mathbb Z_n$ is acting on $B_{\varepsilon}$ fixing the preimage of $D_1\cap U_{\varepsilon}(x)$ in $B_{\varepsilon}$ (which is a flat disk there). Hence $B_{\varepsilon}/\mathbb Z_p$ is a smooth complex disk, and it is an orbi-cover of $U_{\varepsilon}(x)$ with orbi-group $\mathbb Z_{n/p}$. The smoothing of the metric on the complement to $U_{\varepsilon/2}(x)$ in  $U_{\varepsilon}(x)$ can be lifted to its preimage in $B_{\varepsilon}/\mathbb Z_p$ and then extended to the whole preimage by  Corollary \ref{mainsmoothing} 3). In this way we are able to extend the smoothing to the point $x$, making it an orbi-point with stabilizer $\mathbb Z_{n/p}$.
 
Repeating now the above smoothing construction for all the divisors we obtain a K\"ahler orbi-metric on $S$, that satisfies properties 1) and 2). To see that property 3) holds as well, note that whenever we use  Corollary \ref{mainsmoothing} 1) to smooth $g$ along an orbi-curve, we add to $\omega_g$ an exact form $i\partial \bar \partial \varphi$. Hence we don't change the cohomology class of $\omega_g$. 
\end{proof}

\subsection{Orbi-K\"ahler metric along the orbifold locus}
The goal of this section is to prove the following proposition.

\begin{proposition}\label{kahlerdivisor} Let $(M^4,\omega)$ be a symplectic orbifold with cyclic stabilizers and let $D$ be a $2$-dimensional irreducible component of the orbifold locus. Then there is a neighbourhood $U$ of $D$ with a K\"ahler metric $g$ on $U$ compatible with $\omega$ and a Hamiltonian isometric $S^1$-action on $U$, fixing $D$. 
\end{proposition}

The proof of this proposition is composed of four steps. In Step 1 we choose a flat K\"ahler  metric on $M^4$ close maximal orbi-points on $D$. In Step 2 we show that the normal orbi-bundle $N_D$ to $D$ has a natural holomorphic  structure. In Step 3 we choose a K\"ahler structure close to the zero section of $N_D$. In Step 4 we apply Darboux-Weinstein Theorem \ref{relative_local_neighbourhood}. 

\begin{proof} {\it Step 1.} Suppose that the stabilizer of a generic point of $D$ is $\mathbb{Z}_{p}$. Let $x_1,\ldots, x_m$ be all the points on $D$ with stabilizers $\mathbb Z_{n_i}$ with  $n_i>p$. For each $x_i$ we  take a Darboux neighbourhood $U_i$ of $x_i$ that is symplectomorphic to a neighbourhood of $0$ in the quotient $\mathbb C^2/\mathbb Z_{n_i}$. We assume that $\mathbb Z_{n_i}$ acts on $\mathbb C^2$  by isometries and, following presentation (\ref{quotientpresent}), is given by 
$$(x,z)\to (\mu_{i}^{p}\cdot x, \nu_{i}^{q_i}\cdot z).$$
Here $\mu_i$ and $\nu_i$ are primitive $n_i$-th roots of unity, and $q_i$ and $p$ are coprime divisors of $n_i$. Thus we have a flat K\"ahler metric compatible with $\omega$ on the union of $U_i$. This metric will be extended to a  neighbourhood of $D$.

{\it Step 2.} We will now introduce a holomorphic structure on the orbifold line bundle $N_D$ over $D$, isomorphic to the normal orbi-bundle to $D$. First, we can extend the complex structure on $\cup_iU_i$ to an almost complex structure $J$ on a neighbourhood of $D$, compatible with $\omega$. This gives us a complex structure on $D$ and makes $N_D$ a Hermitian orbi-bundle over $D$. 

Notice, that for $\varepsilon$ small enough the total space of the bundle $N_D$ over an $\varepsilon$-neighbourhood $U(x_i,\varepsilon)\subset D$ of $x_i$  can be identified with the quotient  of the cylinder $\{|z|<\varepsilon\}\subset \mathbb C^2$ by the above action of $\mathbb Z_{n_i}$. Hence we have a holomorphic structure on $N_D$, over the union of neighbourhoods  $U(x_i,\varepsilon)$. On each  punctured neighbourhood $U(x_i,\varepsilon)\setminus x_i$ we have a flat connection on $N_D$, induced from the flat metric on $\mathbb C^2/\mathbb Z_{n_i}$. In order to get a holomorphic structure on the whole $N_D$ it suffices to extend this flat connection from the union of punctured neighbourhoods $U_i$ to some Hermitian connection on $N_D$ over $D\setminus \cup_iU(x_i,\varepsilon)$. Such a connection will induce a holomorphic structure on the total space of $N_D$. 

{\it Step 3.} Let us now construct a positive closed $(1,1)$-form  on a neighbourhood of the zero section in the total space of $N_D$. We already have such a flat form on $N_D$ over the union of $U_i(x_i,\varepsilon)$ and we have such a form $\omega_0$ on $D$, $\omega_0=\omega|_D$. Hence we can define a $(1,1)$-form close to $D$ by the formula 
$$\omega_n=\pi^*(\omega_0)+i\partial \bar\partial|z|_h^{\frac{2}{p}},$$
as in Corollary \ref{orbilocal}. This form is positive in some neighbourhood of $D$ and it is $S^1$-invariant by construction.

{\it Step 4.} Finally, by Theorem \ref{relative_local_neighbourhood} 1) there is a symplectomorphism $\varphi$ from a neighbourhood of the zero section of $N_D$ to a neighbourhood of $D$ in $M^4$. This symplectomorphism induces the desired K\"ahler structure close to $D$.
\end{proof}

\subsection{Resolving isolated orbi-points and submanifolds}
The following two statements are standard so we omit their proof.

\begin{lemma}\label{kahleronisolated} Consider $\mathbb C^n$ with a flat K\"ahler metric $\omega$ and let $\Gamma\subset U(n)$ be a finite group acting on $\mathbb C^n$ by isometries so that the action is free on $\mathbb C^n\setminus 0$. Let $\pi: X\to \mathbb C^n/\Gamma$ be a resolution of singularities of the quotient. Then there is a K\"ahler metric on $X$ that coincides with $\pi^*(\omega)$ outside of a neighbourhood of the exceptional divisor $E$. 
\end{lemma}

 


\begin{lemma}\label{smoothblowup} Let $U$ be a possibly open manifold with a K\"ahler metric $g$ and let $X\subset U$ be a smooth compact K\"ahler submanifold of complex codimension $\ge 2$. Consider the blow up of $U$ in $X$, $\pi:U'\to U$. Then there exists a K\"ahler metric on  $U'$ that coincides with $\pi^*g$ outside of a small neighbourhood of the exceptional divisor $E\subset U'$. 

Moreover, in the case when a compact group $G$ is acting by K\"ahler isometries of $(U,X)$ the K\"ahler blow up can be preformed $G$-equivariantly.
\end{lemma}





\subsection{Smooth foliations in $\mathbb C^2$ with holomorphic leaves. }

The goal of this subsection is to give a sketch proof of the following  result.

\begin{lemma}\label{foliationlemma}Consider $\mathbb C^2$ with an almost complex structure $J$. Let $B_1\subset \mathbb C^2$ be the unit ball $B_1:=\{(w_{1},w_{2}) \in \mathbb{C}^{2} : |w_1|^2+|w_2|^2<1\}$, and let $\cal H$ be a  smooth foliation on $B_1$ with $J$-holomorphic leaves. Then for any smooth $J$-holomorphic curve $C\subset B_1$ the points of tangency of  $C$ and $\cal H$ form a discrete subset of $C$. 
\end{lemma}
\begin{proof}[Sketch proof.] Let $p$ be a point of tangency of $C$ and $\cal H$.
As in the proof of Lemma \cite[Lemma 2.4.3]{MS}  one can choose  new coordinates $(w_1,w_2)$ close to $p$ so that $p=(0,0)$, the leaf of $\cal H$ passing through $(0,0)$ is given by $w_2=0$, and $J$ is equal to multiplication by $i$ along the axis $w_2=0$.

Let $v:z\to (v_1(z),v_2(z))$ be a $J$-holomorphic parametrization of $C$. Then, as in the proof of Lemma \cite[Lemma 2.4.3]{MS}, we have the  presentation 
$$(v_1(z),v_2(z))=(p(z)+O(|z^{n+1}|), az^n+O(|z^{n+1}|)),$$
where $p(z)$ is a polynomial of degree at most $n$ with  $p'(0)\ne 0$ and $a\ne 0$.

Making a smooth parametrization in $z$ and  a change of coordinates in $w_1,w_2$, we can assume that $v_1(z)=z$, $v_2(z)=z^n+O(|z^{n+1}|)$, and moreover the foliation $\cal H$ is horizontal close to $(0,0)$. It is now clear that that $(0,0)$ is an isolated tangency point of $\cal H$ and $C$. Indeed, the map $z\to z^n+O(|z^{n+1}|)$ has non-degenerate differential if $z\ne 0$ and $|z|$ is small.
\end{proof}

\end{document}